\documentclass[a4paper,reqno,11pt]{amsart}

\usepackage{amsmath, amsfonts, amssymb, amsthm, amscd}
\usepackage{graphicx}
\usepackage{psfrag}
\usepackage{perpage}
\usepackage{url}
\usepackage{color}
\usepackage{mathrsfs}
\usepackage{mathabx}
\usepackage{scalerel}
\usepackage{soul}

\usepackage{tikz}\usepackage{caption,subcaption}
\usetikzlibrary{arrows,decorations.pathmorphing,backgrounds,positioning,fit,petri}
\usepackage{subdepth}

\usepackage{dsfont} 

\usepackage[utf8]{inputenc}
\usepackage[T1]{fontenc}
\usepackage{microtype}

\usepackage[a4paper,scale={0.72,0.74},marginratio={1:1},footskip=7mm,headsep=10mm]{geometry}

\usepackage{hyperref}

\makeatletter
\def\@secnumfont{\bfseries\scshape}

\def\section{\@startsection{section}{1}
  \z@{.9\linespacing\@plus\linespacing}{.5\linespacing}%
  {\normalfont\large\bfseries\scshape\centering}}

\def\subsection{\@startsection{subsection}{2}%
  \z@{.5\linespacing\@plus.7\linespacing}{-.5em}%
  {\normalfont\bfseries\scshape}}

\def\subsubsection{\@startsection{subsubsection}{3}%
  \z@{.5\linespacing\@plus.7\linespacing}{-.5em}%
  {\normalfont\scshape}}

\def\specialsection{\@startsection{section}{1}%
  \z@{\linespacing\@plus\linespacing}{.5\linespacing}%
  {\normalfont\centering\large\bfseries\scshape}}
\makeatother

\makeatletter

\renewenvironment{proof}[1][\proofname]{\par
\pushQED{\qed}%
\normalfont \topsep4\p@\@plus4\p@\relax
\trivlist
\item[\hskip\labelsep
\bfseries
#1\@addpunct{.}]\ignorespaces
}{%
\popQED\endtrivlist\@endpefalse
}
\makeatother

\makeatletter
\newcommand \Dotfill {\leavevmode \leaders \hb@xt@ 6pt{\hss .\hss }\hfill \kern \z@}
\makeatother

\makeatletter
\def\@tocline#1#2#3#4#5#6#7{\relax
  \ifnum #1>\c@tocdepth 
  \else
    \par \addpenalty\@secpenalty\addvspace{#2}%
    \begingroup \hyphenpenalty\@M
    \@ifempty{#4}{%
      \@tempdima\csname r@tocindent\number#1\endcsname\relax
    }{%
      \@tempdima#4\relax
    }%
    \parindent\z@ \leftskip#3\relax \advance\leftskip\@tempdima\relax
    \rightskip\@pnumwidth plus4em \parfillskip-\@pnumwidth
    #5\leavevmode\hskip-\@tempdima
      \ifcase #1
       \or\or \hskip 1.65em \or \hskip 3.3em \else \hskip 4.95em \fi%
      #6\nobreak\relax
    \Dotfill
    \hbox to\@pnumwidth{\@tocpagenum{#7}}\par
    \nobreak
    \endgroup
  \fi}
\makeatother

\makeatletter
\def\l@section{\@tocline{1}{0pt}{1pc}{}{\scshape}}
\renewcommand{\tocsection}[3]{%
\indentlabel{\@ifnotempty{#2}{\ignorespaces#1 #2.\hskip 0.7em}}#3}
\def\l@subsection{\@tocline{2}{0pt}{1pc}{5pc}{}}

\def\l@subsubsection{\@tocline{3}{0pt}{1pc}{7pc}{}}

\makeatother

\numberwithin{equation}{section}

\newtheoremstyle{mytheorem}{.7\linespacing\@plus.3\linespacing}{.7\linespacing\@plus.3\linespacing}%
     {\itshape}
     {}
     {\bfseries}
     {. }
     {0.3ex}
     {\thmname{{\bfseries #1}}\thmnumber{ {\bfseries #2}}\thmnote{ (#3)}}  

\theoremstyle{mytheorem}

\newtheorem{theorem}{Theorem}[section]
\newtheorem{lemma}[theorem]{Lemma}
\newtheorem{proposition}[theorem]{Proposition}
\newtheorem{corollary}[theorem]{Corollary}
\newtheorem{remark}[theorem]{Remark}
\newtheorem{definition}[theorem]{Definition}

\newcommand{\bbE}{{\ensuremath{\mathbb E}} }

\newcommand{\bbP}{{\ensuremath{\mathbb P}} }

\newcommand{\bbT}{{\ensuremath{\mathbb T}} }

\newcommand{\bbV}{{\ensuremath{\mathbb V}} }

\newcommand{\cB}{{\ensuremath{\mathcal B}} }

\newcommand{\cF}{{\ensuremath{\mathcal F}} }

\newcommand{\cM}{{\ensuremath{\mathcal M}} }
\newcommand{\cN}{{\ensuremath{\mathcal N}} }

\newcommand{\cQ}{{\ensuremath{\mathcal Q}} }

\newcommand{\cS}{{\ensuremath{\mathcal S}} }
\newcommand{\cT}{{\ensuremath{\mathcal T}} }
\newcommand{\cU}{{\ensuremath{\mathcal U}} }

\newcommand{\cZ}{{\ensuremath{\mathcal Z}} }

\newcommand{\gb}{\beta}

\DeclareMathSymbol{\leqslant}{\mathalpha}{AMSa}{"36} 
\DeclareMathSymbol{\geqslant}{\mathalpha}{AMSa}{"3E} 
\DeclareMathSymbol{\eset}{\mathalpha}{AMSb}{"3F}     

\newcommand{\sumtwo}[2]{\sum_{\substack{#1 \\ #2}}} 

\newcommand{\be}{\begin{equation}}
\newcommand{\ee}{\end{equation}}

\newcommand{\R}{\mathbb{R}}

\newcommand{\Z}{\mathbb{Z}}
\newcommand{\N}{\mathbb{N}}

\def\bs{\boldsymbol}

\newcommand{\PEfont}{\mathrm}

\newcommand{\p}{\ensuremath{\PEfont P}}

\newcommand{\E}{\ensuremath{\PEfont  E}}
\renewcommand{\P}{\p}

\newcommand\bE{\ensuremath{\bs{\mathrm{E}}}}

\DeclareMathOperator{\bbvar}{\ensuremath{\mathbb{V}ar}}
\DeclareMathOperator{\bbcov}{\ensuremath{\mathbb{C}ov}}

\newcommand{\ind}{\mathds{1}}

\newcommand{\eps}{\varepsilon}
\renewcommand{\epsilon}{\varepsilon}
\renewcommand{\theta}{\vartheta}
\renewcommand{\rho}{\varrho}

\newenvironment{myenumerate}{%
\renewcommand{\theenumi}{\arabic{enumi}}%
\renewcommand{\labelenumi}{{\rm(\theenumi)}}%
\begin{list}{\labelenumi}
	{%
	\setlength{\itemsep}{0.4em}%
	\setlength{\topsep}{0.5em}%
	\setlength\leftmargin{2.45em}%
	\setlength\labelwidth{2.05em}%
	\setlength{\labelsep}{0.4em}%
	\usecounter{enumi}%
	}%
	}%
{\end{list}
}

{\end{list}
}

{\end{list}
}

{\end{myenumerate}}

\newenvironment{myitemize}{%
\begin{list}{$\bullet$}%
 	{%
	\setlength{\itemsep}{0.4em}%
	\setlength{\topsep}{0.5em}%
	\setlength\leftmargin{2.65em}%
	\setlength\labelwidth{2.65em}%
	\setlength{\labelsep}{0.4em}%
	}%
	}%
{\end{list}}

\renewenvironment{itemize}{
\begin{myitemize}}%
{\end{myitemize}}

\MakePerPage[2]{footnote} 

\date{\today}

\newcommand\dd{\mathrm{d}}

\newcommand\sfC{\mathsf C}
\newcommand\sfG{\mathsf G}

\newcommand\sfL{\mathsf L}

\newcommand\sfQ{\mathsf Q}

\newcommand\sfU{\mathsf U}

\newcommand\sfX{\mathsf X}

\newcommand\sfa{\mathsf a}
\newcommand\sfb{\mathsf b}

\newcommand\sfi{\mathsf i}

\newcommand\sfj{\mathsf j}

\newcommand\sfm{\mathsf m}

\newcommand\bx{\boldsymbol{x}}
\newcommand\by{\boldsymbol{y}}
\newcommand\bz{\boldsymbol{z}}
\newcommand\bw{\boldsymbol{w}}

\newcommand\cg{\mathrm{cg}}

\newcommand\notri{\mathrm{no\, triple}}

\newcommand\diff{\mathrm{diff}}

\newcommand\sfc{\mathsf{c}}

\newcommand\bcA{\vec{\boldsymbol{\mathcal{A}}}}

\newcommand\SHF{\mathsf{SHF}}

\title[The Critical $2d$ Stochastic Heat Flow]{The Critical $2d$ Stochastic Heat Flow\\ and Related Models}

\begin{document}

\author[F. Caravenna]{Francesco Caravenna}
\address{Dipartimento di Matematica e Applicazioni\\
 Universit\`a degli Studi di Milano-Bicocca\\
 via Cozzi 55, 20125 Milano, Italy}
\email{francesco.caravenna@unimib.it}

\author[R. Sun]{Rongfeng Sun}
\address{Department of Mathematics\\
National University of Singapore\\
10 Lower Kent Ridge Road, 119076 Singapore
}
\email{matsr@nus.edu.sg}

\author[N. Zygouras]{Nikos Zygouras}
\address{Department of Mathematics\\
University of Warwick\\
Coventry CV4 7AL, UK}
\email{N.Zygouras@warwick.ac.uk}

\begin{abstract}
In these lecture notes, we review recent progress in the study of the stochastic heat equation and its discrete analogue, the directed polymer model, in spatial dimension~2. It was discovered that a phase transition emerges on an intermediate disorder scale, with Edwards-Wilkinson (Gaussian) fluctuations in the sub-critical regime. In the critical window, a unique scaling limit has been identified and named the {\em Critical $2d$ Stochastic Heat Flow}. This gives a meaning to the solution of the stochastic
heat equation in the critical dimension 2, which lies beyond existing solution theories for singular SPDEs. We outline the proof ideas, introduce the key ingredients, and discuss related literature on disordered systems and singular SPDEs. A list of open questions is also provided.
\end{abstract}

\maketitle


\begin{figure}[h]
\centering
\hfill \
\begin{minipage}[b]{.5\linewidth}
\centering
\includegraphics[width=\linewidth]{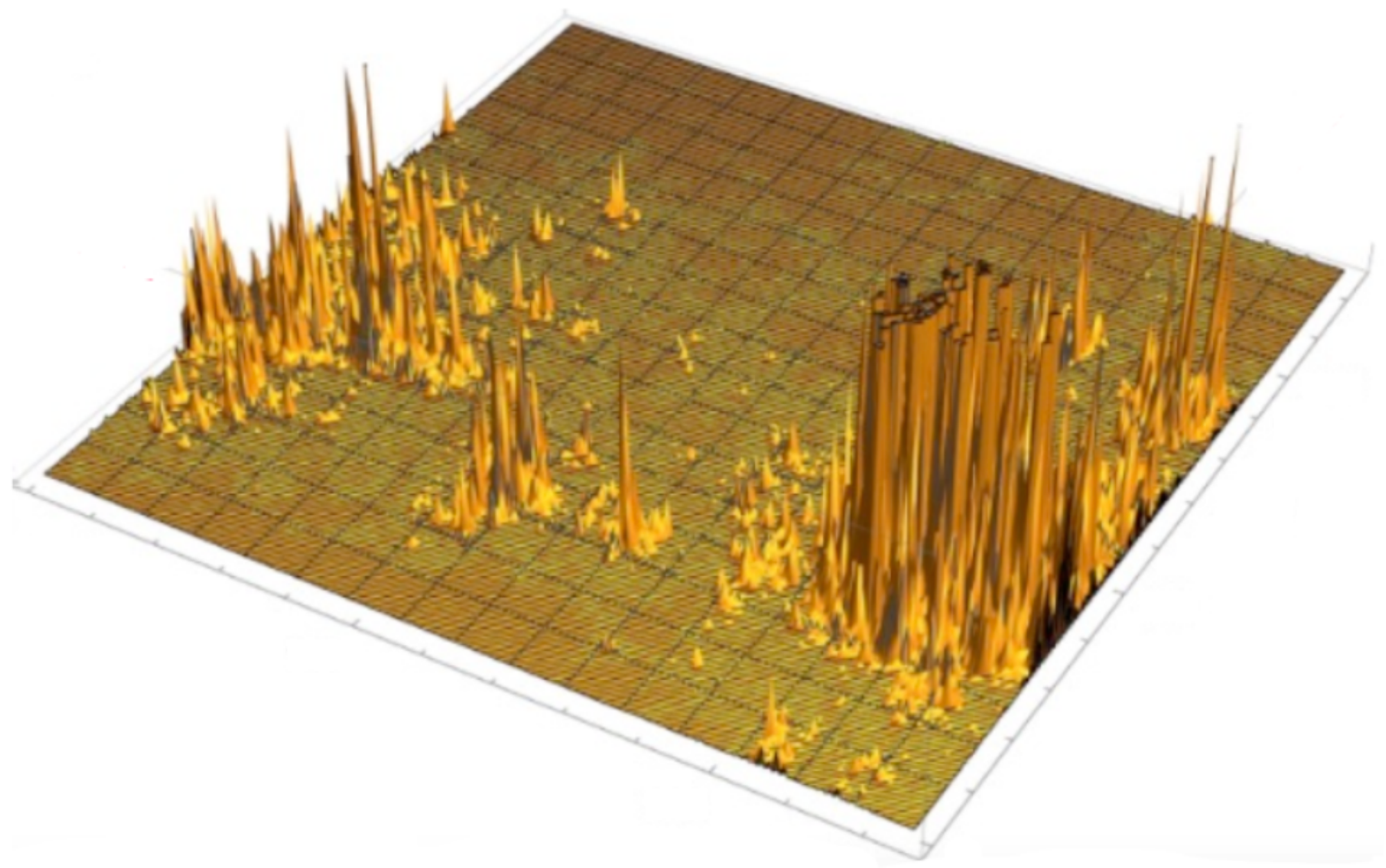}
\end{minipage}
\ \
\begin{minipage}[b]{.4\linewidth}
\centering
\includegraphics[width=\linewidth]{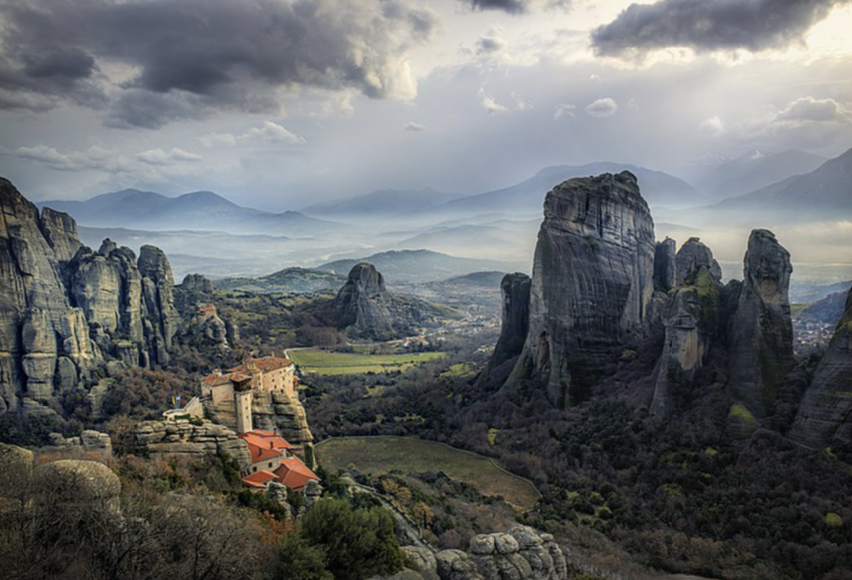}

~
\end{minipage}
\hfill \
\caption*{\it \footnotesize A picture of the critical $2d$ stochastic heat flow is shown on the left, and
on the right, a resembling natural landscape located in central Greece and named Meteora
(\href{https://en.wikipedia.org/wiki/Meteora\#/media/File:Meteora's_monastery_2.jpg}{photo} by
	\href{https://www.flickr.com/photos/antistath/}{Stathis Floros}, \href{https://creativecommons.org/licenses/by-sa/4.0/}{CC BY-SA 4.0}).
	The picture of the critical $2d$ stochastic heat flow has been obtained by simulating the partition function of the directed polymer model. 
	The plateaus in the simulation are a result of a truncation of high peaks.}
\end{figure}

\tableofcontents

\section{Introduction}

Our goal here is to review recent progress in making sense of the two-dimensional {\em stochastic heat equation} (SHE)
\cite{CSZ17b, CSZ19b, GQT21, CSZ23a, CSZ23b, T24}
\begin{equation}\label{eq:SHE}\tag{SHE}
	\partial_t u(t,x) =\frac{1}{2} \Delta u(t,x) + \beta \, \xi(t,x) \, u(t,x) \,,\qquad t>0, \ x\in\R^2 \,.
\end{equation}
This is a {\em singular stochastic partial differential equation} (SPDE) due to the presence of the term
$\xi(t,x)u(t,x)$, where the potential (or disorder)
$\xi(t,x)$ is a space-time white noise, i.e., a generalised Gaussian field with mean $0$ and covariance
$\bbE[\xi(t, x) \xi(s, y)]=\delta(t-s) \delta(x-y)$, and $u(t,x)$ is also expected
to be a generalised function (i.e.\ distribution in the sense of Schwartz). The constant $\beta \ge 0$
tunes the strength of the interaction.

Via the Cole-Hopf
transformation $h=\log u$, SHE is also related to the celebrated Kardar-Parisi-Zhang (KPZ) equation
\begin{equation}\label{eq:KPZ}\tag{KPZ}
    \partial_t h(t, x) = \frac{1}{2} \Delta h(t, x) + \frac{1}{2} |\nabla h(t,x)|^2 +
    \beta\,\xi(t, x) \,,\qquad t>0, \ x\in\R^d,
\end{equation}
which is a model for random interface growth and has been studied extensively in $d=1$ as a canonical example in the KPZ universality
class \cite{QS15, Cor12, Cor16, Z22}.

In recent years, there have been fundamental breakthroughs in the study of singular SPDEs \cite{H13, H14, GIP15, Kup16, GJ14}, including
the KPZ equation in $d=1$. However, these theories only apply to sub-critical singular SPDEs, while dimension $d=2$ is critical for SHE
and KPZ. In a series of papers culminating in \cite{CSZ23a}, we
were able to give a meaning to the solution of the $2d$  SHE on an
\emph{intermediate disorder scale} (i.e.\ by regularising the noise $\xi$
and, simultaneously, \emph{rescaling the disorder strength  $\beta \downarrow 0$ at a suitable rate},
as the regularisation is removed).
In particular, it was shown that
a phase transition exists on this scale, with Edwards-Wilkinson (Gaussian) fluctuation in the sub-critical regime \cite{CSZ17b}.
For the $2d$  KPZ, the same Edwards-Wilkinson fluctuation in the subcritical regime was established in \cite{CSZ20, G20}.
Most interestingly, it was shown in \cite{CSZ23a} that in a window
around the critical point, we can make sense of the solution of the $2d$  SHE as
a random measure-valued process, which is called the {\em critical $2d$ stochastic heat flow} (SHF). This is a rare example where a model
in the critical dimension and at the phase transition point admits a non-Gaussian limit, in contrast to e.g.\ the Ising and $\Phi^4$ models at the
critical dimension $d=4$ \cite{ADC21}.

Our motivation to make sense of the $2d$  SHE came from a different direction, which is the study of continuum limits of
disordered systems (see discussions in Section \ref{S:disrel}), in particular, the {\em directed polymer model} (DPM). The DPM is one of the simplest and yet most challenging
disordered system, which models a random walk interacting in equilibrium with a random environment (disorder). As the strength
of the disorder (inverse temperature $\beta$) increases, the model undergoes a phase transition from diffusive behaviour for the random walk,
with path delocalisation, to expected super-diffusive behaviour and path localisation (see \cite{C17, Z24} for more details). The DPM lies at the heart
of two different areas of intense research
in recent years. On the one hand, it is a canonical example in the KPZ universality class of random interface growth models, with the $0$-temperature
DPM in $d=1$ being exactly the {\em last passage percolation} model. On the other hand, it provides a discretisation of the SHE (and the KPZ equation via the Cole-Hopf transformation), which are classical examples of singular SPDEs. It is this latter connection that we will exploit to make sense of the $2d$  SHE.

In the rest of the introduction, we will first explain heuristically why $d=2$ is critical for SHE and what is the standard procedure to define the solution
of a singular SPDE such as the $2d$  SHE. We will then introduce the directed polymer model, its connection to SHE, and the intermediate
disorder scale on which a phase transition exists and our analysis will be carried out. We will then conclude the introduction with an outline of the
rest of the article.

\subsection{Stochastic heat equation}
We start with  a heuristic scaling argument that shows why $d=2$ is critical for the stochastic heat equation (SHE)
\begin{equation}\label{eq:SHE2}
	\partial_t u(t,x) =\frac{1}{2} \Delta u(t,x) + \beta \, \xi(t,x) \, u(t,x) \,,\qquad t>0, \ x\in\R^d.
\end{equation}
In the spirit of renormalisation group theory \cite{Kup16}, we note that the space-time rescaled solution
$\tilde u(t, x) := u(\epsilon^2 t, \epsilon x)$ formally solves
\begin{align}\label{eq:SHErg}
\partial_{ t} \tilde u =\frac{1}{2} \Delta \tilde u +
\beta\, \epsilon^{1-\tfrac{d}{2}} \,\tilde \xi \, \tilde u \,,
\qquad \tilde t>0, \ \tilde x\in\R^d,
\end{align}
where $\tilde \xi(t, x) := \epsilon^{1+\frac{d}{2}} \xi(\epsilon^2 t,
\epsilon x)$ is a new space-time white noise with the same distribution as $\xi$. We now see that as we send $\epsilon\downarrow 0$
and zoom into smaller and smaller
space-time scales, the strength of the noise vanishes in dimensions $d<2$ and diverges in dimensions $d>2$. But in the critical dimension
$d=2$, the exponent $1-\tfrac{d}{2}$ vanishes, and how scaling affects the effective strength of the noise depends on finer details of the model.
In the language of singular SPDEs \cite{H13, H14, GIP15}, the regimes $d<2$, $d=2$ and $d>2$ are called respectively {\em sub-critical}, {\em critical}, and {\em super-critical}.

The main difficulty in making sense of the solution of singular SPDEs such as \eqref{eq:SHE} is the roughness on small scales of the solution
(expected to be a generalised function), which makes terms such as $\xi(t, x)u(t, x)$ mathematically undefined. The standard procedure is
to perform a regularisation, also called {\em ultraviolet cutoff}, to smoothen things below a small spatial scale $\eps$, which leads to an approximate solution $u^\eps$. We then check whether there exist suitable centering and scaling constants $A_\eps$ and $B_\eps$, and suitable choice of the model parameter $\beta=\beta_\eps$, such that $(u^\eps-A_\eps)/B_\eps$ converges to a non-trivial limit as $\eps\downarrow 0$. If such a non-trivial limit exists, then we can interpret it
as the solution of (a renormalised version of) this singular SPDE.

There are several different ways of performing the ultraviolet cutoff. The most common approach is to mollify the space-time white noise $\xi$
on the spatial scale $\eps$. More precisely, let $j\in C_c^\infty(\R^2)$ be a smooth probability density function with compact support, and
define the mollified noise $\xi^\eps := j^\epsilon * \xi$ by convolving $\xi$ in space with $j^\eps(x):=\epsilon^{-2}j(\frac{x}{\eps})$. This
leads to the mollified SHE
\begin{equation} \label{eq:mollSHE}
\partial_t u^\eps = \frac{1}{2}\Delta u^\eps + \beta_\eps u^\eps \xi^\eps, \qquad u^\eps(0, \cdot)\equiv 1,
\end{equation}
which admits a classical It\^o solution. Furthermore, the solution admits a Feynman-Kac representation (see \cite[Section 3]{BC95}) with
\begin{equation} \label{eq:FK1}
u^\eps(t,x) = \E\Big[e^{\int_0^t (\beta_\eps \xi^\eps(t-s, B_s) - \lambda_\eps){\rm d}s} \Big| B_0=x \Big],
\end{equation}
where $\E[\cdot]$ is w.r.t.\ a standard Brownian motion $B$ in $\R^2$, and $\lambda_\eps:= \frac{\beta_\eps^2}{2} \bbV{\rm ar}(\xi^\eps(s, x))=\frac{\beta_\eps^2}{2 \eps^2}\Vert j\Vert_2^2$ ensures that $\bbE[u^\eps(t,x)]=1$.

Another approach to smoothen things on small scales is to discretize space and define an approximate solution on a grid with lattice spacing
$\eps$. This is the approach we will follow, where the approximate solution
turns out to coincide with the partition functions of the directed
polymer model. Yet another approach is to truncate the high frequency modes in the Fourier decomposition of $u(t, \cdot)$, see e.g.\ \cite{CT24} and
the references therein. It is expected that modulo a change of parameters, the different procedures of ultraviolet cutoff will lead to the same
limits, which are believed to be universal. This has been shown for the $2d$ SHE (and KPZ) in the subcritical regime in \cite{CSZ17b}, and in the critical window in \cite{CSZ23a, T24}.

\subsection{The KPZ equation}
The KPZ equation \eqref{eq:KPZ} in the critical dimension $d=2$ is also singular,
due to the presence of the term $|\nabla h|^2$.
To make sense of its solution, we also need to
perform an ultraviolet cutoff and then pass to the limit. In particular, if $\xi^\eps$ denotes the same spatially mollified noise as in \eqref{eq:mollSHE}, then we consider the mollified KPZ equation
\begin{equation}\label{eq:mollKPZ}
\partial_t h^{\eps} = \frac{1}{2}\Delta h^{\eps} + \frac{1}{2} |\nabla h^{\eps}|^2
	+ \beta_\eps \,  \xi^\eps - C_\eps, \qquad h^\eps(0, \cdot)\equiv 0,
\end{equation}
where $C_\eps:=\beta_\eps^2 \, \eps^{-2} \, \Vert j\Vert_2^2$ ensures that $h^\eps$ is related to the solution of the mollified SHE \eqref{eq:mollSHE} through the Cole-Hopf transformation $h^\eps(t, x) = \log u^\eps(t, x)$, see e.g.~\cite{CSZ20}.

One could also consider the KPZ equation with different coefficients in front of the Laplacian and the non-linearity, that is,
\begin{equation}\label{eq:genKPZ}
	\partial_t \tilde h^{\eps} = \frac{\nu}{2}\Delta \tilde h^{\eps} + \frac{\lambda}{2}|\nabla \tilde h^{\eps}|^2 + D \xi^\eps \,.
\end{equation}
For example, we could set $\nu=D=1$ and $\lambda=\beta_\eps$ as in \cite{CD20}.
However, by scaling (see e.g.\ \cite[Appendix A]{CSZ20}), the three parameters $\nu, \lambda$ and $D$ can be reduced to a single {\em effective coupling constant} $\beta^2:= \lambda^2 D/\nu^3$. Therefore it suffices to consider \eqref{eq:mollKPZ} with only the noise strength $\beta_\eps$ dependent on $\eps$.

\subsection{Directed polymer model}\label{S:DPM} The directed polymer model (DPM) is defined as a random walk interacting with a random environment (disorder) \cite{C17, Z24}. To avoid complications caused by periodicity, instead of considering the simple symmetric random walk on $\Z^2$ as done in \cite{CSZ23a}, we consider an irreducible aperiodic random walk $S = (S_n)_{n\ge 0}$ whose increment $\xi:=S_1-S_0$ has mean $0$, covariance matrix being the identity matrix, and $\E[e^{b|\xi|}]<\infty$
for some $b>0$. Let $\P$ and $\E$ denote probability and expectation for $S$, and denote its
transition probability kernels by
\begin{equation} \label{eq:rw}
	q_n (z) := \P (S_n = z \,|\, S_0 = 0), \qquad z\in \Z^2, n\in \N_0.
\end{equation}

The disorder (random environment) is given by a family of i.i.d.\ random variables $\omega:=(\omega(n,z))_{n\in\N, z\in\Z^2}$ with
\begin{equation}\label{eq:lambda}
\begin{gathered}
	\bbE[\omega]=0 \,, \qquad \bbE[\omega^2]=1 \,, \\
	\exists\, \beta_0>0 \quad \text{such that} \qquad 	\lambda(\beta)
	:= \log \bbE[e^{\beta\omega}] < \infty \qquad \forall \beta \in [0, \beta_0] \,.
\end{gathered}
\end{equation}
Probability and expectation for $\omega$ will be denoted by $\bbP$ and $\bbE$.

Given $N\in\N$, $\beta>0$, $z\in \Z^2$, and $\omega$, the polymer with length $N\in\N$, starting point $z$,
and disorder strength (inverse temperature) $\beta$ in the random environment $\omega$ is defined by the Gibbs measure
\begin{align}\label{eq:pathmeasure}
	\dd\P^{\beta,\, \omega}_{N}(S \,|\,S_0=z)
	:=\frac{1}{Z_N^{\beta}(z)}
	e^{\sum_{n=1}^{N-1} \{\beta \omega(n,S_n) - \lambda(\beta)\}} \,
	\dd\P(S \,|\, S_0 = z) \,,
\end{align}
where
\begin{align} \label{eq:paf}
	Z_N^{\beta}(z) &=
	\E\bigg[ e^{\sum_{n=1}^{N-1} \{\beta \omega(n,S_n) - \lambda(\beta)\}}
	\,\bigg|\, S_0 = z \bigg]
\end{align}
is called {\it point-to-plane partition function}. Note that $\lambda(\beta)$ in the exponent ensures that $\bbE[Z_N^{\beta}(z)]=1$ and, in fact, $(Z_N(z))_{N\in\N}$
is a martingale. It is also useful to consider the point-to-point partition functions: for $M \le N \in \N_0 = \{0,1,2,\ldots\}$ and $x, y\in\Z^2$, we
define
\begin{equation} \label{eq:ZMN}
	Z^{\beta}_{M,N}(x,y) := \E \bigg[
	e^{\sum_{n=M+1}^{N-1} \{\beta \omega(n,S_n) - \lambda(\beta)\}}
	\, \ind_{S_N = y} \,\bigg|\, S_M = x \bigg]  \,,
\end{equation}
and denote $Z^{\beta}_{M,N}(x,\ind) := \sum_{y\in\Z^2} Z^{\beta}_{M,N}(x,y)$ and $Z^{\beta}_{M,N}(\ind, y) := \sum_{x\in\Z^2} Z^{\beta}_{M,N}(x,y)$.

Comparing \eqref{eq:paf} with \eqref{eq:FK1}, we note that the Feynman-Kac representation of the solution $u^\eps(t,x)$ of the mollified SHE is
the partition function of a Brownian directed polymer in a white-noise random environment,
where the random walk $(S_i)_{0\leq i\leq N}$ in \eqref{eq:paf} is replaced by the Brownian motion
$(B_s)_{0\leq s\leq t}$, and the i.i.d.\ space-time disorder $\omega(\cdot, \cdot)$ is replaced by the
time-reversed mollified space-time white noise
$\xi^\eps(t-\cdot, \cdot)$. If we rescale space-time diffusively in the DPM, then $(S_{\lfloor Nt\rfloor}/\sqrt{N})_{t\geq 0}$ converges in law to $(B_t)_{t\geq 0}$, while $(\frac{1}{N}\omega(\lfloor Nt\rfloor, \lfloor \sqrt{N}x\rfloor))_{(t, x)\in\R\times\R^2}$ converges in law to the space-time
white noise $\xi$. In particular, the family of partition functions
\begin{equation}\label{eq:uNtx}
u^N(t, x):= Z^{\beta_N}_{0, tN}(\ind, x\sqrt{N}), \qquad t\in \frac{1}{N}\N_0, \ x\in \frac{1}{\sqrt N}\Z^2,
\end{equation}
gives an approximate solution to the $2d$  SHE through a discretisation of space and time. Compared with $u^\eps$ in \eqref{eq:mollSHE},
we see that $N$ corresponds to $\eps^{-2}$, and $\beta_N$ corresponds to $\beta_\eps$.

Instead of asking whether there are suitable choices of $\beta_\eps$, $A_\eps$ and $B_\eps$ such that $(u^\eps-A_\eps)/B_\eps$ has a non-trivial
limit, we can consider instead whether there are suitable choices of $\beta_N$, $A_N$ and $B_N$ such that $(u^N-A_N)/B_N$ has a non-trivial
limit. It turns out that the correct choice of $\beta_N$ is exactly the scale on which the DPM undergoes a phase transition.

\subsection{Phase transition on an intermediate disorder scale}

It is known that (see \cite{C17, Z24} and the references therein) the DPM undergoes a phase transition in the sense that, there is a critical value
$\beta_c\in [0,\infty)$ such that for $\beta<\beta_c$, the partition function $Z^\beta_N(0)$ defined in \eqref{eq:paf} converges a.s.\ to
a positive random variable $Z_\infty(0)$, and the polymer measure $\P_N^{\beta, \omega}$ converges to the Wiener measure under diffusive
scaling of space-time and is hence delocalised in space \cite{CY06}. On the other hand, when $\beta>\beta_c$, $Z^\beta_N(0)$ converges almost surely
to $0$, and the polymer measure $\P_N^{\beta, \omega}$ experiences localisation in space \cite{CSY03}. A recent breakthrough \cite{JL24} has shown that
when the disorder $\omega$ is bounded, there is also delocalisation at $\beta_c$, while localisation always manifests itself in a strong form.

It is known that $\beta_c=0$ in dimensions $d=1, 2$, and $\beta_c>0$ in $d\geq 3$. However, it was discovered in \cite{CSZ17b} that in $d=2$, there is
still a phase transition on an intermediate disorder scale. Recall that the expected replica overlap between two independent random walks
with the same distribution as $S$ is defined by
\begin{equation}\label{eq:RL}
R_{N} := \sum_{n = 1}^{N}
     \sum_{x \in \mathbb{Z}^2} q_n (x)^2 = \sum_{n = 1}^{N} q_{2 n} (0) \sim  \frac{\log N}{4\pi} \qquad \text{as } N\to\infty \,,
\end{equation}
where the asymptotics follows from the local limit theorem, see \eqref{eq:llt} below.

It was shown in \cite{CSZ17b} that if we choose $\beta$ to be
\begin{equation}\label{eq:hbeta}
\beta_N = \frac{\hat \beta + o(1)}{\sqrt{R_N}}, \qquad \hat\beta \in (0, \infty),
\end{equation}
then a phase transition occurs at the critical value $\hat\beta_c=1$ (see Section \ref{S:subcrit} below for more details).
We call $\hat\beta<1$ the {\em subcritical regime}, and $\hat\beta>1$ the {\em supercritical regime}. For the mollified SHE \eqref{eq:mollSHE},
the analogous choice is (see \cite[Theorem 2.15]{CSZ17b})
\begin{equation}\label{eq:epshbeta}
\beta_\eps = (\hat\beta +o(1)) \sqrt{\frac{2\pi}{\log \eps^{-1}}}, \qquad \hat\beta \in (0, \infty).
\end{equation}
This intermediate disorder scale turns out to be the correct scale on which $u^N$ from \eqref{eq:uNtx} has non-trivial limits.

As will be discussed in more detail in Section \ref{S:critical}, around the critical point $\hat\beta=1$, there is in fact a finer window which we call the {\em critical window}, with $\beta_N$ in this window satisfying
\begin{equation} \label{eq:sigma}
	\sigma_N^2:= e^{\lambda(2\beta_N)-2\lambda(\beta_N)}-1
	= \frac{1}{R_N} \bigg(1 + \frac{\theta + o(1)}{\log N}\bigg) \quad\quad
	\text{for some fixed } \theta\in \R \,.
\end{equation}
Note that for every $\theta\in \R$, we have $\beta_N \sim 1/\sqrt{R_N}$ because $\lambda(\beta) \sim \frac{1}{2} \beta^2$ as $\beta \downarrow 0$,
see \eqref{eq:lambda}. So the parameter $\theta$ only appears in the second order asymptotics of $\beta_N$. However, different $\theta$ will lead to different scaling limits for the DPM partition functions regarded as a random field, which are the so-called {\em critical $2d$ stochastic heat flows}. It will become clear in Section \ref{S:2momtrans} why it is more convenient to define the critical window in terms of $\sigma_N^2= \sigma_N^2(\theta)$.

For the mollified SHE \eqref{eq:mollSHE}, the critical window corresponds to choosing
\begin{equation}\label{eq:DBbeta2}
\beta^2_\epsilon= \frac{2\pi}{\log \frac{1}{\epsilon}} \Big(1+\frac{\theta + o(1)}{|\log \epsilon|}\Big).
\end{equation}

\subsection{Outline} \label{S:outline}
Assume that $\beta_N$ is chosen as in \eqref{eq:hbeta} for some $\hat\beta>0$, and in the case $\hat\beta=1$, $\beta_N$ is chosen in the critical
window as in \eqref{eq:sigma} for some $\theta\in \R$. Recall $u^N(t,x)$ from \eqref{eq:uNtx}. We will discuss the following results in the rest of
this article.

\begin{itemize}
\item Sections \ref{S:subcrit} and \ref{sec:2dKPZ} will be devoted to results in the subcritical regime $\hat\beta<1$, which include:
\begin{itemize}
\item[-] As $N\to\infty$, for each $(t, x)$, $u^N(t,x)$ defined in \eqref{eq:uNtx} converges to a log-normal limit when $\hat\beta<1$ and converges to $0$ when $\hat\beta\geq 1$. This shows that there is a phase transition with critical value $\hat\beta_c=1$.

\item[-] After suitable scaling, the centred random field $u^N(t, \cdot)-1$ converges to a Gaussian limit, which is the solution of the additive SHE, also known as the Edwards-Wilkinson equation.

\item[-] After suitable scaling, the centred random field $\log u^N(t, \cdot)- \bbE[\log u^N(t, \cdot)]$, which solves the $2d$ KPZ equation regularised through space-time discretisation, converges to the same Gaussian limit as $u^N(t, \cdot)-1$.
\end{itemize}

\item In Section \ref{S:related}, we discuss various results connected to the $2d$ SHE in the subcritical regime.

\item Sections \ref{S:critical} and \ref{S:CG} will be devoted to the following key result (and its proof outline) in the critical window for some $\theta\in \R$:
\begin{itemize}
\item[-] Without centering and scaling, the random measure $u^N(t, \cdot)$ converges in law to a unique limit, which is called the {\em critical $2d$ stochastic heat flow} (SHF) with parameter $\theta$,
denoted by $\SHF(\theta) = (\SHF^\theta_{s,t}(\dd x, \dd y))_{0 \le s \le t < \infty}$.
\end{itemize}

\item In Sections \ref{S:2ndmom} and \ref{S:highmom}, we will explain some key technical ingredients in the proof, focusing on the second moment and higher moments calculations.

\item In Section \ref{sec:prop}, we will discuss some properties of the critical $2d$ SHF, including:
\begin{itemize}
\item[-] An almost sure Chapman-Kolmogorov type property that justifies the name {\em flow};

\item[-] A recent axiomatic characterisation of the SHF$(\theta)$ by Tsai \cite{T24};

\item[-] SHF$(\theta)$ cannot be obtained as the exponential of a generalised Gaussian field (i.e., it cannot be a Gaussian multiplicative chaos);

\item[-] The marginal distribution of SHF$(\theta)$ at each time is almost surely singular w.r.t.\ the Lebesgue measure.
\end{itemize}

\item In Section \ref{S:open}, we will discuss some open questions.
\end{itemize}

Before discussing the results above, in Section \ref{S:prelim}, we will first warm up by introducing the polynomial chaos expansion for the directed polymer partition function and performing a second moment calculation that identifies the intermediate disorder scale.

\section{Basic tools and preliminary results} \label{S:prelim}

\subsection{Polynomial chaos expansion} \label{S:poly}
The starting point of our analysis is the \emph{polynomial chaos expansion} of the partition function,
based on the Mayer cluster expansion. More precisely, with $\beta=\beta_N$ allowed to depend on $N$,
for each $(n, z) \in \N\times\Z^2$, we can write
$$
e^{(\beta_N \omega(n, z) - \lambda(\beta_N)) \ind_{S_n=z}} = 1 + \xi_N(n,z)\ind_{S_n=z},
$$
where $(\xi_N(n,z))_{(n,z)\in \N\times \Z^2}$ is an i.i.d.\ family defined by
\begin{equation} \label{eq:xi}
\begin{gathered}
	\xi_N(n,z) := e^{\beta_N\omega(n,z)-\lambda(\beta_N)}-1 \\
	\text{with} \qquad \bbE[\xi_N(n,z)] = 0 \,, \qquad
	\bbvar[\xi_N(n,z)] =  e^{\lambda(2\beta_N)-2\lambda(\beta_N)}-1 =:\sigma_N^2 \,.
\end{gathered}
\end{equation}
We can then rewrite the partition function $Z^{\beta_N}_N(z)$ from \eqref{eq:paf} as follows:
\begin{align}
	Z^{\beta_N}_N(z) & = \E\bigg[ e^{\sum_{n=1}^{N-1} (\beta_N \omega(n,S_n) - \lambda(\beta_N))} \bigg] \notag \\
                 & = \E\bigg[ \prod_{(n, x) \in \{1, \ldots, N-1\}\times \Z^2} e^{(\beta_N \omega(n,x) - \lambda(\beta_N)) \ind_{S_n=x}}\bigg] \notag \\
	             & = \E\bigg[ \prod_{(n, x) \in \{1, \ldots, N-1\}\times \Z^2} \big(1+ \xi_N(n,x)\ind_{S_n=x}\big)\bigg] \notag \\
                 & = 1 + \sum_{r=1}^\infty \sum_{n_0=0<n_1<\cdots <n_r<N \atop z_0:=z, z_1, \ldots, z_r\in \Z^2} \prod_{i=1}^r q_{n_i-n_{i-1}}(z_i-z_{i-1}) \xi_N(n_i, z_i). \label{eq:expan1}
\end{align}
This is called a polynomial chaos expansion because it can be seen as a discrete analogue of a Wiener-It\^o chaos expansion with respect to a space-time white noise on $[0,\infty)\times \R^2$. The terms of the expansion are $L^2$-orthogonal, and each term is a multilinear polynomial in the i.i.d.\ random
variables $(\xi_N(n,x))_{(n,x)\in \N\times \Z^2}$.

For the point-to-point partition function $Z^{\beta_N}_{M, L}(x, y)$ defined in \eqref{eq:ZMN}, we have a similar expansion:
\begin{align}
	Z^{\beta_N}_{M, L}(x, y)
                 & = q_{L-M}(y-x)  \label{eq:expan2} \\
                 & \quad + \sum_{r=1}^\infty \sum_{z_0:=x, z_1, \ldots, z_r\in \Z^2 \atop n_0:=M <n_1<\cdots <n_r<L} \prod_{i=1}^r q_{n_i-n_{i-1}}(z_i-z_{i-1}) \xi_N(n_i, z_i) \cdot q_{L-n_r}(y-z_r)  \notag \\
                 & = q_{L-M}(y-x) +\sum_{r=1}^\infty \!\!\!\!\!\! \!\! \sum_{z_0:=x, z_1, \ldots, z_r\in \Z^2 \atop n_0:=M <n_1<\cdots <n_r<L} \!\!\!\!\!\!\!\!\!\!
\begin{tikzpicture}[scale=0.45]]
\draw [] (0, 0)  circle [radius=0.1];
\draw [] (14, 1)  circle [radius=0.1];
\draw[thick] (0,0) to [out=70,in=160] (2,1);
\draw [fill] (2, 1)  circle [radius=0.1];
\draw[thick] (2,1) to [out=-50,in=110] (3,-1);
\draw [fill] (3, -1)  circle [radius=0.1];
\draw[thick] (3,-1) to [out=80,in=-180] (4.5,0.5);
\draw [fill] (4.5, 0.5)  circle [radius=0.1];
\draw[thick] (4.5,0.5) to [out=60,in=120] (6, 0.5);
\draw [fill] (6, 0.5)  circle [radius=0.1];
\draw[thick] (6,0.5) to [out=-80,in=160] (7, -1);
\draw [fill] (7, -1)  circle [radius=0.1];
\draw[thick] (7,-1) to [out=60,in=200] (8, 0);
\node at (9,0) {$\cdots$};\node at (10.2,0) {$\cdots$};
\draw[thick] (11,1) to [out=10,in=100] (13, 0);
\draw [fill] (13, 0)  circle [radius=0.1];
\draw[thick] (13,0) to [out=60,in=180] (14, 1);
\node at (13,-0.5) {\scalebox{0.6}{$\xi_N(n_r, z_r)$}};
\node at (2,1.7) {\scalebox{0.6}{$\xi_N(n_1,z_1)$}};
\node at (0,-0.5)  {\scalebox{0.6}{$(M, x)$}};
\node at (14.5,0.5)  {\scalebox{0.6}{$(L, y)$}};
\end{tikzpicture}, \notag
\end{align}
where in the graphical representation, the laces represent the transition kernels $q_{n_i-n_{i-1}}(z_i-z_{i-1})$,  the solid dots
represent the disorder variables $\xi_N(n_i, z_i)$, and the hollow circles at $(M, x)$ and $(L, y)$
means the absence of the disorder variables $\xi_N(M, x)$ and $\xi_N(L, y)$.

Similarly, the solution $u^\eps$ of the mollified SHE \eqref{eq:mollSHE} admits a Wiener-It\^o chaos expansion
w.r.t.\ the underlying white noise $\xi$ (see \cite[(5.9)]{CSZ20}), which can also be obtained by Picard iterations:
\begin{equation} \label{eq:WienerKPZ}
	u^\epsilon(t, x) \overset{\rm dist }{=} 1+ \!\! \sum_{k\geq1}
	\beta_\epsilon^k \!\!\!\!\!
	\idotsint\limits_{z_1, \ldots, z_k \in \R^2 \atop 0<s_1<\cdots<s_k< t\epsilon^{-2}} \!\!
	\Bigg( \, \int\limits_{(\R^2)^k} \!\prod_{i=1}^k g_{s_i-s_{i-1}}(y_i-y_{i-1})
	\,j(y_i-z_i) \,\dd \vec y\Bigg) \prod_{i=1}^k \tilde \xi(s_i, z_i) \dd s_i \dd z_i \,,
\end{equation}
where $s_0 :=0$, $y_0 :=\eps^{-1}x$, $g_s(\cdot)$ is the heat kernel, and $\tilde \xi$ is obtained from $\xi$ via
diffusive rescaling and has the same distribution.

\subsection{A transition in the second moment}\label{S:2momtrans}
We now show that the second moment of the point-to-plane partition function $\bbE[Z^{\beta_N}_N(0)^2]$, see \eqref{eq:paf}, undergoes a transition exactly at the critical point $\hat\beta=1$ defined in \eqref{eq:hbeta}. The fact that this $L^2$ transition point coincides with the actual phase transition point is unique to dimension $d=2$ and far from trivial (see Section \ref{S:subcrit} for more details). We will also see in Corollary \ref{C:2mom} that for $\beta_N$ chosen as in \eqref{eq:hbeta}, changing the time horizon from $N$ to $N^\alpha$ effectively replaces $\hat\beta^2$ by $\alpha \hat\beta^2$.

We first recall the local central limit theorem. Let $g_t$ be the heat kernel on $\R^2$, i.e.,
\begin{equation}\label{eq:gt}
	g_t(x) := \frac{1}{2\pi t} \, e^{-\frac{|x|^2}{2t}} \,,
	\qquad g_t(x,y) := g_t(y-x) \,.
\end{equation}
Since the increments of the random walk $(S_n)_{n\geq 0}$ have an identity covariance matrix, by the local central limit theorem, see e.g.\ \cite[Theorems 2.3.5 \& 2.3.11]{LL10}, we have
\begin{equation}\label{eq:llt}
	\begin{aligned}
	q_n(x) \,=\, g_n(x) + O\big(\tfrac{1}{n^2}\big) \,=\, g_n(x) \, e^{O\big(\tfrac{1}{n}\big) + O\big(\tfrac{|x|^4}{n^3}\big)}, \qquad n\in\N, x\in \Z^2.
\end{aligned}
\end{equation}

\begin{lemma}[Transition in the second moment]\label{L:2mom}
Let $\beta_N =\hat \beta/\sqrt{R_N}$ for some $\hat\beta\in (0,\infty)$. Then the point-to-plane partition function $Z_N:=Z^{\beta_N}_N(0)$ defined in \eqref{eq:paf} satisfies
\begin{equation}\label{eq:2momlim}
\lim_{N\to\infty} \bbE[Z_N^2] = \left\{ \begin{aligned}
\frac{1}{1-\hat\beta^2} & \qquad \mbox{for } \hat\beta<1, \\
\infty \quad \quad & \qquad \mbox{for } \hat\beta\geq 1.
\end{aligned}
\right.
\end{equation}
\end{lemma}
We will show in Section \ref{S:2ndmom} that $\bbE[Z_N^2]\sim C\log N$ at the critical point $\hat\beta=\hat\beta_c=1$.
\begin{proof}
By the polynomial chaos expansion \eqref{eq:expan1} for the point-to-plane partition function $Z_N:=Z^{\beta_N}_N(0)$, which is an $L^2$-orthogonal expansion,
we have
\begin{align}
	\bbE[Z_N^2] & = 1 + \sum_{r=1}^\infty \sum_{n_0=0<n_1<\cdots <n_r<N \atop z_1, \ldots, z_r\in \Z^2} \prod_{i=1}^r q_{n_i-n_{i-1}}(z_i-z_{i-1})^2 \bbE[\xi_N(n_i, z_i)^2] \notag \\
& = 1 + \sum_{r=1}^\infty \sum_{n_0=0<n_1<\cdots <n_r<N} \sigma_N^{2r} \sum_{z_1, \ldots, z_r\in \Z^2} \prod_{i=1}^r q_{n_i-n_{i-1}}(z_i-z_{i-1})^2 \notag \\
& = 1 + \sum_{r=1}^\infty \sum_{n_0=0<n_1<\cdots <n_r<N} \sigma_N^{2r} \prod_{i=1}^r q_{2(n_i-n_{i-1})}(0).
\label{eq:2mom1}
\end{align}
Note that since $\beta_N\to 0$, we have
\begin{equation}\label{eq:sigN2}
\sigma_N^2 = e^{\lambda(2\beta_N)-2\lambda(\beta_N)}-1 \sim \beta_N^2=\frac{\hat\beta^2}{R_N} \sim \frac{4\pi \hat\beta^2}{\log N} \quad \mbox{as} \quad N\to\infty.
\end{equation}
For each $r\in\N$, we then have the upper bound
\begin{align*}
\sum_{n_0=0<n_1<\cdots <n_r<N} \!\!\! \!\!\sigma_N^{2r} \prod_{i=1}^r q_{2(n_i-n_{i-1})}(0)
&\leq \Big(\sigma_N^2 \sum_{n=1}^N q_{2n}(0)\Big)^r
\sim \Big(\frac{4\pi \hat\beta^2}{\log N} \sum_{n=1}^{N} \frac{1}{4\pi n}\Big)^r = (\hat\beta+o(1))^{2r},
\end{align*}
and the lower bound
\begin{align*}
\sum_{n_0=0<n_1<\cdots <n_r<N} \sigma_N^{2r} \prod_{i=1}^r q_{2(n_i-n_{i-1})}(0) \geq \Big(\sigma_N^2 \sum_{n=1}^{N/r} q_{2n}(0)\Big)^r
\sim \Big(\frac{4\pi \hat\beta^2}{\log N} \sum_{n=1}^{N/r} \frac{1}{4\pi n}\Big)^r \sim \hat\beta^{2r},
\end{align*}
where we used that $R_N=\sum_{n=1}^N q_{2n}(0)$ and the local limit theorem \eqref{eq:llt}. For each $r\in\N$, the upper and lower bounds
match asymptotically. The conclusion \eqref{eq:2momlim} then follows readily.
\end{proof}

The following corollary shows that changing the time horizon from $N$ to $N^\alpha$ effectively changes $\hat\beta^2$ to $\alpha\hat\beta^2$.
It also implies that in the subcritical regime $\hat\beta<1$, the partition function $Z^{\beta_N}_N(0)$ can be approximated in $L^2$ by $Z^{\beta_N}_{N^\alpha}(0)$ with $\alpha\in (0,1)$ close to $1$. In particular, $Z^{\beta_N}_N(0)$ is essentially determined by disorder variables $\omega(n,x)$ (equivalently $\xi_N(n, x)$) with $n\ll N$. This fact will be an important feature that drives the phenomenology
and the analysis for $\hat\beta<1$. This feature will become clear in our sketch of the proof of Theorem \ref{S:1pt}.

\begin{corollary}[Exponential time scale]\label{C:2mom}
Let $\beta_N =\hat \beta/\sqrt{R_N}$ for some $\hat\beta\in (0,\infty)$, and let $\alpha\in (0,\infty)$. Then we have
\begin{equation}\label{eq:2momlim2}
\lim_{N\to\infty} \bbE[(Z^{\beta_N}_{N^\alpha})^2] = \left\{ \begin{aligned}
\frac{1}{1-\alpha \hat\beta^2} & \qquad \mbox{for } \alpha\hat\beta^2<1, \\
\infty \quad \quad & \qquad \mbox{for } \alpha\hat\beta^2\geq 1.
\end{aligned}
\right.
\end{equation}
Furthermore, for $\hat\beta \in (0,1)$ and $\alpha \in (0, 1]$, we have
\begin{equation}\label{eq:2momlim3}
\lim_{N\to\infty} \Vert Z^{\beta_N}_{N^\alpha} -Z^{\beta_N}_N \Vert :=  \lim_{N\to\infty}\bbE[(Z^{\beta_N}_{N^\alpha} -Z^{\beta_N}_N)^2]^{1/2}
= \frac{(1-\alpha) \hat\beta^2}{(1-\alpha \hat\beta^2)(1-\hat\beta^2)} \stackrel{\alpha\uparrow 1}{\longrightarrow} 0.
\end{equation}
\end{corollary}
\begin{proof}
If we set $M:=N^\alpha$, then using the asymptotics for $R_N$ from \eqref{eq:RL}, we can write
$$
\beta_N= \frac{\hat \beta}{\sqrt{R_N}} = \frac{1}{\sqrt{R_M}} \cdot \frac{\hat\beta \sqrt{R_M}}{\sqrt{R_N}} = \frac{\hat\beta \sqrt{\alpha}+o(1)}{\sqrt{R_M}}.
$$
The corollary then follows immediately from Lemma \ref{L:2mom}. Alternatively, we can also repeat the proof of Lemma \ref{L:2mom} and note that the key calculation is now
$$
\sigma_N^2 \sum_{n=1}^{N^\alpha} q_{2n}(0) = \frac{\hat\beta^2}{R_N} \cdot R_{N^\alpha} \sim \hat\beta^2 \frac{\log N^\alpha}{\log N}=\alpha \hat\beta^2,
$$
which replaces the original parameter $\hat\beta^2$.

Note that for $\alpha<1$, the polynomial chaos expansion for $Z^{\beta_N}_{N^\alpha}$, see \eqref{eq:expan1}, contains only a subset of terms in the polynomial chaos expansion for $Z^{\beta_N}_{N}$, which is an $L^2$-orthogonal expansion. The claim in \eqref{eq:2momlim3} then follows from \eqref{eq:2momlim2}.
\end{proof}

\section{Gaussian fluctuations in the subcritical regime}\label{S:subcrit}
In this section, we review results in the subcritical regime, that is, results (a)-(c) listed in Section \ref{S:outline}. Throughout the rest
of this section, we will consider $\beta_N$ chosen as in \eqref{eq:hbeta} for some $\hat\beta \in (0,1)$. Instead of considering $u^N(t,x)$, the solution of the space-time discretised SHE defined in \eqref{eq:uNtx}, for simplicity, we will fix $t=1$ and consider the time reversed field of point-to-plane
partition functions $Z^{\beta_N}_N(x\sqrt{N})= Z^{\beta_N}_{0, N}(x\sqrt{N}, \ind)$, $x\in \frac{1}{\sqrt N} \Z^2$.

\subsection{One point fluctuation} \label{S:1pt}
In \cite[Theorem 2.8]{CSZ17b}, it was shown that on the intermediate disorder scale defined in \eqref{eq:hbeta}, the DPM in dimension $d=2$ undergoes a phase transition similar to the phase transition in $d\geq 3$.

\begin{theorem}[Limit of individual partition function] \label{T:subcritical}
We have the following convergence result for the point-to-plane partition function $Z^{\beta_N}_N(0)$:
\begin{equation} \label{eq:conve}
	Z^{\beta_N}_N(0) \xrightarrow[\,N\to\infty\,]{d}
	\bs{Z}_{\hat\beta} :=
	\begin{cases}
	\exp\bigg(\sigma_{\hat\beta} \, W_1 - \frac{\sigma_{\hat\beta}^2}{2}
	\bigg)
	& \text{if } \hat \beta < 1 \\
	0 & \text{if } \hat\beta \ge 1
	\end{cases} \,.
\end{equation}
where $W_1$ is a standard Gaussian random variable and
\begin{equation}\label{eq:sigmahatbeta}
	\sigma_{\hat\beta}^2 := \log \frac{1}{1-\hat\beta^2} \,.
\end{equation}
\end{theorem}
\begin{remark}\label{R:1pt}
In \cite[Theorem 2.12]{CSZ17b}, it was shown that $\log Z^{\beta_N}_N(x_N)$ and $\log Z^{\beta_N}_N(y_N)$ converge to a pair of independent normal random variables in the subcritical regime $\hat\beta<1$ if $x_N/\sqrt{N}$ and $y_N/\sqrt{N}$ converge to distinct $x, y\in \R^2$. A non-trivial bivariate normal limit arises only if $\Vert x_N-y_N\Vert = N^{\alpha+o(1)}$ for some $\alpha \in (0, 1/2)$. Similar multivariate normal limits hold for
the joint distribution of $k$ $\log$-partition functions $\log Z^{\beta_N}_{N^{\gamma_i}}(x^i_N)$, $1\leq i\leq k$, with non-trivial change of the covariance matrix as we vary the exponents $\gamma_i\in (0,1)$ in the time horizon $N^{\gamma_i}$.
\end{remark}

In what follows, we sketch several alternative proof strategies for Theorem \ref{T:subcritical}, which will help shed light on the structure of the $2d$  DPM and SHE in the subcritical regime.

\medskip

\noindent
{\bf Proof strategy 1 \cite[Theorem 2.15]{CSZ17b}.} We start with the polynomial chaos expansion \eqref{eq:expan1} for the point-to-plane partition function
\begin{equation} \label{eq:expan11}
Z_N^{\beta_N}(0) = 1 + \sum_{r=1}^\infty \sum_{n_0=0<n_1<\cdots <n_r<N \atop z_0:=0, z_1, \ldots, z_r\in \Z^2} \prod_{i=1}^r q_{n_i-n_{i-1}}(z_i-z_{i-1}) \xi_N(n_i, z_i),
\end{equation}
where the parameter $\hat\beta<1$ is contained in the variance of $\xi_N$, see \eqref{eq:sigN2}. From the second moment calculations in Lemma \ref{L:2mom} and Corollary \ref{C:2mom}, we see that the dominant contribution to the series in \eqref{eq:expan11} comes from terms which of order $r=O(1)$ -- there is an exponential decay of order $\hat\beta^r$ in the chaoses of order $r$ -- the times differences are of order $n_i-n_{i-1}=: N^{t_i}$ for $t_i\in (0,1)$, and the spatial coordinates $\Vert z_i-z_{i-1}\Vert = O(\sqrt{n_i-n_{i-1}})=O(N^{t_i/2})$, for $1\leq i\leq r$. The contributions coming from $t_i\approx t_j$ for $i\neq j$ is negligible, and furthermore, the dominant contribution comes from $\Vert z_i\Vert =O(\sqrt{n_i})$, $1\leq i\leq r$.

The key observation of \cite{CSZ17b} is that, we should rewrite \eqref{eq:expan11} by partitioning
$$
(t_1, t_2, \ldots, t_r) = (s_{1, 1}, \ldots, s_{1, r_1}; s_{2, 1}, \ldots, s_{2, r_2}; \ldots; s_{k, 1} \ldots, s_{k, r_k}),
$$
where $s_{1, 1}$, $s_{2, 1}$, \ldots, $s_{k, 1}$ are the successive running maxima along the original sequence $(t_1, \ldots, t_r)$, and each
subsequence $(s_{1, 1}, \ldots, s_{i, r_i})$ is called a {\em dominated sequence} because the first element $s_i:= s_{i, 1}$ dominates the subsequent elements $s_{i, 2}$, \ldots, $s_{i, r_i}$. Recall that in \eqref{eq:expan11}, $N^{s_i}=n_{M+1}-n_M$ with $M=r_1+\cdots + r_{i-1}$, and
$$
n_M= N^{t_1} +\cdots N^{t_M}.
$$
The fact that $s_i=s_{i, 1}$ is a running maxima of $(t_1, \ldots, t_r)$ means that $s_i > t_1, t_2, \ldots, t_M$, and hence
$$
n_{M+1}-n_M = N^{s_i} \gg N^{t_1} + \cdots N^{t_M} = n_M.
$$
In particular, $n_{M+1}-n_M \approx n_{M+1}$. Therefore in \eqref{eq:expan11}, at the cost of a small error in $L^2$, we can replace the factor $q_{n_{M+1}-n_M}(z_{M+1}-z_M)$ by $q_{n_{M+1}}(z_{M+1})$, since the dominant contribution to \eqref{eq:expan11} comes from
$\Vert z_M\Vert =O(\sqrt{n_M}) \ll \sqrt{n_{M+1}}$.

Now observe that the replacement $q_{n_{M+1}-n_M}(z_{M+1}-z_M) \rightsquigarrow q_{n_{M+1}}(z_{M+1})$ in \eqref{eq:expan11} at each running maxima of the sequence $n_i-n_{i-1}=N^{t_i}$, $1\leq i\leq r$, leads to the factorisation
\begin{equation}\label{eq:expan12}
Z_N^{\beta_N}(0) \approx 1 + \sum_{k=1}^\infty \sum_{0< s_1 <\cdots < s_k<1 \atop N^{s_i}\in \N} \prod_{i=1}^k \Xi_N^{\beta_N}(N^{s_i}),
\end{equation}
where for $s\in (0,1)$ with $n_1:=N^s\in \N$,
\begin{equation}\label{eq:expan13}
\Xi_N^{\beta_N}(n_1) := \sum_{r=1}^\infty \sum_{z_1, \ldots, z_r\in \Z^2 \atop n_i-n_{i-1}<n_1} q_{n_1}(z_1) \xi_N(n_1, z_1) \prod_{i=2}^r q_{n_i-n_{i-1}}(z_i-z_{i-1}) \xi_N(n_i, z_i).
\end{equation}
The key point is that for $s\in (0,1)$, $\Xi_N^{\beta_N}(N^s)$ depends essentially only on $\xi_N(n, \cdot)$ with $n\in [N^s, N^{s+o(1)}]$. In particular,
for $s<t$, $\Xi_N^{\beta_N}(N^s)$ and $\Xi_N^{\beta_N}(N^t)$ are essentially independent and approximate the increments of a time changed Brownian motion.
More precisely, it turns out that
\begin{equation}\label{eq:expan14}
M_N^{\beta_N}(t) := \sum_{1\leq n_1\leq N^t} \Xi_N^{\beta_N}(n_1) \xrightarrow[\,N\to\infty\,]{d} \int_0^t \frac{\hat\beta}{\sqrt{1-\hat\beta^2 s}} {\rm d}W(s), \qquad t\in [0,1],
\end{equation}
where $W$ is a standard Brownian motion. Just to check that the variances match, one can use the same calculations as in Section \ref{S:2momtrans} to show that
$\mathbb{V}{\rm ar}(M_N^{\beta_N}(t)) \to \int_0^t \frac{\hat\beta^2}{1-\hat\beta^2 s} {\rm d}s$. It follows that
$\Xi_N^{\beta_N}(N^t) \approx \frac{\hat\beta}{\sqrt{1-\hat\beta^2 t}} {\rm d}W(t)$, which implies that the r.h.s.\ of \eqref{eq:expan12} should converge to the Wick exponential $: \exp\big( \int_0^1 \frac{\hat\beta}{\sqrt{1-\hat\beta^2 t}} {\rm d}W(t)\big) :$, thus matching the conclusion of Theorem \ref{T:subcritical}.

The actual proof in \cite{CSZ17b} is fairly involved and gives more information. In particular, the terms in \eqref{eq:expan13} corresponding to different
$r\in\N$ can be shown to converge to increments of independent Brownian motions. In particular, the white noise ${\rm d}W$ in \eqref{eq:expan14}
is the sum of a sequence of independent white noises, which arise as the limits of random fields defined from degree $r$ polynomials of $\xi_N(\cdot, \cdot)$, one for each $r\in\N$. The proof is based on the {\em Fourth Moment Theorem} for Gaussian limits, see \cite[Theorem 4.2]{CSZ17b} and the references therein.
\qed

\bigskip

\noindent
{\bf Proof strategy 2 \cite[Theorems 3.5 \& 3.6]{CC22}.}
Instead of proving that $Z_N^{\beta_N}(0)$ converges
in distribution to a log-normal random variable, i.e.\ the exponential
of a Gaussian, see \eqref{eq:conve},
it is also possible to prove directly that $\log Z_N^{\beta_N}(0)$
converges to a Gaussian. To this end, we approximate $\log Z_N^{\beta_N}(0)$ in terms
of a random variable $X_N^{\mathrm{dom}}$ obtained by
restricting the sum in \eqref{eq:expan11} to a \emph{single dominated sequence}
$n_1 \ge \max\{n_2-n_1, \ldots, n_k - n_{k-1}\}$:
\begin{equation}\label{eq:XNdom}
	X_N^{\mathrm{dom}} :=
	\sum_{r=1}^\infty
	\sum_{\substack{n_0=0<n_1<\cdots <n_r<N \\
	n_1 \ge \max\{n_2-n_1, \ldots, n_k - n_{k-1}\}
	\\
	z_0:=0, z_1, \ldots, z_r\in \Z^2}}
	\prod_{i=1}^r q_{n_i-n_{i-1}}(z_i-z_{i-1}) \xi_N(n_i, z_i) \,.
\end{equation}
One can show that $X_N^{\mathrm{dom}}$ is asymptotically Gaussian,
more precisely $\bbE[(X_N^{\mathrm{dom}})^2] \to \sigma_{\hat\beta}^2$
and $X_N^{\mathrm{dom}} \to N(0, \sigma_{\hat\beta}^2)$ in distribution,
see \cite[Theorems 3.6]{CC22}. The proof is based on the classic
Feller-Lindeberg Central Limit Theorems for sums of independent random variables
(obtained by partitioning the sums in \eqref{eq:XNdom} in disjoint intervals).
A strengthened version of Theorem~\ref{T:subcritical} is then obtained by showing that
$\log Z_N^{\beta_N}(0) \approx X_N^{\mathrm{dom}} - \frac{1}{2}
\bbE[(X_N^{\mathrm{dom}})^2]$ in $L^2$ (not just in distribution),
which is shown in \cite[Theorems 3.5]{CC22} by exploiting a decomposition in terms of
dominated sequences similar to the one described in Proof strategy~1.

\bigskip

\noindent
{\bf Proof strategy 3.} In \cite{CD24}, yet another proof of Theorem \ref{T:subcritical} was presented.
The main idea was to approximate $Z_N^{\beta_N}(0)$ by a product of independent random variables.
More precisely,  $Z_N^{\beta_N}(0) \approx \prod_{l=1}^M Z_{N,k}^{\beta_N}$ where
$
Z_{N,k}^{\beta_N}:= \E\Big[ e^{\sum_{n=\tau_{k-1} +1}^{\tau_k} (\beta_N \omega(n,S_n) -\lambda(\beta_N)) }\Big]
$, with $\tau_0:=0$ and $\tau_k=\lceil N^{k/M} \rceil$ for $k=1,...,M$. Note that $(\tau_k)$ encode the same
exponential time scales as in Proof strategy 1, and $(Z_{N,k}^{\beta_N})_{1\leq k\leq M}$ are independent
random variables with mean $1$ and variance of order $1/M$ (cf.\ Corollary \ref{C:2mom}). Applying a classic Lindeberg Central Limit Theorem
to  $\log Z_N^{\beta_N}(0) \approx \sum_{k=1}^N \log Z_{N,k}^{\beta_N}$ then gives the the convergence of $\log Z_N^{\beta_N}(0)$ in distribution
to a Gaussian random variable with mean $-\sigma_{\hat\beta}^2/2$ and variance $\sigma^2_{\hat\beta}$. This requires Taylor expanding
$\log Z_{N,k}^{\beta_N}  = \log \big( 1+\big(Z_{N,k}^{\beta_N} -1\big)\big)$ and bounding the $(2+\epsilon)$ moments of  $Z_{N,k}^{\beta_N} -1$.
The latter is achieved by applying hypercontractivity estimates for polynomial chaos expansions from \cite{MOO10}, which bounds moments higher than $2$ by the second moment of a modified polynomial chaos expansion. In out setting, this amounts to increasing $\hat\beta$ in \eqref{eq:hbeta} to some $\hat\beta'>\hat\beta$. Thanks to the fact that we are considering the subcritical regime $\hat\beta<1$, we can choose $\hat\beta'<1$ to be subcritical as well, which ensures that we have the desired control on the second moments of the partition functions.

\subsection{Random field fluctuation}\label{S:Rfield}
In \cite[Theorem 2.13]{CSZ17b}, it was shown that in the subcritical regime $\hat\beta<1$, the fluctuation of the random field of partition functions
$(Z^{\beta_N}_{tN}(x\sqrt{N}))_{t\in \frac{1}{N}[0,1], x\in \frac{1}{\sqrt N} \Z^2}$ (and its mollified SHE analogue $u^\eps$ from \eqref{eq:mollSHE}) is asymptotically Gaussian and solves the Edwards-Wilkinson equation (or additive SHE).

\begin{theorem}[Edwards-Wilkinson fluctuation for DPM] \label{T:EWDPM}
For any test function $\phi \in C_c(\R^2)$ and any $t\in (0, 1]$, we have
\begin{align}\label{averaged}
\frac{1}{\beta_N} \int_{\R^2} \phi(x) \Big(Z^{\beta_N}_{tN}(\lfloor x\sqrt{N}\rfloor)-1\Big) {\rm d}x \xrightarrow[\,N\to\infty\,]{d} \int_{\R^2} \phi(x) v(t, x) {\rm d}x,
\end{align}
where $v(t, x)$ is a generalised Gaussian field that solves the $2d$  Edwards-Wilkinson equation
\begin{equation}\label{eq:EW}
\begin{aligned}
	\partial_t v(t,x) & =\frac{1}{2} \Delta v(t,x) + \sqrt{\frac{1}{1-\hat\beta^2}}\, \xi(t,x), \\
               v(0,x) & \equiv 0.
\end{aligned}
\end{equation}
\end{theorem}
\begin{proof}[Proof sketch]
Without loss of generality, assume $t=1$. Using the polynomial chaos expansion for $Z_N^{\beta_N}(z)$ from \eqref{eq:expan1}, we can rewrite the left-hand side of \eqref{averaged} as
\begin{align}
\Phi_N & := \frac{1}{\beta_N N} \sum_{r=1}^\infty
\sum_{z_0, z_1, \ldots, z_r\in \Z^2 \atop 0 =n_0 < n_1<\cdots <n_r<N}
\phi\big(\tfrac{z_0}{\sqrt{N}} \big)  \prod_{i=1}^r q_{n_i-n_{i-1}}(z_i-z_{i-1}) \xi_N(n_i, z_i) \notag \\
& = \underbrace{\frac{1}{N} \sum_{z_1\in \Z^2 \atop 0<n_1<N} \Big(\sum_{z_0\in \Z^2} \phi\big(\tfrac{z_0}{\sqrt{N}} \big)q_{n_1}(z_1-z_0)\Big)}_{q_{N, \phi}(n_1, z_1)} \Xi_N^{\beta_N}(n_1, z_1),
\label{avg_second-exp}
\end{align}
where
\begin{align*}
 \Xi_N^{\beta_N}(n_1, z_1) & :=
\frac{\xi_N(n_1, z_1)}{\beta_N} \Big(1+\sum_{k=2}^\infty \sum_{z_2, \ldots, z_k \in \Z^2 \atop n_1< n_2<\cdots <n_k<N}  \prod_{i=2}^k q_{n_i-n_{i-1}}(z_i-z_{i-1}) \xi_N(n_i, z_i)\Big)
\end{align*}
is simply the product of $\xi_N(n_1, z_1)/\beta_N$ with the point-to-plane partition function $Z^{\beta_N}_N(n_1, z_1)$ starting at $(n_1, z_1)$ and ending at
time $N$.

Note that the sum in \eqref{avg_second-exp} is an $L^2$-orthogonal decomposition, each $\Xi_N^{\beta_N}(n_1, z_1)$ has variance of order $1$,
and the averaging w.r.t.\ the kernel $q_{N, \phi}(n_1, z_1)$ ensures that the dominant contribution to $\Phi_N$ comes from $(n_1, z_1)$ of the order $(N, \sqrt{N})$. Now the key observation is that, by Corollary \ref{C:2mom} and the calculations therein, each $\Xi_N^{\beta_N}(n_1, z_1)$ depends only on $\xi_N(n, z)$ with $(n,z)$ in a sub-diffusive time-space window around $(n_1, z_1)$. Therefore $\Phi_N$ is an average of the field of random variables
$\Xi_N^{\beta_N}(n_1, z_1)$ with very local dependence. It is then not surprising that such an average has a Gaussian limit as $N\to\infty$.
This was first proved in \cite{CSZ17b} using the Fourth Moment Theorem for polynomial chaos expansions. See also \cite{CC22} for an alternative proof
using more classic tools from the proof of CLT for triangular arrays.
\end{proof}

\section{The $2d$ KPZ  in the subcritical regime}\label{sec:2dKPZ}
In this section we will discuss the fluctuation of the solution to the KPZ equation \eqref{eq:mollKPZ} in the subcritical regime $\hat\beta<1$.
The one-point statistics can be easily deduced from the log-normality of the SHE via the Cole-Hopf transformation
$h^\epsilon(t,x)= \log u^\epsilon (t,x)$ as
\begin{equation}\label{eq:hlim}
	h^\epsilon(t, x) \, \xrightarrow[\epsilon\downarrow 0]{d} \,
	\begin{cases}
	\sigma_{\hat\beta} W_1 -\tfrac{1}{2}\sigma_{\hat\beta}^2 & \text{ if } \hat\beta < 1  \\
	-\infty & \text{ if } \hat \beta\geq 1
	\end{cases}
	\qquad \text{with} \quad  \sigma^2_{\hat\beta}:= \log \tfrac{1}{1-\hat\beta^2} \,,
	\quad W_1  \sim N(0,1) \,,
\end{equation}
The random field fluctuations are more subtle as the operation of averaging does not commute with nonlinear functions such as $\log x$, i.e.,
$\int \phi(x) \log u^\epsilon (t,x) \,\dd x \neq  \log  \int \phi(x) u^\epsilon (t,x) \,\dd x$. Surprisingly, the limiting random field fluctuation for
$\log u^\epsilon$ turns out to be the same as the fluctuation for $u^\eps$ and its polymer partition function analogue in Theorem \ref{T:EWDPM}.
The theorem we obtained in \cite{CSZ20} (see also \cite{G20}) is the following:

\begin{theorem}[Edwards-Wilkinson fluctuation for subcritical $2d$ KPZ - \cite{CSZ20, G20}]\label{T:KPZ}
Let $h^\epsilon$ be the solution of the mollified KPZ equation \eqref{eq:mollKPZ} with initial condition $h^\epsilon(0,x)\equiv 0$ and
 $\beta_\epsilon= \hat\beta\sqrt{2\pi/\log \epsilon^{-1}}$,
 $\hat\beta\in (0,1)$. Denote
\begin{align}
	\label{eq:mainresult}
	\mathfrak{h}^\epsilon(t,x) := \frac{h^\epsilon(t, x) - \bbE[h^\epsilon(t,x)]}{\beta_\epsilon}
	= \frac{\sqrt{\log \epsilon^{-1}}}{\sqrt{2\pi} \, \hat\beta}
	\big( h^\epsilon(t, x) - \bbE[h^\epsilon(t,x)] \big) \,,
\end{align}
where the centering satisfies $\bbE[h^\epsilon(t,x)] = -\frac{1}{2}
\sigma_{\hat\beta}^2 + o(1)$ as $\epsilon \downarrow 0$, see \eqref{eq:hlim}.

For any $t > 0$ and $\phi \in C_c(\R^2)$, the following convergence in law holds:
\begin{equation}\label{eq:mainconv}
	\langle \mathfrak{h}^\epsilon(t,\cdot), \phi(\cdot)\rangle
	=  \int_{\R^2} \mathfrak{h}^\epsilon(t,x) \phi(x) {\rm d}x
	\  \xrightarrow[\epsilon \downarrow 0]{d} \ \langle
	v(t, 	\cdot), \phi(\cdot) \rangle,
\end{equation}
where $v(\cdot, \cdot)$ is the solution of the two-dimensional Edwards-Wilkinson equation in \eqref{eq:EW}.
\end{theorem}

The analogue of Theorem \ref{T:KPZ} for the log partition functions of the directed polymer is also given in  \cite{CSZ20}. This will be the theorem whose proof we will outline.
The proof of Theorem \ref{T:KPZ} follows exactly the same lines if instead of working with the polynomial chaos expansion of the partition function we work with the
Wiener chaos expansion of $u^\epsilon$ as in \eqref{eq:WienerKPZ}. We refer to \cite[Section 5]{CSZ20} for details.

\begin{theorem}
\label{log-polymer}
Assume the same setting as in Theorem \ref{T:EWDPM} for some $\hat\beta \in (0,1)$.
Apart from \eqref{eq:lambda}, suppose that the disorder $\omega$ also satisfies the concentration of measure property
\footnote{Condition \eqref{assD2} is satisfied if $\omega$ are bounded,  Gaussian,
or if they have a density $\exp(-V(\cdot) + U(\cdot))$, with $V$ uniformly strictly convex and
$U$ bounded. We refer to \cite{L01} for more details.
}:
\begin{equation}\label{assD2}
\begin{split}
	\exists \gamma>1, C_1, C_2 \in (0,\infty): \text{ for all $n\in\N$ and } f: \R^n \to \R
	\text{ convex and $1$-Lipschitz} \\
	\bbP \Big( \big| f(\omega_1, \ldots, \omega_N) - M_f \big|
	\ge t\Big) \le C_1\exp \bigg(-\frac{t^\gamma}{C_2}\bigg) \,, \qquad \qquad \quad
\end{split}
\end{equation}
where $M_f$ denotes a median of
$f(\omega_1, \ldots, \omega_N)$.
Denote\footnote{The scaling constant $\sqrt{4\pi}$ here is different than the constant $\sqrt{\pi}$ in \cite{CSZ20}
	due to the aperiodicity of the random walk we consider here.}
\begin{equation}\label{eq:H}
	\mathfrak{h}_N(t,x) := \frac{\log Z^{\beta_N}_{tN}(x \sqrt{N}) -
	\bbE[\log Z^{\beta_N}_{tN}]}{\beta_N} =
	\frac{\sqrt{\log N}}{\sqrt{4 \pi} \, \hat\beta}
	\big(\log Z^{\beta_N}_{tN}(x \sqrt{N}) -
	\bbE[\log Z^{\beta_N}_{tN}] \big)  ,
\end{equation}
For any $t>0$ and $\phi \in C_c(\R^2)$, the following convergence in law holds
with $v(\cdot, \cdot)$ as in \eqref{eq:EW}:
\begin{equation} \label{eq:Hconv}
	\langle \mathfrak{h}_N(t,\cdot), \phi(\cdot) \rangle
	= \int_{\R^2} \mathfrak{h}_N(t,x) \, \phi(x) \, \dd x
	\ \xrightarrow[N\to\infty]{d} \ \langle v(t,\cdot), \phi(\cdot) \rangle  \,,
\end{equation}
\end{theorem}
\begin{remark}
We note that the constant $4\pi$ that appears in \eqref{eq:H} is different than the constant $2\pi$ that appears in the
corresponding statement in \cite{CSZ20}. This discrepancy is due to the fact that in this review we work with aperiodic random
walks.
\end{remark}

\begin{proof}[Proof sketch for Theorem \ref{log-polymer}]
The main idea is to ``linearize'' $\log Z^{\beta_N}_{N}(x)$ by Taylor expansion, but not around $\E[Z^{\beta_N}_{N}(x)]=1$. Rather, we should
Taylor expand $\log Z^{\beta_N}_{N}(x)$ around $Z_{N,\gb_N}^A(x)$, which depends only on disorder in a small time-space window around the starting point of the polymer, but with $Z_{N,\gb_N}^A(x) \approx Z^{\beta_N}_{N}(x)$ by Corollary \ref{C:2mom}. More precisely, we define the time-space window
\begin{align}\label{def:setA}
	A_N^x:=\Big\{ (n,z)\in \N\times \Z^2 \,\, \colon \,\, n\leq N^{1-a_N}\, , \,
	|z-x|<N^{\tfrac{1}{2}-\tfrac{a_N}{4}}  \Big\} \,,
\end{align}
where
\begin{align}\label{def_an}
a_N=\frac{1}{(\log N)^{1-\gamma}} \qquad \text{with} \qquad \gamma\in (0,\gamma^*),
\end{align}
for some $\gamma^*>0$ depending only on $\hat\beta$, which makes the time window $\ll N$. The precise choice of $\gamma^*$ is more of a technical nature and we will not bother with it here; one can refer to \cite{CSZ20} for details. The spatial window of the set $A_N^x$ is slightly
superdiffusive compared to the time window to ensure that the random walk starting from $x$ at time $0$ will stay inside $A_N^x$ till the end of
the time window with high probability.

We define now the partition function $Z_{N,\beta}^A(x)$
which only uses disorder in $A_N^x$, i.e.,
\begin{equation}\label{eq:ZA0}
	Z_{N,\beta_N}^A(x) :=
	\E_x\big[ e^{H_{A_N^x}^{\beta_N}} \big] \,,
	\qquad \text{where} \qquad
	H_{A_N^x}^{\beta_N} := \sum_{(n,x) \in A_N^x} (\beta_N \omega_{n,x} -
	\lambda(\beta_N)) \ind_{\{S_n = x\}} \,.
\end{equation}

This allows us to decompose the original
partition function $Z_{N}^{\gb_N}(x)$ as follows:
\begin{align}\label{decomposition1}
Z_{N}^{\gb_N}(x) = Z_{N,\gb_N}^A(x) + \hat Z_{N,\gb_N}^A( x),
\end{align}
where $\hat Z_{N,\gb_N}^A(x)\stackrel{L^2}{\approx} 0$ is the ``remainder'', although it will be the source of the limiting fluctuation in \eqref{eq:Hconv}.
We now perform the Taylor expansion
\begin{align}\label{first_approx_heur0}
\log Z_{N}^{\gb_N}(x)
&= \log Z_{N,\gb_N}^A( x) + \log \Big(1+ \frac{\hat Z_{N,\gb_N}^A( x)}{Z_{N,\gb_N}^A( x)} \Big)
= \log Z_{N,\gb_N}^A( x) + \frac{\hat Z_{N,\gb_N}^A( x)}{Z_{N,\gb_N}^A( x)} +O_N(x),
\end{align}
where $O_N(x)$ is the error term. This approximation is quantified via the following estimates.
\medskip

{\bf Estimate 1.} For suitable test functions $\phi(\cdot)$, we have
\begin{equation} \label{eq:Rconv}
	\sqrt{\log N} \cdot \frac{1}{N} \sum_{x\in \Z^2}
	\big( O_N(x) - \bbE[O_N(x)] \big) \, \phi(\tfrac{x}{\sqrt{N}})
	\ \xrightarrow[N\to\infty]{L^2(\bbP)} \ 0 \,.
\end{equation}
The proof of this estimate uses a simple Taylor expansion estimate, which says that, essentially, the error term $O_N(z)$ is bounded by
$\big(\frac{\hat Z^A_{N,\beta_N}(x)}{Z^A_{N,\beta_N}(x)} \big)^2$.
In order to bound this error, we apply H\"older inequality to separate the numerator and denominator. This in turn requires bounds
on the moments of the partition function of order slightly higher than $2$, and bounds on arbitrarily negative moments.
In \cite{CSZ20}, hyper-contractivity was used to bound moments of order higher than $2$. Alternatively, one can also use more
recent techniques from \cite{LZ23, CZ23} to bound moments of all orders in the subcritical regime.

To bound the negative moments, we use Assumption \eqref{assD2} and concentration of measure estimates (\cite[Poposition 3.1]{CSZ20}):
\smallskip

{\bf Negative tails.} For any $\hat\beta \in (0,1)$, there exists
$c=c(\hat\beta) \in (0,\infty)$ with the following property:
for every $N\in\N$ and for every choice
of $\Lambda \subseteq \{1,\ldots, N\} \times \Z^2$, one has
\begin{align} \label{eq:boundf}
	\forall t \ge 0: \qquad
	\ \bbP(\log Z_{\Lambda,\gb_N}\leq -t) \leq
	c \, e^{- t^\gamma / c} \,,
\end{align}
where $\gamma > 1$ is the same exponent appearing in assumption \eqref{assD2}.

Next, we need:
\smallskip

{\bf Estimate 2.}
For $Z_{N,\gb_N}^A(\cdot)$ defined as in \eqref{eq:ZA0} and suitable test function $\phi$,
\begin{equation} \label{eq:Z'conv}
	\sqrt{\log N}
	\cdot \frac{1}{N} \sum_{x\in\Z^2}
	\big( \log Z_{N,\gb_N}^A(x) - \bbE[\log Z_{N,\gb_N}^A(x)] \big)
	\, \phi(\tfrac{x}{\sqrt{N}})
	\ \xrightarrow[N\to\infty]{L^2(\bbP)} \ 0 \,.
\end{equation}
The proof of this is a fairly simple $L^2(\bbP)$ estimate and uses the fact that $Z_{N,\gb_N}^A(\cdot)$ has very local dependence on the disorder, i.e., $Z_{N,\gb_N}^A(x)$ and $Z_{N,\gb_N}^A(y)$ are independent if $|x-y|\geq 2 N^{\tfrac{1}{2}-\tfrac{a_N}{4}} \ll \sqrt{N}$.
\medskip

{\bf Estimates 1} and {\bf  2} imply that the fluctuations of the field $\log Z_{N}^{\beta_N}(\cdot)$ is governed by the fluctuations of the field
$ \tfrac{\hat Z_{N,\gb_N}^A( \cdot)}{Z_{N,\gb_N}^A( \cdot)} $. The crucial point here is that the numerator {\it approximately factorises }
in a way that cancels the denominator, and what remains is a restricted polymer partition function for which we can apply a variant of Theorem \ref{T:EWDPM}.
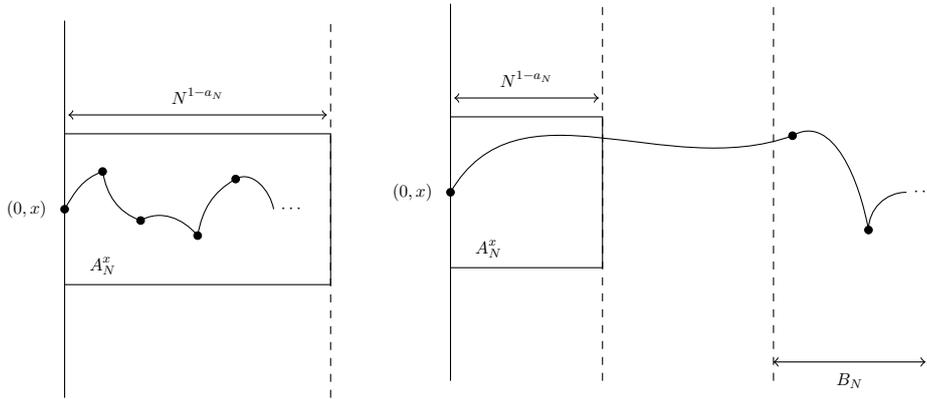
\begin{figure}[t]
\hskip -2cm
\begin{minipage}[b]{.33\linewidth}
\centering
\begin{tikzpicture}[scale=0.5]
\draw (0,-5)--(0,5); \draw (0,-2)--(7,-2)--(7,2)--(0,2);
\draw[dashed] (7,-5)--(7,5);
\node at (1,-1.5) {\scalebox{0.6}{$A_N^x$}};
\draw[<->] (0.1,2.5)--(6.9, 2.5);  \node at (3.5,3) {\scalebox{0.6}{$N^{1-a_N}$}};
\draw  [fill] (0, 0)  circle [radius=0.1]; \draw  [fill] (1, 1)  circle [radius=0.1]; \draw  [fill] (2, -0.3)  circle [radius=0.1]; \draw  [fill] (3.5, -0.7)  circle [radius=0.1];
\draw  [fill] (4.5, 0.8)  circle [radius=0.1]; \node at (6,0) {\scalebox{0.6}{{$\cdots$}}};
\draw (0,0) to [out=60,in=-160] (1,1) to [out=-80,in=160] (2,-0.3) to [out=30,in=130] (3.5,-0.7) to [out=80,in=-150] (4.5, 0.8) to [out=30,in=100] (5.5, 0);
\node at (-1,0) {\scalebox{0.6}{$(0,x)$}};
\end{tikzpicture}
\subcaption{Partition fuction $Z_{N,\beta_N}^A(x)$.
\label{figure:ZA}}
\end{minipage}
\,\,
\begin{minipage}[b]{.33\linewidth}
\centering
\begin{tikzpicture}[scale=0.5]
\draw (0,-5)--(0,5); \draw (0,-2)--(4,-2)--(4,2)--(0,2);
\draw[<->] (8.5,-4.5)--(12.5, -4.5);
\draw[<->] (0.1,2.5)--(3.9, 2.5);  \node at (2,3) {\scalebox{0.6}{$N^{1-a_N}$}};
\draw[dashed] (4,-5)--(4,5); \draw[dashed] (8.5,-5)--(8.5,5);
\node at (1,-1.5) {\scalebox{0.6}{$A_N^x$}};
\node at (10.5,-5) {\scalebox{0.6}{$B_N$}};
\node at (-1,0) {\scalebox{0.6}{$(0,x)$}}; \draw  [fill] (0, 0)  circle [radius=0.1];
 \draw  [fill] (9,1.5)  circle [radius=0.1];
\draw  [fill] (11, -1)  circle [radius=0.1];
\draw (0,0) to [out=60,in=-160]   (9,1.5) to [out=30, in=100] (11,-1)
to [out=90, in=180] (12,0); \node at (12.5,0) {\scalebox{0.6}{$\cdots$}};
\end{tikzpicture}
\subcaption{Partition function $Z_{N,\gb_N}^{B}(x).$\label{figure:ZB}}
\end{minipage}
\hskip -0.3cm
\caption{
The above figures depict the chaos expansions of $Z_{N,\gb_N}^A(x)$ and $Z_{N,\gb_N}^{B}(x)$. The disorder used by
 $Z_{N,\gb_N}^A(x)$ is restricted to the set $A_N^x$, while the disorder used by $Z_{N,\gb_N}^{B}(x)$ is restricted to $B_N$. \label{fig:multi_repAB}}
\end{figure}
With reference to Figure \ref{fig:multi_repAB}~(B), we define the partition function
\begin{equation*}
	Z_{N,\beta_N}^{B}(x) :=
	\E_x\big[ e^{H_{B_N}^{\beta_N}} \big] \,
	\quad \text{where} \quad
	H_{B_N}^{\beta_N} := \sum_{(n,x) \in B_{N}} (\beta_N \omega_{n,x} -
	\lambda(\beta_N)) \ind_{\{S_n = x\}} \,
\end{equation*}
where the set $B_N$ is as in Figure \ref{fig:multi_repAB}.
The crucial approximation that we establish in \cite{CSZ20} is
\begin{align}\label{main-appro}
\hat Z_{N,\gb_N}^{A}(x) \approx Z_{N,\gb_N}^A(x) \,
\big( Z_{N,\gb_N}^{B}(x) - 1 \big)  \,.
\end{align}
Here we use again the fact that the time-space window $A_N^x$ is of microscopic scale compared to the diffusive scale $(N, \sqrt{N})$.
The quantitative estimate related to this approximation is
\medskip

{\bf Estimate 3.}
For $Z_{N,\gb_N}^A(\cdot)$, $\hat Z_{N,\gb_N}^A(\cdot)$, $Z_{N,\gb_N}^{B}(\cdot)$
defined as above and for suitable test function $\phi $, we have
\begin{equation*}
	\sqrt{\log N}  \cdot \frac{1}{N} \sum_{x\in\Z^2}
	\bigg(
	\frac{\hat Z_{N,\gb_N}^A(x)}{Z_{N,\gb_N}^A(x)} \, - \,
	\big( Z_{N,\gb_N}^{B}(x) - 1 \big)
	\bigg)
	\, \phi(\tfrac{x}{\sqrt{N}})
	\ \xrightarrow[N\to\infty]{L^1(\bbP)} \ 0 \,.
\end{equation*}
\medskip

The above estimates reduce the study of the fluctuation of the field $\log Z_{N}^{\beta_N}(\cdot)$ to those of the field $Z_{N,\beta_N}^B(\cdot)$,
which is a restricted partition function and we can apply the same proof as for Theorem \ref{T:EWDPM} to establish the Edward-Wilkinson limit.
\end{proof}

\section{Related results outside the critical regime} \label{S:related}
In this section we will briefly list some further related work. The focus will be on the
critical dimension, but below the critical disorder strength. We will also discuss results in the quasi-critical regime and results in
higher dimensions.

\subsection{Other marginally relevant directed polymer models}
Theorems \ref{T:subcritical} and \ref{T:EWDPM} are special cases of more
general results in \cite{CSZ17b} concerning one-point and random field fluctuations of directed polymer models that are {\it marginally relevant}. Roughly speaking, a directed polymer model in the time-space domain $\N\times\Z^d$ is marginally relevant if its replica overlap
\begin{align*}
R_N:=\sum_{n=1}^N \sum_{x\in \Z^d} \P(S_n=S_n'),
\end{align*}
with $S$ and $S'$ being two independent random walks on $\Z^d$ with law $\P$, diverges as a slowly varying function in $N$.
We will explain in more detail what is meant by the marginal relevance of a disordered systems and the connection with singular
SPDE in Section \ref{S:disrel}.

As we have seen, for the $2d$ DPM, the divergence of the replica overlap $R_N$ is logarithmic.  Besides the $2d$ DPM, other examples include
directed polymer in $\N\times \Z^d$ with $d=1$, defined from a heavy-tailed random walk in the domain of attraction of the
Cauchy distribution, i.e.
$\P(S_n>x) \sim x^{-1}$,  or the disordered pinning model which can be regarded as a directed polymer on $\N\times \Z^d$
with $d=0$, defined from a renewal process on $\N\times \{0\}$ (marginally relevant if the renewals are the times when a simple symmetric random walk on $\Z$ returns to $0$).
It was shown in \cite{CSZ17b} that these marginally relevant directed polymer models all exhibit the
same asymptotic limit behavior as in Theorem~\ref{T:subcritical} 
for the $2d$ DPM, which shows a surprising universality across different dimensions.

Very recently, a log-normality result in the spirit of Theorem~\ref{T:subcritical}
was obtained in \cite{CCD25} for the directed polymer model in dimension
$d=2$ with a Gaussian environment which is independent in time but \emph{critically correlated in space},
with spatial correlations decaying as $|x-y|^{-2}$ (with possible logarithmic corrections).

\subsection{Stochastic Heat Equation}
The analogue of Theorem \ref{T:subcritical} for the solution $u^\eps$ of the mollified SHE \eqref{eq:mollSHE} was proved in \cite[Theorem 2.15]{CSZ17b}. More precisely,
\begin{theorem}[1-point statistics for the mollified SHE]\label{T:subcriticalSHE}
If $\beta_\eps=(\hat\beta+o(1)) \sqrt{\frac{2\pi}{\log \frac{1}{\epsilon}}}$ with $\hat\beta\in (0,1)$,
then for any $t>0$ and $x\in \R^2$, we have
\begin{align*}
u^\eps(t,x)
 \xrightarrow[\,\eps \to 0\,]{d}
	\begin{cases}
	\exp\bigg(\sigma_{\hat\beta} \, W_1 - \frac{\sigma_{\hat\beta}^2}{2} \bigg)
	&  \quad \text{if } \hat \beta < 1 \\
	0 & \quad \text{if } \hat\beta \ge 1
	\end{cases} \,,
\end{align*}
with $W_1$ a standard normal random variable and $\sigma^2_{\hat\beta}= \log \frac{1}{1-\hat\beta^2}$.
\end{theorem}
 Moreover, the analogue Theorem \ref{T:EWDPM}, i.e. Edwards-Wilkinson limit
for the field of the solution $u^\eps(t,\cdot)$ was also established in \cite{CSZ17b}:

\begin{theorem}[Edwards-Wilkinson fluctuation for the mollified SHE]\label{T:EWDSHE}
If $\beta_\eps=(\hat\beta+o(1)) \sqrt{\frac{2\pi}{\log \frac{1}{\epsilon}}}$ with $\hat\beta\in (0,1)$,
then for any test function $\phi\in C_c(\R^2)$, we have
\begin{align*}
\frac{1}{\beta_\epsilon} \int_{\R^2} \phi(x) \Big(u^\eps(t,x)-1\Big) {\rm d}x
\xrightarrow[\,\eps\to 0\,]{d} \int_{\R^2} \phi(x) v(t, x) {\rm d}x,
\end{align*}
where $v(t, x)$ is an in \eqref{eq:EW}.
\end{theorem}

Theorem 2.17 in \cite{CSZ17b} actually considered the fluctuation of the solution as a
space-time field instead of as a field in space at a given time, as stated above. The proof
of Theorem \ref{T:EWDSHE} is simpler. The joint limit of the solution at
finitely many space-time points at sub-diffusive scales is identified in \cite[Theorem 2.15]{CSZ17b}.

\subsection{Alternative approach to subcritical $2d$ KPZ}
In a subset of the subcritical regime for the $2d$ KPZ considered in Theorem \ref{T:KPZ}, tightness of
the fluctuation field was first established in \cite{CD20}, while \cite{G20} gave an alternative proof using
tools from stochastic analysis. The linearisation strategy is the same as in \cite{CSZ20}, but \cite{G20}
used the Clark-Ocone formula (see \cite{N06} for a reference and more background on Malliavin calculus)
to express
\begin{align*}
\log u^\eps(t,x) = \bbE[\,\log u^\eps \,] + \int_0^t \int_{\R^2}
\bbE\big[ D_{s,y} \log u^\eps(t,x) \,|\, \cF_s \,\big] \,\dd W(s,y),
\end{align*}
where $D_{s,y}$ is the Malliavin derivative and $\cF_s$ is the filtration generated by the white noise field
$(W(r,x) \colon r<s )$ up to time $s$. Using Malliavin calculus, the above is equal to
 \begin{align}\label{COformula}
 \log u^\eps(t,x) = \bbE[\,\log u^\eps \,] + \int_0^t \int_{\R^2}
  \bbE\Big[ \frac{D_{s,y}  u^\eps(t,x)}{u^\eps(t,x)} \,\Big|\, \cF_s \,\Big] \,\dd W(s,y) \,.
 \end{align}
The linearisation is based on the same observation as in Section \ref{sec:2dKPZ}, namely, $u^\eps(t, x)$ can be approximated
by a quantity that only depends on the noise in a tiny neighborhood of $(t,x)$, and hence
decouples from the numerator in the right-hand side of \eqref{COformula} as well as the $\sigma$-field $\cF_s$.
Once the linearisation is achieved, the Edwards-Wilkinson limit is obtained via the {\it second
order Poincar\'e inequality} \cite{Cha09, NPR09}, which states that given a random variable $X$
defined from a white noise $W$ on a domain $\Lambda$, the total variation distance
between the distribution of $X$ and that of a normal random variable $\zeta$ with matching
mean and variance can be bounded by
\begin{align}\label{2ndPoinc}
d_{TV}(X,\zeta) \leq C \,\bbE\big[ \|DX\|_H^4 \big]^{1/4} \, \bbE\big[ \|D^2 X\|_{\rm op}^4 \big]^{1/4},
\end{align}
where $D$ is again the Malliavin derivative, $H$ is the $L^2$ space of functions on $\Lambda$, and
$\|\cdot\|_{\rm op} $ is the operator norm of the Hessian operator $D^2X$. The fact that the second order Poincar\'e
inequality involves a fourth moment is why the result of \cite{G20} is restricted to a subset of the subcritical regime.

We also mention that, spatial averages of the solution of the mollified $2d$ KPZ equation on the mesoscopic scale
have been studied in \cite{Tao24b}. In \cite{NN23}, Theorem \ref{T:KPZ} has been generalised to the fluctuation field of $F(u^\eps(t, x))$ with general initial condition $u^\eps(0, \cdot)$ and for a suitable class of functions $F$ that includes $F(z)=\log z$.

\subsection{Nonlinear stochastic heat equations} A program to study semilinear stochastic heat equations at the critical dimension two
was initiated in \cite{DG22} and studied further in \cite{Tao24, DG23, DG24}. In \cite{DG22, Tao24}, they consider SPDEs of the form
\begin{align}\label{nonlinSHE}
\partial_t u^\epsilon = \frac{1}{2} \Delta u^\epsilon
+ \frac{\sqrt{2\pi}}{\sqrt{|\log\epsilon|}} \sigma\big( u^\epsilon \big) \,\xi^\epsilon, \qquad t>0, \, x\in \R^2,
\end{align}
where $\sigma\colon [0,\infty) \to [0,\infty)$ is a globally Lipschitz function with Lipschitz constant ${\rm Lip}(\sigma) <1$,
which is below the critical value $\hat\beta_c=1$ of the linear SHE.

In \cite{DG22}, the analogue of Theorem \ref{T:subcriticalSHE} on the limit of one-point distribution was established. More precisely,
if $u^\epsilon(0,\cdot) \equiv a \geq 0$, then the solution $u^\eps(t, x)$ of \eqref{nonlinSHE} at any $t>0$ and $x\in \R^2$ converges
in distribution to a random variable $\Xi_{a, Q}(Q)$ with $Q=2$, which solves the following forward-backward SDE (FBSDE):
\begin{equation}\label{FBSDE}
\begin{split}
\dd \Xi_{a, Q}(q)&= J(Q-q, \Xi_{a, Q}(q)) \, \dd B(q), \qquad q\in (0,Q],\\
\Xi_{a, Q}(0)&=a, \\
J(r,b) &= \frac{1}{\sqrt{2}} \bE\big[\sigma^2\big( \Xi_{b, r}(r) \big)\big]^{1/2},
\end{split}
\end{equation}
where $(B(q))_{q\geq 0}$ is a one-dimensional Brownian motion, and $a$ is the initial condition of the FBSDE $\Xi_{a, Q}(\cdot)$, with
terminal time $Q$. When $\sigma(u)=c u$ with $c<1$, $\Xi_{a, 2}(2)$ is log-normal and this recovers Theorem \ref{T:subcriticalSHE}.

The connection between \eqref{nonlinSHE} and the FBSDE \eqref{FBSDE} is that for $Q\in (0,2]$, $u^\eps(\eps^{2-Q}, x) \sim \Xi_{a, Q}(Q)$.
Furthermore, given $Q\in (0, 2]$, one has
$$
u^\eps(\eps^{2-Q} -\eps^{2-(Q-q)}, x)\approx \Xi_{a, Q}(q), \qquad q\in [0, Q],
$$
where
$\eps^{-(Q-q)}$ corresponds to the exponential time scale $N^{\frac{Q-q}{2}}$ in Corollary \ref{C:2mom} with $N=\eps^{-2}$. Such exponential
time scales for the nonlinear SHE are motivated by the same type of second moment calculations as in Corollary \ref{C:2mom}. The heuristic justification for \eqref{FBSDE} is the following. Denote $T=\eps^{2-Q}$ and $r=Q-q$. By second moment
calculations, a small error is incurred if we turn off the noise $\xi^\eps$ (i.e., evolve according to the heat equation) in the time interval $(T-\eps^{2-r}, T-\eps^{2-r+\delta})$ for a very small $\delta>0$. Smoothing by the heat flow on this time interval allows one to replace $u^\eps(T-\eps^{2-r+\delta}, \cdot)$ by a random constant distributed roughly as $\Xi_{a, Q}(q)=\Xi_{a, Q}(Q-r)$, which becomes the initial condition for the evolution of $u^\eps$ from time $T-\eps^{2-r+\delta}$ onward. Consider the evolution of $u^\eps$ till a later time $T-\eps^{2-r+\Delta}$ for some small $\Delta>\delta$. The time duration of this evolution is $T-\eps^{2-r+\Delta} - (T-\eps^{2-r+\delta})\approx \eps^{2-r}$, and therefore to compute the increment of the quadratic variation of $u^\eps(\cdot, x)$ in this time interval, it suffices to consider \eqref{nonlinSHE}  with initial condition $\Xi_{a, Q}(Q-r)$ and time duration $\eps^{2-r}$, which leads to the SDE in \eqref{FBSDE} with diffusion coefficient $J(r, \Xi_{a, Q}(Q-r))$.

The analogue of Theorem \ref{T:EWDSHE} on the Edwards-Wilkinson limit for the non-linear SHE \eqref{nonlinSHE} was established in \cite{Tao24}.

In \cite{DG23, DG24}, the results in \cite{DG22} and \cite{Tao24} were extended to $\sigma$ that do not necessarily satisfy the condition
Lip$(\sigma)<1$. For technical convenience, they considered a variant of \eqref{nonlinSHE} where smoothing is performed on $\sigma(u) \xi$ instead of on $\xi$, but the result is expected to be the same. The key idea of \cite{DG23} is that if Lip$(\sigma)=:\sqrt{A}>1$, then the result of \cite{DG22}
applies to $u^\eps$ up to a time scale of $\tilde \eps^2:=\eps^{2-\frac{2}{A}}$ with the corresponding spatial scale $\tilde \eps= \eps^{1-\frac{1}{A}}$.
Heuristically, if one performs an averaging on this space-time scale, then $\tilde \eps$ replaces $\eps$ as the basic spatial scale of smoothing, while $J(2/A, \cdot)=: \tilde\sigma(\cdot)$ replaces $\sigma(\cdot)$ as the renormalised diffusion function. If Lip$(\tilde \sigma)=:\sqrt{\tilde A}<\infty$,
then one can apply \cite{DG22} to this renormalised non-linear SHE up to a time scale of $\tilde \eps^{2-\frac{2}{\tilde A}}=(\eps^{(2-\frac{2}{A})})^{1-\frac{1}{\tilde A}}\gg \tilde \eps^2$. This procedure can then be iterated to find the maximal time scale $\eps^{2-Q}$ at which one can identify the distributional limit of $u^\eps(\eps^{2-Q}, x)$. The corresponding Edwards-Wilkinson limit was obtained in \cite{DG24}.

Similar results have been obtained very recently \cite{DHL25} for semilinear stochastic heat equations in dimension $d\geq 3$ with a noise that is white in time and has a critical power law decay in space. 

\subsection{Anisotropic KPZ and related models}
The general 2-dimensional KPZ is formally written as
\begin{align}\label{AKPZ}
\partial_t h= \frac{1}{2} \Delta h +\langle \nabla h , Q \nabla h \rangle + \xi, \qquad t>0 , \, x\in \R^2,
\end{align}
where $\langle\cdot, \cdot \rangle$ is the Euclidean inner product,
$\xi$ is the usual space-time white noise, and $Q$ is a $2\times 2$ matrix. When $Q=\left( \begin{array}{cc} 1 & 0 \\ 0 & \!\!-1 \end{array}\right)$,
it is called the purely anisotropic KPZ (AKPZ).
 This model emerged in the physics literature in \cite{W91}.
 It was shown in \cite{CET23b} that the correlation length of the interface
grows like $\sqrt{t (\log t)^{1/2}}$.

Although the 2-dimensional anistropic KPZ is also a critical singular SPDE, just as the 2-dimensional SHE and isotropic KPZ, the techniques are very different. We refer to the lecture notes by Cannizzaro and Toninelli \cite{CT24} for more details, where their techniques apply to a family of
critical and supercritical singular SPDEs, including the stochastic Navier Stokes equation with divergence-free driving noise, Burgers equation, diffusion in the curl of a $2d$ Gaussian Free Field, self-repelling Brownian polymers, etc.

The weak coupling limit of \eqref{AKPZ}, with the same noise strength as in Theorem \ref{T:KPZ} for the isotropic KPZ, has also been studied \cite{CES21, CET23b}. More precisely, in \cite{CES21, CET23b}, instead of mollifying the noise, all Fourier modes beyond order $1/\eps$ are
truncated, which is an alternative method of performing the so-called ultraviolet cutoff to smoothe things at spatial scale $\eps$. The strength of
non-linearity in the AKPZ (c.f.~\eqref{eq:genKPZ}) is then scaled down as $\lambda=\sqrt{\hat\lambda/ \log \epsilon^{-1}}$.
The random field fluctuation of the AKPZ in this weak coupling regime is shown to converge to the solution of the following
Edwards-Wilkinson equation:
\begin{align}\label{AKPZheat}
\partial_t h = \frac{\nu_{\rm eff}(\hat\lambda)}{2} \Delta h +
\sqrt{\nu_{\rm eff}(\hat\lambda)} \, \xi, \qquad \text{with}
\qquad \nu_{\rm eff}(\hat\lambda)= \sqrt{\frac{2}{\pi} \hat\lambda^2+1}.
\end{align}
We note that, in contrast to the isotropic $2d$ KPZ treated in Theorem \ref{T:KPZ}, the coefficient of $\Delta h$ is also modified.
But there is no phase transition in $\hat\lambda$ in the sense that the strength of the limiting noise $\xi$ is bounded for all $\hat \lambda>0$.
A key property needed in the analysis of the AKPZ is that it has an explicit Gaussian invariant measure (the Gaussian free field). This is a common
feature shared by models in this class. Again, we refer to the lecture notes \cite{CT24} for more details.

 \subsection{The quasi-critical regime}
 \label{sec:quasi-critical}
The fluctuations of the partition functions \emph{as a random field} have very different features
in the sub-critical and critical regimes:
\begin{itemize}
\item in the \emph{sub-critical regime} $\hat\beta < 1$,
fluctuations are \emph{asymptotically Gaussian}
after centering and \emph{rescaling by $\beta_N^{-1} \sim \hat\beta^{-1} \sqrt{\log N}$}, see Theorem~\ref{T:EWDPM}:
\begin{equation} \label{eq:EWcomp}
	\sqrt{\log N} \int_{\R^2} \phi(x) \Big(Z^{\beta_N}_{tN}(\lfloor x\sqrt{N}\rfloor)-1\Big) {\rm d}x
	\xrightarrow[\,N\to\infty\,]{d}
	\frac{\hat\beta}{\sqrt{1-\hat\beta^2}} \, \int_{\R^2} \phi(x) \hat v(t, x) {\rm d}x \,,
\end{equation}
where $\partial_t \hat v(t, x)$ solves a normalised version of the
Edwards-Wilkinson equation \eqref{eq:EW}, namely
$\partial_t \hat v(t, x) = \frac{1}{2} \Delta \hat v(t, x)
+ \xi(t,x) \, \hat v(t, x)$ with $\partial_t \hat v(0, \cdot) = 1$
(we moved the scaling factor $(1-\hat\beta^2)^{-1/2}$ into \eqref{eq:EWcomp}
for later comparison, see below);

\item
on the other hand, in the \emph{critical regime} $\hat\beta = 1$,
and actually in the whole critical window
$\hat\beta^2 = 1 + O(\frac{1}{\log N})$, see \eqref{eq:sigma},
a \emph{non-Gaussian} scaling limit emerges
with \emph{no need of centering and rescaling}, see Theorem~\ref{th:main0}:
\begin{equation*}
	\int_{\R^2} \phi(x) \, Z^{\beta_N}_{tN}(\lfloor x\sqrt{N}\rfloor) \, {\rm d}x
	\xrightarrow[\,N\to\infty\,]{d}
	\int_{\R^2} \varphi(x) \, \SHF_{0,t}^\theta(\dd x, \R^2) \,,
\end{equation*}
where $\SHF_{s,t}^\theta(\dd x, \dd y)$ is the critical $2d$ stochastic heat flow,
see Section~\ref{S:critical}.
\end{itemize}

To interpolate between these regimes, we can let
$\hat\beta=\hat\beta(N) \uparrow 1$ from below the critical window.
In view of \eqref{eq:sigma},
this amounts to taking $\beta =\beta_N$ in the
\emph{quasi-critical regime} defined by:
\begin{equation}\label{eq:quasicrit}
	\sigma_{\beta_N}^2 = \frac{1}{R_N} \bigg( 1 -
	\frac{\theta_N}{\log N}\bigg)
	\qquad \text{with} \quad 1 \ll \theta_N \ll \log N  \,.
\end{equation}
It is proved in \cite{CCR25}  that throughout this quasi-critical regime, the
partition function  as a random field
has \emph{Gaussian fluctuations},
after centering and \emph{rescaling by $\sqrt{\theta_N}$}, a  factor which interpolates between
$\sqrt{\log N}$ and~$1$:
\begin{equation}\label{eq:mainCCR}
	\sqrt{\theta_N} \int_{\R^2} \phi(x) \Big(Z^{\beta_N}_{tN}(\lfloor x\sqrt{N}\rfloor)-1\Big) {\rm d}x
	\xrightarrow[\,N\to\infty\,]{d} \,
	\int_{\R^2} \phi(x) \hat v(t, x) {\rm d}x \,,
\end{equation}
where $\hat v(t,x)$, as in \eqref{eq:EWcomp}, solves the Edwards-Wilkinson
equation.\footnote{To compare with \cite[Theorem~1.1]{CCR25}, note that the RHS
of \eqref{eq:mainCCR} is a Gaussian random variable with variance $\sigma_{t,\phi}^2 :=
\int_{\R^2 \times \R^2} \phi(x) \, K_t(x,x') \, \phi(x') \, \mathrm{d}x \, \mathrm{d}x'$
with $K_t(x,x') := \int_{0}^{t} \frac{1}{2u} \, e^{-\frac{|x-x'|^2}{2u}} \, \mathrm{d}u$.}
This shows that non-Gaussian behavior does not appear below the critical window.
Note that setting formally $\hat\beta^2 = 1 - \frac{\theta_N}{\log N}$ in \eqref{eq:EWcomp},
as prescribed by \eqref{eq:quasicrit}, we would obtain \eqref{eq:mainCCR}.

The strategy to prove \eqref{eq:mainCCR}, inspired by \cite{CC22}, is to
decompose the space-averaged partition function approximately into a sum of independent random variables
and then apply a Central Limit Theorem under a Lyapunov condition,
which requires estimating moments of the partition function
of order higher than two. A key difficulty is the lack of hypercontractivity, which holds
in the sub-critical regime $\hat\beta < 1$, but not in the quasi-critical regime \eqref{eq:quasicrit}
(nor in the critical regime $\hat\beta = 1$),
since the main contribution now comes from terms of diverging order in the chaos expansion.
To this end, novel moment estimates are derived in \cite{CC22} by
exploiting and extending the strategy in \cite{CSZ23a, LZ23}.

\subsection{Higher dimensions}
The Directed Polymer Model, the Stochastic Heat Equation and the KPZ equation have also
been studied in dimensions $d\geq 3$. A difference in the higher dimensional setting is that, contrary to the two-dimensional
case, the $L^2$ critical point (at which the second moment of the polymer partition function diverges) and the true critical point
(where the polymer undergoes a localisation-delocalisation phase transition) do not coincide.
This is a phenomenon known for some time in the directed polymer literature, see the references in \cite{C17, Z24}.
The analogue of Theorems \ref{T:EWDPM}, \ref{T:EWDSHE}, and \ref{T:KPZ} have been established in dimensions $d\geq 3$
below the $L^2$ critical point \cite{MU18, CCM24, CCM20, CNN22, LZ22}, which requires scaling of the noise strength in mollified
SHE and KPZ by a factor of $\epsilon^{\frac{d-2}{2}}$ (equivalent to no scaling in the inverse temperature $\beta$ in the directed
polymer model). Characterising the random field in $d\geq 3$ beyond the $L^2$ regime remains a challenging task. Some progress
has been made in the directed polymer setting by Junk \cite{J23, J25}, and the equivalence of the fluctuation for the mollified SHE
and the mollified KPZ has recently been established in \cite{JN25}, where the limiting field is believed to solve an analogue of the
Edwards-Wilkinson equation driven by a stable noise field.

\subsection{Space-time mollifications of noise}
So far, we have been discussing spatial mollifications of the noise, which is convenient because the noise is still white in time and hence
we can define the solution of the mollified SHE \eqref{eq:mollSHE} in terms of It\^o integral. However, it is also meaningful to consider a regularisation where the noise is mollified in both space and time.
This approach was considered in \cite{GRZ18} for the SHE in dimensions $d\geq 3$ and in \cite{K24} in $d=2$.
Even though there is still Edwards-Wilkinson limit in the weak disorder regime, there are additional difficulties because the
time mollification destroys the martingale structure of the process and the path measures are tilted by self-intersections local times.
We refer to \cite{GRZ18, K24} for details.

\begin{figure}[h]
\centering
\begin{minipage}[b]{.33\linewidth}
\centering
\includegraphics[height=70mm]{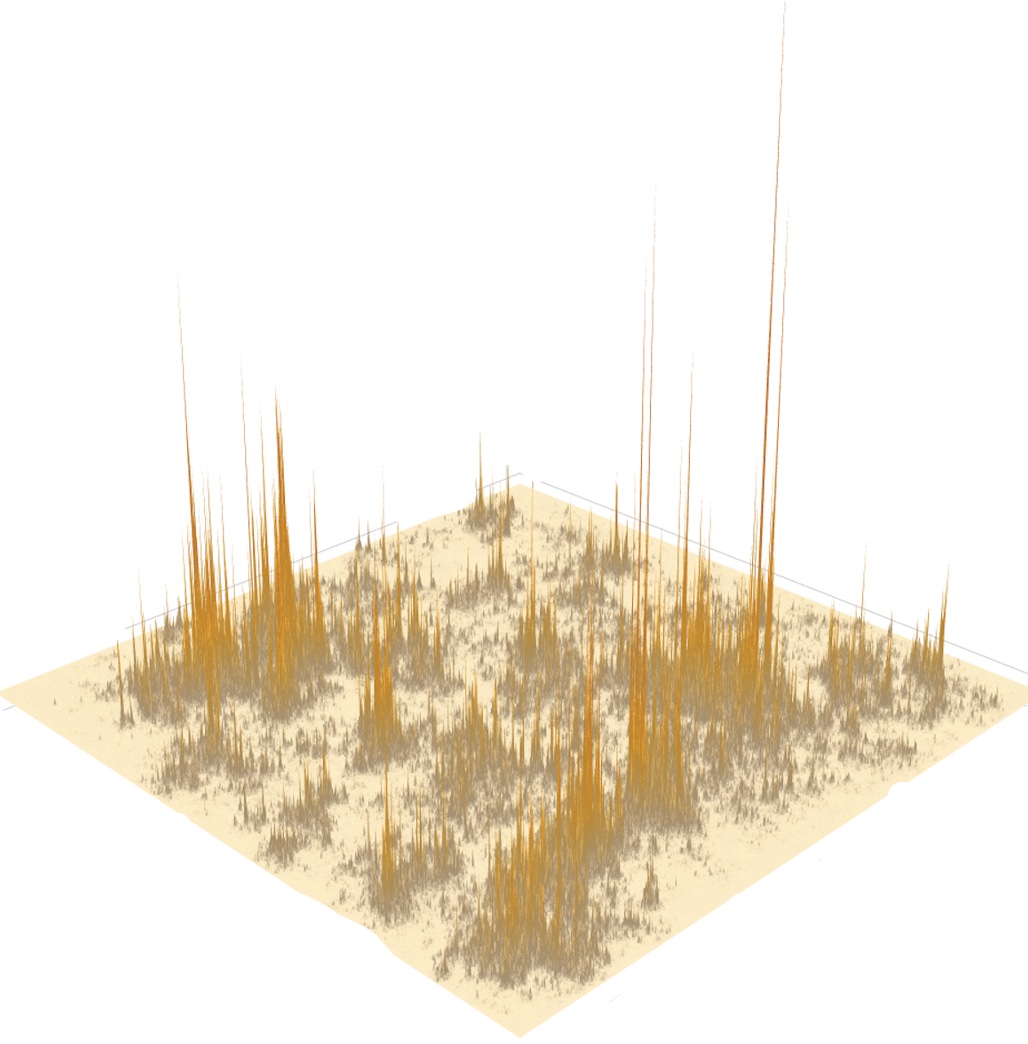}
\end{minipage}
\hfill
\begin{minipage}[b]{.43\linewidth}
\centering
\includegraphics[height=70mm]{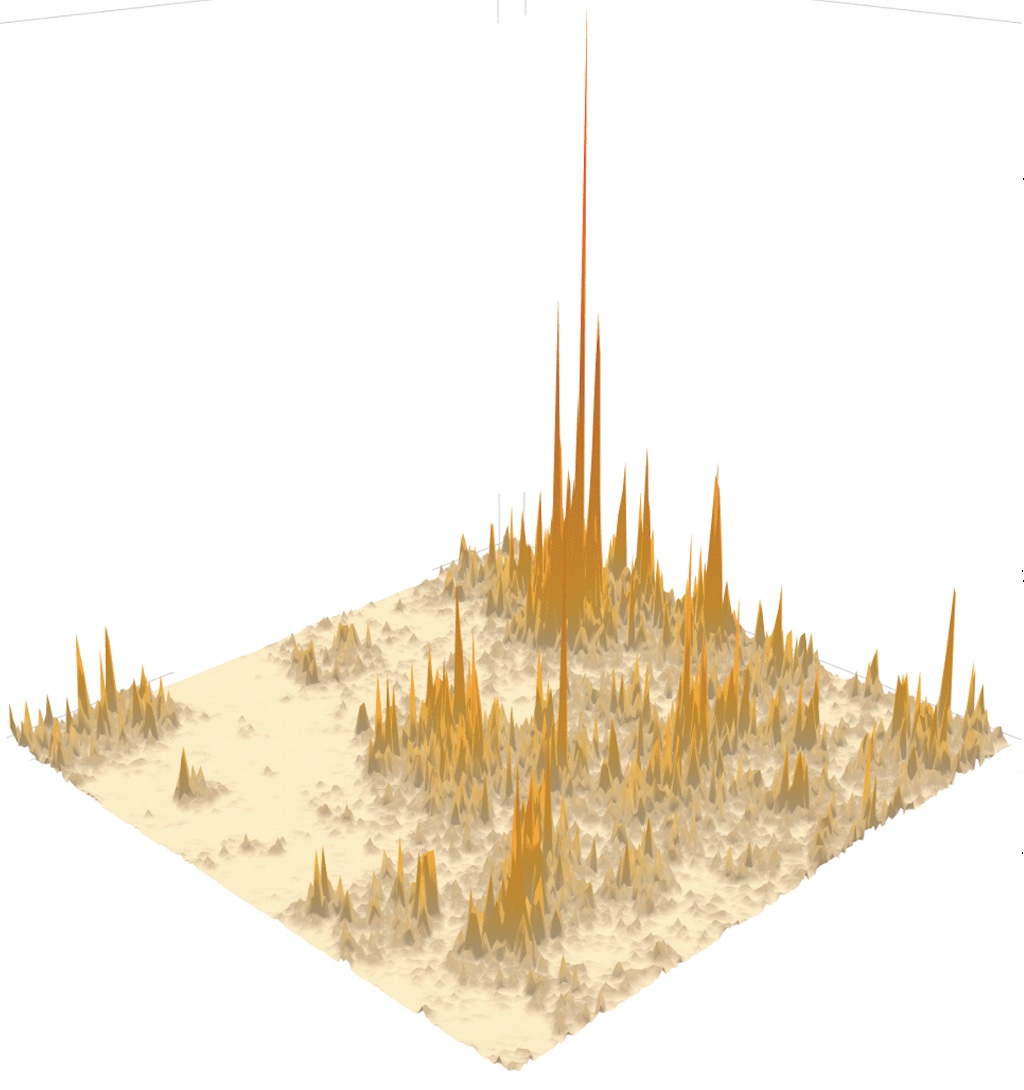}
\end{minipage} \\
\hskip -1cm
\begin{minipage}[b]{.33\linewidth}
\centering
\includegraphics[height=70mm]{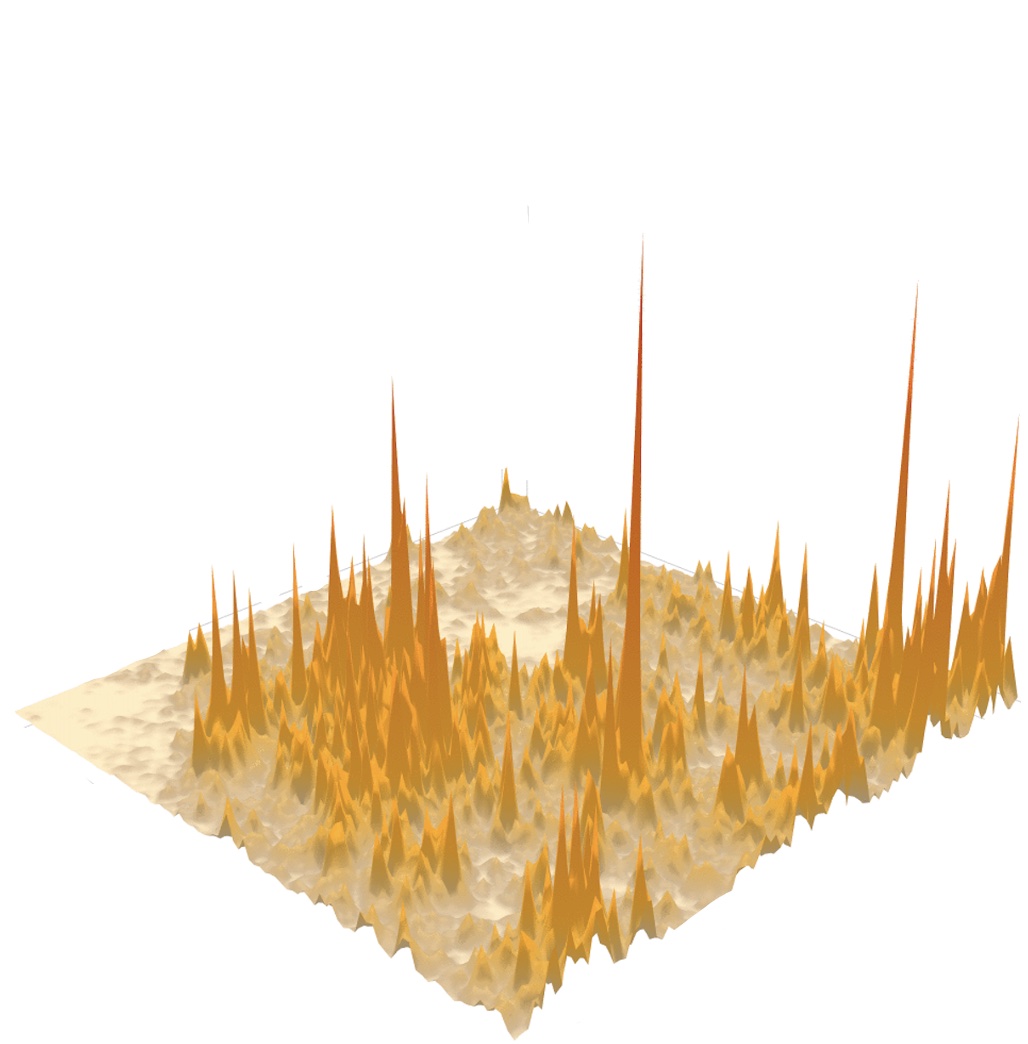}
\end{minipage}
\hskip 3cm
\begin{minipage}[b]{.33\linewidth}
\centering
\includegraphics[height=70mm]{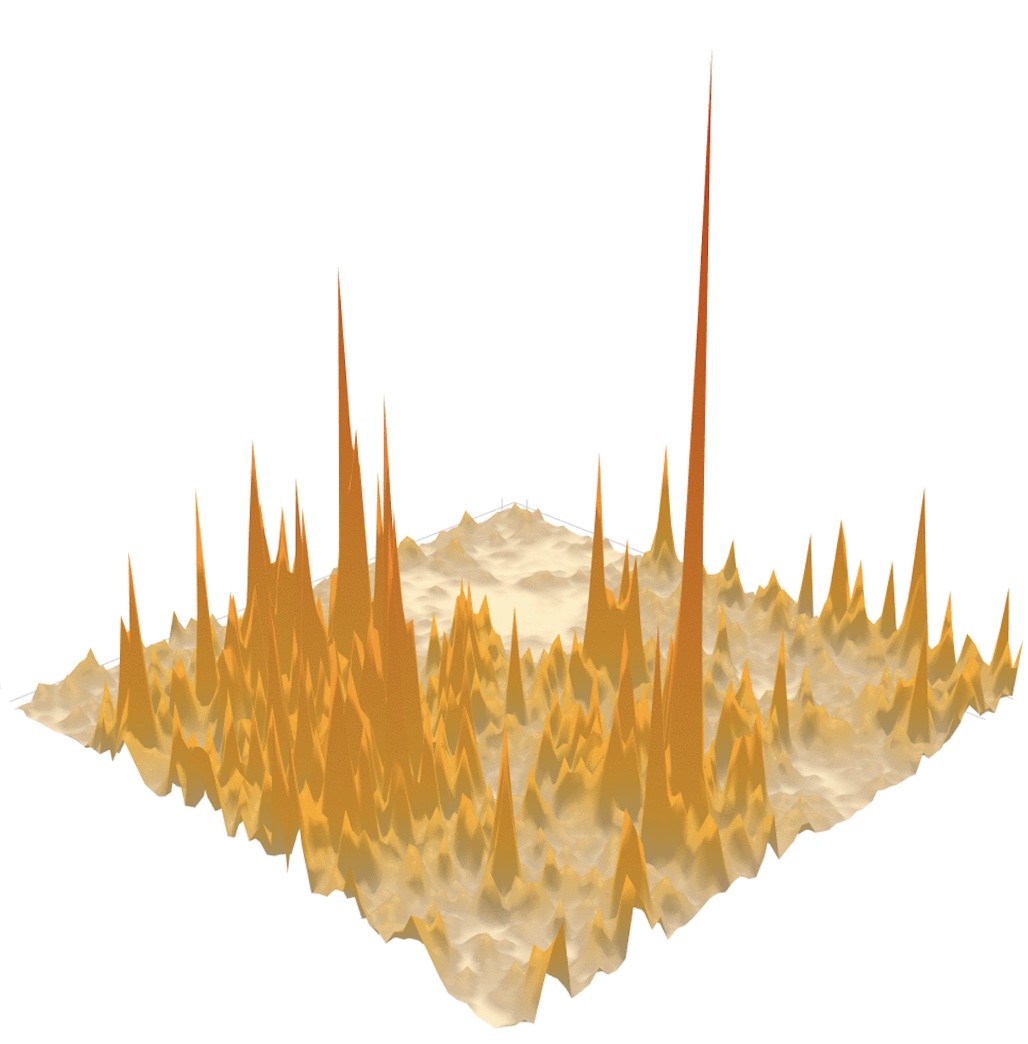}
\end{minipage}
\caption{The pictures above illustrate the critical $2d$ stochastic heat flow, $\SHF(\theta)$, and its scaling properties.
The top-right picture and the two bottom pictures are zoomed-in snapshots of the upper-left picture,
which is a simulation of $\SHF(\theta)$
with $\theta=0$.
The singularity of the critical $2d$ SHF w.r.t.\ the Lebesque measure can be seen from the upper-left picture.
The zoomed-in snapshots demonstrate the scaling covariance of $\SHF(\theta)$ stated in
\eqref{eq:scaling}. In particular, the zoomed-in fields appear to be smoother than the original field, since
\eqref{eq:scaling} shows that zooming in has the effect of decreasing the disorder strength $\theta$.}
	\label{fig:SHFzoom}
\end{figure}

\section{The Critical $2d$ SHF}\label{S:critical}

\subsection{Main result}
We assume throughout this section that $\beta_N$ and $\sigma_N^2$ have been chosen in the critical window as in \eqref{eq:sigma} for some $\theta\in \R$.

Recall the family of point-to-point partition functions $Z^{\beta_N}_{M,N}(x,y)$ from \eqref{eq:ZMN}. To keep the notation
simple, we will omit $\beta_N$ from $Z^{\beta_N}_{M,N}$. We regard $Z_{M, N}(\cdot, \cdot)$ as a random measure on $\Z^2\times \Z^2$. Scaling space and time diffusively and also the mass of the random measure, we define the family of rescaled random measures $\cZ_N := \big(\cZ_{N;\, s,t}(\dd x, \dd y)\big)_{0\leq s\leq t}$ by
\begin{equation} \label{eq:rescZmeas}
\iint_{\R^2\times \R^2} \varphi(x)\psi(y) \cZ_{N;\, s,t}(\dd x, \dd y) :=
    \frac{1}{N} \sum_{x, y\in \Z^2} \varphi\Big(\frac{x}{\sqrt{N}}\Big) \psi\Big(\frac{y}{\sqrt{N}}\Big) Z_{\lfloor Ns\rfloor, \lfloor Nt\rfloor}(x, y),
\end{equation}
where $\varphi, \psi\in C_c(\R^2)$ are test functions with compact support. We regard $\cZ_{N;\, s,t}(\dd x, \dd y)$ as a random variable taking values
in the space of Radon measures $\cM(\R^2 \times \R^2)$ on $\R^2\times \R^2$, equipped with the vague topology.

The main result from \cite{CSZ23a} is that in the critical window, the rescaled partition function $\cZ_{N;\, s,t}(\dd x, \dd y)$, regarded as a stochastic process in the space $\cM(\R^2 \times \R^2)$
indexed by $0 \le s \le t < \infty$, converges in law
to a unique limit, that we denote by $\SHF_{s,t}^\theta(\dd x , \dd y)$.

\begin{theorem}[Critical $2d$ SHF]\label{th:main0}
Let $\beta_N$ be chosen in the critical window \eqref{eq:sigma} for some $\theta\in \R$. As $N\to\infty$, the family of random measures
$\cZ_N = (\cZ_{N;\, s,t}(\dd x, \dd y))_{0\le s \le t}$ converges in finite dimensional distribution to a {\it unique} limit
\begin{equation*}
	\SHF(\theta) := (\SHF_{s,t}^\theta(\dd x , \dd y))_{0 \le s \le t <\infty} \,,
\end{equation*}
called the \emph{critical $2d$ stochastic heat flow}. The limit $\SHF(\theta)$ does not depend on the law of the disorder $\omega$ except for the assumptions in \eqref{eq:lambda}.
The first and second moments are
\begin{align}
	\label{eq:mean-SHF}
	\bbE[\SHF^\theta_{s,t}(\dd x , \dd y)]
	&= g_{t-s}(y-x) \, \dd x \, \dd y \,, \\
	\label{eq:var-SHF}
	\bbcov[\SHF^\theta_{s,t}(\dd x , \dd y),
	\SHF^\theta_{s,t}(\dd x' , \dd y')]
	&= K_{t-s}^\theta(x,x'; y, y') \, \dd x \, \dd y \, \dd x' \, \dd y' \,,
\end{align}
with the kernel $K^\theta_{t-s}$ defined by
\begin{equation}
\label{eq:m2-lim}
\begin{split}
	K_{t}^{\theta}(x,x'; y,y')
	&\,:=\, 4\pi \: g_{\frac{t}{2}}\big(\tfrac{y+y'}{2} - \tfrac{x+x'}{2}\big)
	\!\!\! \iint\limits_{0<a<b<t} \!\!\! g_{2a}(x'-x) \,
	G_\theta(b-a) \, g_{2(t-b)}(y'-y) \, \dd a \, \dd b \,,
\end{split}
\end{equation}
where $g$ is the heat kernel in $\R^2$ defined in \eqref{eq:gt}, and $G_\theta$ is defined in \eqref{dick-green1}-\eqref{dick-green2}.
\end{theorem}

\begin{remark}
The original definition of the Stochastic Heat Flow in
\cite[Theorem~1.1]{CSZ23a}, denoted by $\mathscr{Z}^\theta$,
was derived from directed polymers with simple symmetric random walk,
which is periodic and has variance $\frac{1}{2}$ (in each component).
The definition of $\SHF(\theta)$ here is slightly different,
because the random walk $S$ is aperiodic and has unit variance,
which matches the scaling limit of the solution of the mollified SHE.
To match the two definitions, it suffices to rescale space, time, and the measure
in order to match the first two moments:
compare \eqref{eq:mean-SHF}-\eqref{eq:m2-lim} with \cite[eqs.\ (1.13) and (3.56)]{CSZ23a}.
This leads to the identification
\begin{equation} \label{eq:SHF-Z}
	\SHF^\theta_{s,t}(\dd x , \dd y) \overset{d}{=}
	4 \, \mathscr{Z}^\theta_{s,t}(\dd \tfrac{x}{\sqrt{2}},
	\dd \tfrac{y}{\sqrt{2}})
	\overset{d}{=} 2 \, \mathscr{Z}^{\theta-\log 2}_{2s, 2t}(\dd x, \dd y) \,,
\end{equation}
where the second equality holds by scaling covariance, see
\cite[Theorem~1.2]{CSZ23a} and Theorem~\ref{th:main1} below,
and $\nu(\dd x, \dd y) = \mu(\dd \frac{x}{\sqrt{2}}, \dd \frac{y}{\sqrt{2}})$
denotes the measure $\nu(A) := \mu( \frac{1}{\sqrt{2}} A)$ for $A \subseteq \R^2 \times \R^2$.
Relation \eqref{eq:SHF-Z} corrects a misprint in \cite[Remark~1.5]{CSZ23a}.
\end{remark}

\begin{remark}
The covariance kernel $K_{t-s}^\theta(x,x'; y, y')$ was
first identified in \cite{BC98} $($see also \cite{CSZ19b}$)$
and is {\em logarithmically divergent} near the diagonals $x=x'$ or $y=y'$. One can also express
$K_{t-s}^\theta(x,x'; y, y')$ in terms of the Volterra function, see \cite{CM25, CM24} and the references therein.
\end{remark}

In what follows, we will sketch the proof Theorem \ref{th:main0} in \cite{CSZ23a}, which was carried out without any axiomatic characterisation of the limit SHF$(\theta)$. This difficulty was overcome by comparing with a family of coarse-grained models and then applying a Lindberg principle. Recently, Tsai \cite{T24} gave an axiomatic characterisation of SHF$(\theta)$ and then applied it to show that the solution of the mollified $2d$ SHE in the critical window also converges to SHF$(\theta)$. We will discuss this in more detail in Section \ref{S:Tsai}.

\subsection{Universality via a Lindeberg principle}\label{S:lind1}
To illustrate some of the proof ingredients for Theorem \ref{th:main0} in a simpler setting, we sketch here how one might apply a
Lindeberg principle to show that the limit of the random measures $\cZ_N$ in Theorem \ref{th:main0}, {\em if it
exists}, does not depend on the law of the disorder variables $\omega=(\omega(n,z))_{(n,z)\in\N\times\Z^2}$.

Let us consider the averaged point-to-plane partition function
\begin{equation}\label{eq:cZNphi}
\cZ_N(\varphi) := \frac{1}{N} \sum_{x\in \Z^2} \varphi\Big(\frac{x}{\sqrt{N}}\Big) Z^{\beta_N}_N(x), \qquad \varphi\in C_c(\R^2),
\end{equation}
which equals $\iint \varphi(x)\cZ_{N; 0,1}({\rm d}x, {\rm d}y)$ in the notation introduced in \eqref{eq:rescZmeas}.

By \eqref{eq:expan1}, $\cZ_N(\varphi)$ admits the polynomial chaos expansion
\begin{equation}\label{eq:expanphi}
\cZ_N(\varphi) = \overline{\varphi}_N + \frac{1}{N} \sum_{r=1}^\infty \sum_{ z_0, z_1, \ldots, z_r\in \Z^2 \atop n_0:=0<n_1<\cdots <n_r<N}
\varphi\Big(\frac{z_0}{\sqrt N}\Big)  \prod_{i=1}^r q_{n_i-n_{i-1}}(z_i-z_{i-1}) \xi_N(n_i, z_i),
\end{equation}
where
\begin{equation}
\overline{\varphi}_N := \frac{1}{N} \sum_{z_0\in \Z^2} \varphi\Big(\frac{z_0}{\sqrt N}\Big).
\end{equation}
Note that $\cZ_N(\varphi)$ is a function of the i.i.d.\ random variables $\xi_N:=(\xi_N(n,x))_{n,x\in \N\times \Z^2}$.

A Lindeberg principle is said to hold when the law of a function $\Psi(\zeta)$ of a family of random variables $\zeta=(\zeta_i)_{i\in \bbT}$ does not change much if $\zeta$ is replaced by another family of random variables $\eta=(\eta_i)_{i\in \bbT}$ with matching first few moments. Lindeberg principles provide powerful tools to prove universality, the simplest instance being the universality of the Gaussian distribution from the central limit theorem. The usual formulation such as in \cite{Cha06} requires the family of random variables to be independent (or exchangeable), and $\Psi$ needs to have bounded first three partial derivatives. This is not satisfied when $\Psi$ is a multilinear polynomial, whose derivatives are unbounded. This case was treated in  \cite{R79, MOO10} (see also \cite{CSZ17a}). Although the results in these references are not directly applicable to $\cZ_N(\varphi)$
for reasons that we will explain shortly, we can still apply the general approach. We first recall the Lindeberg principle from \cite{MOO10}.
\medskip

\noindent
{\bf Lindeberg principle \cite{MOO10}.} Consider a function $\Psi(x)$ defined as a sum of mutilinear polynomials in the variables $x=(x_i)_{i\in \bbT}$ by
\begin{equation}\label{eq:CPsi}
	\Psi(x) := \sum_{I \subset \bbT, |I|<\infty} \psi(I) \prod_{i\in I} x_i,
\end{equation}
where $\psi$ is a real-valued function defined on finite subsets of the index set $\bbT$.

Let $\zeta=(\zeta_i)_{i\in \bbT}$ and $\eta=(\eta_i)_{i\in \bbT}$ be two i.i.d.\ family with mean $0$, variance $1$, and $\bbE[|\zeta_i|^3] \vee \bbE[|\eta_i|^3]<\infty$. Then we note that
$$
\bbE[\Psi(\zeta)]=\bbE[\Psi(\eta)] = \psi(\emptyset) \quad \mbox{and} \quad \bbV{\rm ar}(\Psi(\zeta))=\bbV{\rm ar}(\Psi(\eta)) =  \sum_{I \subset \bbT, 1\leq |I|<\infty} \psi(I)^2.
$$
Without loss of generality, we assume that $\Psi$ has been normalised with $\bbV{\rm ar}[\Psi(\zeta)^2]=1$. To identify conditions that ensure $\Psi(\zeta)$ and $\Psi(\eta)$ are close in distribution, we consider
$$
\bbE[f(\Psi(\eta))] - \bbE[f(\Psi(\zeta))] \qquad \mbox{for } \quad f\in C^3_b(\R),
$$
i.e., bounded continuous test functions with bounded first three derivatives.

Assume for simplicity that $\bbT$ is a finite set with elements labelled by $1, \ldots, |\bbT|$. Then we replace $\zeta_i$ by $\eta_i$ one variable at a time and for each $0\leq k\leq |\bbT|$, define the $k$-th interpolating family by $\zeta^{(k)}$ with $\zeta^{(k)}_i=\eta_i$ for $1\leq i\leq k$ and $\zeta^{(k)}_i=\zeta_i$ for $i>k$. Note that $\zeta^{(0)}=\zeta$ and $\zeta^{(|\bbT|)}=\eta$. We can then write the telescopic sum
\begin{align}\label{eq:teles1}
\bbE[f(\Psi(\eta))] - \bbE[f(\Psi(\zeta))] = \sum_{k=1}^{|\bbT|} \bbE[f(\Psi(\zeta^{(k)})) - f(\Psi(\zeta^{(k-1)}))].
\end{align}
Denote $g_{k, \zeta, \eta}(x):= f(\Psi(\eta_1, \ldots, \eta_{k-1}, x, \zeta_{k+1}, \ldots, \zeta_{|\bbT|}))$. By Taylor expansion,
\begin{align*}
f(\Psi(\zeta^{(k)})) - f(\Psi(\zeta^{(k-1)})) & = g_{k, \zeta, \eta}(\eta_k) - g_{k, \zeta, \eta}(\zeta_k) \\
& = g'_{k, \zeta, \eta}(0) (\eta_k-\zeta_k) + \frac{1}{2} g''_{k, \zeta, \eta}(0) (\eta_k^2-\zeta_k^2) + R(\eta_k) - R(\zeta_k),
\end{align*}
where $R(x)$ is the remainder in the Taylor expansion of $g_{k, \zeta, \eta}(x)$ and satisfies
$$
|R(x)| \leq \Vert g'''_{k, \zeta, \eta}\Vert_\infty |x|^3 \leq \Vert f'''\Vert_\infty |\partial_k \Psi(\zeta^{(k)})|^3 |x|^3.
$$
Here $\partial_k \Psi$ denotes the partial derivative w.r.t.\ $x_k$ and does not depend on $x_k$ --
note that because $\Psi$ is multilinear, $\partial_k^2 \Psi=0$.
Substituting the above calculations into
\eqref{eq:teles1} and using the fact that $(\zeta_k, \eta_k)$ are independent of $(\zeta_i, \eta_i)_{i\neq k}$, we find that the first and second
order terms in the Taylor expansion cancel out, and
\begin{equation} \label{eq:lind1}
\begin{aligned}
\big| \bbE[f(\Psi(\eta))] - \bbE[f(\Psi(\zeta))] \big| & \leq \Vert f'''\Vert_\infty \sum_{k=1}^{|\bbT|} \bbE[|\partial_k \Psi(\zeta^{(k)})|^3] \, (\bbE[|\zeta_k|^3] +\bbE[|\eta_k|^3]) \\
& \leq \Vert f'''\Vert_\infty \max_{k} \big(\bbE[|\zeta_k|^3] + \bbE[|\eta_k|^3]\big) \sum_{k=1}^{|\bbT|} \bbE\big[\big|\partial_k \Psi(\zeta^{(k)})\big|^3\big].
\end{aligned}
\end{equation}
This is the point where the analysis in \cite{MOO10} is no longer adequate when considering the critical case.

Note that for each $k$, $\partial_k \Psi(\zeta)$ (we replaced $\zeta^{(k)}$ by $\zeta$ for simplicity) also admits a polynomial chaos expansion in the variables $\zeta$. In \cite{MOO10},
$\bbE\big[\big|\partial_k \Psi(\zeta)\big|^3\big]$ is bounded using hyper-contractivity (see also \cite[Theorem B.1]{CSZ20}), which
bounds higher than $2$ moments by the second moment with an additional weight $c_3^{k-1}$
on each chaos component of order $k$. Therefore,
$$
\bbE\big[\big|\partial_k \Psi(\zeta)\big|^3\big] \leq \Big(\sum_{I\subset \bbT:\, k\in I} c_3^{2(|I|-1)} \psi(I)^2 \Big)^{\frac{3}{2}}
\leq c_3^{3 {\rm deg}(\Psi)} \Big(\sum_{I\subset \bbT:\, k\in I} \psi(I)^2 \Big)^{\frac{3}{2}} = c_3^{3 {\rm deg}(\Psi)} {\rm Inf}_k(\Psi)^{\frac{3}{2}},
$$
where $c_3>1$ depends only on the law of $\zeta$, ${\rm deg}(\Psi) := \max\{|I|: \psi(I)>0\}$ is the degree of $\Psi$, and
$$
{\rm Inf}_k(\Psi) := \bbE[(\partial_k \Psi(\zeta))^2]  = \sum_{I \subset \bbT: \, k\in I} \psi(I)^2
$$
is called the influence of $\zeta_k$ on $\Psi(\zeta)$ in \cite{MOO10}. We can then bound the r.h.s.\ of \eqref{eq:lind1} by
\begin{align}
\big| \bbE[f(\Psi(\eta))] - \bbE[f(\Psi(\zeta))] \big| & \leq C_{f, \zeta, \eta} c_3^{3 {\rm deg}(\Psi)} \big(\max_k {\rm Inf}_k(\Psi)\big)^{\frac{1}{2}}
\sum_{k=1}^{|\bbT|} \sum_{I \subset \bbT: \, k\in I} \psi(I)^2 \notag \\
& = C_{f, \zeta, \eta} c_3^{3 {\rm deg}(\Psi)} \big(\max_k {\rm Inf}_k(\Psi)\big)^{\frac{1}{2}} \sum_{I \subset \bbT: |I|\geq 1} |I| \psi(I)^2 \notag \\
& \leq C_{f, \zeta, \eta} c_3^{3 {\rm deg}(\Psi)} {\rm deg}(\Psi)  \big(\max_k {\rm Inf}_k(\Psi)\big)^{\frac{1}{2}} \bbV{\rm ar}(\Psi(\zeta)). \label{eq:lind2}
\end{align}
When we consider a sequence $\Psi_N$ with uniformly bounded degree and variance,  the r.h.s. of
\eqref{eq:lind2} tends to $0$ once we show that the maximum influence $\max_k {\rm Inf}_k(\Psi_N) \to 0$ as $N\to\infty$.

The argument above fails when applied to the averaged partition function $\cZ_N(\varphi)$ in the critical window, because we will see in Section \ref{S:2ndmom} that in the polynomial chaos expansion of $\cZ_N(\varphi)$ in \eqref{eq:expanphi}, the dominant contribution comes from terms of degree $A\log N$ where $A>0$ can be arbitrarily large. Therefore in \eqref{eq:lind2}, $c_3^{3 {\rm deg}(\Psi)} \approx c_3^{3A\log N}$ can overwhelm $\max_k {\rm Inf}_k(\Psi_N)$, and we will need sharper bounds than provided by hyper-contractivity.
\medskip

\noindent
{\bf Applying Lindeberg to $\cZ_N(\varphi)$.} Our starting point is \eqref{eq:lind1}, which was actually derived only using the assumption that
for each $i\in \bbT$, $\eta_k$ and $\zeta_k$ have matching first two moments. Recall the i.i.d.\ family $\xi_N=(\xi_N(n, z))_{(n,z)\in \N\times\Z^2}$
from \eqref{eq:xi}, and denote
$$
\Psi_N(\xi_N) := \cZ_N(\varphi) =
\overline{\varphi}_N + \frac{1}{N} \sum_{r=1}^\infty  \!\!\!\!\!\! \sum_{z_0, z_1, \ldots, z_r\in \Z^2 \atop n_0:=0<n_1<\cdots <n_r<N}
\!\!\!\!\!\! \varphi\Big(\frac{z_0}{\sqrt N}\Big)  \prod_{i=1}^r q_{n_i-n_{i-1}}(z_i-z_{i-1}) \xi_N(n_i, z_i).
$$
To show that the distribution of $\Psi_N(\xi_N)$ hardly changes if we replace $\xi_N$ by another family $\zeta_N$ with matching first two moments
(and fourth moment of the same order), we apply the bound in \eqref{eq:lind1}, where the index set is now $\bbT=\N\times \Z^2$. For each $(n,z)\in \bbT$, we observe the factorisation
\begin{equation}
\partial_{(n,z)} \Psi_N(\xi_N) = \frac{1}{N} Z_{0,n}(\varphi_N, z) Z_{n, N}(z, \ind),
\end{equation}
where $\ind(z)\equiv 1$, $\varphi_N(z) := \varphi(z/\sqrt{N})$ for $z\in\Z^2$, and for any $\psi:\Z^2 \to \R$ and $m<n$,
$$
Z_{m, n}(\psi, z) = \sum_{x\in \Z^2} \psi(x) Z_{m, n}(x, z) \qquad \mbox{and} \qquad Z_{m, n}(z, \psi) = \sum_{y\in \Z^2}  Z_{m, n}(z, y) \psi(y)
$$
are the point-to-plane partition functions with the free endpoint weighted by the function $\psi$ (recall the point-to-point partition function $Z_{m, n}(x,y)$ from \eqref{eq:expan2}).

Note that $Z_{0,n}(\varphi_N, z)$ and $Z_{n, N}(z, \ind)$ are independent, while $Z_{0,n}(\varphi_N, z)$ has the same distribution as $Z_{0,n}(z, \varphi_N)$ and $Z_{n, N}(z, \ind)$ has the same distribution as $Z_{0, N-n}(z, \ind)$. Therefore we have
\begin{align}
& \big| \bbE[f(\Psi_N(\xi_N))] - \bbE[f(\Psi_N(\zeta_N))] \big| \notag \\
\leq\ &  \Vert f'''\Vert_\infty \max_{(n,z)} \big(\bbE[|\xi_N(n,z)|^3] + \bbE[|\zeta_N(n,z)|^3]\big) \!\!\!\!\!\! \sum_{(n,z)\in\N\times\Z^2} \!\!\!\!\!\!  \bbE\big[\big|\partial_{(n,z)} \Psi(\zeta_N^{(n, z)})\big|^3\big]  \label{eq:lind3} \\
\leq\ &  \frac{\Vert f'''\Vert_\infty}{N^3} \max_{(n,z)} \big(\bbE[|\xi_N(n,z)|^3] + \bbE[|\zeta_N(n,z)|^3]\big) \!\!\!\!\!\! \sum_{(n,z)\in\N\times\Z^2} \!\!\!\!\!\! \bbE\big[|Z_{0,n}(z, \varphi_N)|^3\big] \bbE\big[|Z_{0, N-n}(z, \ind)|^3\big], \notag
\end{align}
where $\zeta_N^{(n, z)}$ is a family that interpolates between $\xi_N$ and $\zeta_N$ as in $\zeta^{(k)}$ in \eqref{eq:teles1}.

As will be discussed in Section \ref{S:highmom}, the key bound is that for any $\delta>0$, we have
\begin{align} \label{eq:lind4}
\bbE\big[Z_{0, N-n}(z, \ind)^4\big] \leq C N^{\delta} \qquad \mbox{and} \qquad
\bbE\big[Z_{0,n}(z, \varphi_N)^4\big] \leq  C e^{-\frac{c|z|}{\sqrt N}} N^\delta,
\end{align}
where the exponential factor $e^{-\frac{c|z|}{\sqrt N}}$ comes from the assumption that $\varphi_N=\varphi(\cdot/\sqrt{N})$ is supported on a subset of
$\Z^2$ within distance $C\sqrt{N}$ of the origin. Substituting these bounds into \eqref{eq:lind3} and using the fact that
\begin{equation}\label{eq:xiN3}
\bbE[|\xi_N(n,z)|^3]\leq \bbE[\xi_N(n,z)^4]^{3/4} \leq C \sigma_N^3 \leq \frac{C}{(\log N)^3},
\end{equation}
which follows easily from the definition of $\xi_N$ in \eqref{eq:xi}, this gives an upper bound of order $\frac{\sigma_N^3}{N^3} N^{2+2\delta}$, with the factor $\frac{1}{N^3}$ coming from the normalisation in \eqref{eq:lind3}.
This upper bound tends to $0$ as $N\to\infty$ if we choose $\delta<1/2$.

Although this argument does not address whether $\cZ_N(\varphi)$ converges in distribution to a unique limit, which is the heart of Theorem
\ref{th:main0}, it does bring out several key ingredients in the proof of Theorem \ref{th:main0}, namely, the Lindeberg principle that led to \eqref{eq:lind3}, and the necessary moment bounds \eqref{eq:lind4} for the partition function. We will see more complex variants of these when we sketch the proof of Theorem~\ref{th:main0} in Section \ref{S:CG}.

\subsection{Proof outline of Theorem~\ref{th:main0}} \label{S:sketch}
We now sketch the proof of Theorem~\ref{th:main0} -- the main result of \cite{CSZ23a}. The key difficulty is the lack of a characterisation of the limit, i.e., the
critical $2d$ stochastic heat flow (SHF). As will be discussed in Section \ref{S:highmom}, earlier results \cite{CSZ19b, GQT21, Che24}
established the convergence of all positive integer moments of the averaged point-to-point partition functions
\begin{equation}\label{eq:cZNphipsi}
\cZ^{\beta_N}_N(\varphi, \psi) := \frac{1}{N} \sum_{x\in \Z^2} \varphi\Big(\frac{x}{\sqrt{N}}\Big) Z^{\beta_N}_{0, N}(x, y) \psi\Big(\frac{x}{\sqrt{N}}\Big), \qquad \varphi\in C_c(\R^2), \psi\in C_b(\R^2),
\end{equation}
which equals $\iint \varphi(x)\psi(y)\cZ_{N; 0,1}({\rm d}x, {\rm d}y)$ in the notation introduced in \eqref{eq:rescZmeas}. However, as shown in \cite{CSZ23b}, the limiting moments $\lim_{N\to\infty} \bbE[\cZ^{\beta_N}_N(\varphi, \psi)^k]$ diverge too fast in $k\in\N$ to uniquely characterise the distributional limit of $\cZ^{\beta_N}_N(\varphi, \psi)$.

Our strategy is to show that the laws of $(\cZ^{\beta_N}_{N}(\varphi, \psi))_{N\in\N}$ form
a \emph{Cauchy sequence}, i.e.,
\begin{equation} \label{eq:close}
	\text{\emph{$\cZ_M^{\beta_M}(\varphi, \psi)$ and $\cZ_N^{\beta_N}(\varphi, \psi)$
	are close in distribution for large $M, N \in \N$}} \,.
\end{equation}
We took inspiration from the work of Kozma \cite{Koz07} on the convergence of the loop erased random walk on $\Z^3$, which also lacked
a characterisation of the scaling limit.

To establish \eqref{eq:close}, for each $\eps\in (0,1)$, we will define a \emph{coarse-grained partition function}
$\mathscr{Z}_{\epsilon}^{(\cg)}(\varphi, \psi | \Theta)$, which has a similar structure to $\cZ^{\beta_N}_N(\varphi, \psi)$ and admits a
polynomial chaos expansion w.r.t.\ a family of random variables $\Theta$. We will use $\mathscr{Z}_{\epsilon}^{(\cg)}$
as a bridge between $\cZ_M^{\beta_M}$ and $\cZ_N^{\beta_N}$.  More precisely, we first perform a coarse-grained approximation for
$\cZ_N^{\beta_N}(\varphi, \psi)$ on the time-space scale $(\eps N, \sqrt{\eps N})$ and show that
\begin{equation}\label{eq:ZcgL2}
\cZ^{\beta_N}_{N}(\varphi, \psi) \stackrel{L^2}{\approx} \mathscr{Z}_{\epsilon}^{(\cg)}(\varphi, \psi | \Theta_{N, \eps}),
\end{equation}
where $\Theta_{N, \eps}$ is a family of weakly dependent coarse-grained disorder variables that depend on $N$ and~$\eps$, and the approximation error
can be made arbitrarily small in $L^2$ by choosing $\eps$ small and $N$ large.

We then prove \eqref{eq:close} by showing that
\begin{equation}\label{eq:ZcgNM}
\mathscr{Z}_{\epsilon}^{(\cg)}(\varphi, \psi | \Theta_{N, \eps}) \stackrel{\rm dist}{\approx} \mathscr{Z}_{\epsilon}^{(\cg)}(\varphi, \psi | \Theta_{M, \eps}),
\end{equation}
which follows by applying a Lindeberg principle for sums of multilinear polynomials of a family of weakly dependent random variables $\Theta$.
Such a Lindeberg principle was developed in \cite[Appendix A]{CSZ23a}.

Similar to the discussions in Section \ref{S:lind1} where we tried to apply a Lindeberg principle directly to $\cZ_N^{\beta_N}(\varphi, \psi)$ as a function of the disorder variables $\xi_N(n,z)$, the key ingredients for the Lindeberg principle are uniform bounds on $\bbE[|\Theta_{N, \eps}|^3]$, where $\Theta_{N, \eps}$ turns out to be an averaged
partition function itself, and similar moment bounds as in \eqref{eq:lind4} for the point-to-plane versions of the
coarse-grained partition function $\mathscr{Z}_{\epsilon}^{(\cg)}(\varphi, \psi | \Theta_{N, \eps})$.

Here is a summary of the key proof ingredients:
\begin{itemize}
\item[{\bf A.}] \emph{Coarse-Graining}, which leads to \emph{coarse-grained partition functions}
with a similar structure to the original polymer partition functions;
\item[{\bf B.}] \emph{Time-Space Renewal Structure},
which gives a probabilistic representation for the second moment calculations and leads in the continuum limit
to the so-called {\it Dickman subordinator};
\item[{\bf C.}] \emph{Lindeberg Principle} for multilinear polynomials of {\em dependent} random
variables, which controls the effect of changing $\Theta$ in the coarse-grained partition function
$\mathscr{Z}_{\epsilon}^{(\cg)}(\varphi, \psi | \Theta)$;
\item[{\bf D.}] \emph{Functional Inequalities for Green's Functions} of multiple random walks on $\Z^2$,
which yield sharp higher moment bounds for both the original partition functions and the coarse-grained
partition functions.
\end{itemize}
\begin{remark}
The proof steps outlined above can be applied to models with different microscopic
details in the critical window, such as directed polymer models with different underlying random walks, or the solution of
the mollified $2d$  SHE \eqref{eq:mollSHE}. As long as space-time is scaled to produce
the same diffusion constant, they can all be compared to
the same coarse-grained model $\mathscr{Z}_{\epsilon}^{(\cg)}(\varphi, \psi | \Theta)$, just with different coarse-grained
variables $\Theta_{N, \eps}$. In particular, they should lead to the same limiting SHF $\mathscr{Z}^\theta$ in Theorem
\ref{th:main0} if their mean are normalised to the same value, the disorder strength is chosen
to produce matching values of $\lim_{N\to\infty} \bbE[\Theta_{N, \eps}^2]$ for each $\eps>0$, and there are uniform bounds on
$\bbE[|\Theta_{N, \eps}|^3]$. See Section \ref{S:CG} for more details.
\end{remark}

\section{Coarse graining and Lindeberg} \label{S:CG}

\subsection{Coarse graining}
We now explain how to define the coarse-grained partition function $\mathscr{Z}_{\epsilon}^{(\cg)}(\varphi, \psi | \Theta_{N, \eps})$,
see \eqref{eq:ZcgL2}, and the family of coarse-grained disorder variables $\Theta_{N, \eps}$. We will sketch the key ideas. The more
precise details can be found in \cite[Section 4]{CSZ23a}.

Given $\eps\in (0,1)$, we partition space-time into {\it mesoscopic} blocks indexed by $\sfi \in \N$ and $\sfa =(\sfa_1,\sfa_2) \in \Z^2$:
\begin{equation}\label{Bbox}
	\cB_{\epsilon N}(\sfi,\sfa)
	\ := \ \underbrace{((\sfi-1)\epsilon N, \sfi \epsilon N]}_{\cT_{\epsilon N}(\sfi)}
	\times \underbrace{((\sfa-(1,1))\sqrt{\epsilon N},
	\sfa\sqrt{\epsilon N}]}_{\cS_{\eps N}(\sfa)}
	 \ \, \cap \ \ \Z^3 \,,
\end{equation}
where
\begin{align*}
\big((\sfa-(1,1))\sqrt{\epsilon N}, \sfa\sqrt{\epsilon N}\,\big]
:=\big((\sfa_1-1)\sqrt{\epsilon N}, \sfa_2\sqrt{\epsilon N}\,\big] \times \big((\sfa_2-1)\sqrt{\epsilon N}, \sfa_2\sqrt{\epsilon N} \,\big].
\end{align*}

\begin{figure}
\begin{tikzpicture}[scale=0.4]]
\draw[ystep=1cm, xstep=2,gray,very thin, lightgray] (0,-2) grid (29, 8);
\draw (0, 5)  circle [radius=0.1];    circle [radius=0.08];
\draw[thick] (0,5) to [out=45,in=150] (4.2,6.5);
\draw[lightgray] (4, 6)--(6, 6)--(6, 7)--(4, 7)--(4, 6);
\draw [fill] (4.2, 6.5)  circle [radius=0.08];  \draw [fill] (5.8, 6.6)  circle [radius=0.08];
\draw[thick, densely dotted] (4.2, 6.5) to [out=20,in=180] (4.4,6.8) to [out=-30,in=120] (4.8,6.2) to [out=0, in=180] (5.8, 6.6);
\draw [fill] (10.2, 1.5)  circle [radius=0.08];   \draw [fill] (11.6,1.5) circle [radius=0.08];
\draw[thick, densely dotted] (10.2, 1.5) to [out=20,in=180] (10.4,1.8) to [out=-30,in=120] (10.8,1.2) to [out=0, in=180] (11.8, 1.6);
\draw[thick] (5.8, 6.6) to [out=25,in=110] (10.2, 1.5);
\draw[thick, densely dotted] (14.2, 0.5) to [out=20,in=180] (14.5,1.4) to [out=-30,in=120] (14.9,0.2) to [out=0, in=180] (15.2, 0.6)
to [out=0, in=180] (15.8, -0.5);
\draw[thick] (11.6,1.5) to [out=30,in=150] (14.2, 0.5);
 \draw [fill] (14.2, 0.5)  circle [radius=0.08];
 \draw [fill] (15.8, -0.6)  circle [radius=0.08];
 \draw[thick] (15.8, -0.6) to [out=70,in=160] (20.2, 4.5);
 \draw [fill] (20.2, 4.5)  circle [radius=0.08];
\draw [fill] (21.8, 4.6) circle [radius=0.08];
\draw[thick, densely dotted] (20.2, 4.5) to [out=20,in=180] (20.4,4.8) to [out=-30,in=120] (20.8,4.2) to [out=0, in=180] (21.8, 4.6);
\draw[thick] (21.8, 4.6) to [out=40,in=110] (26.5, 3.5);
\draw[thick, densely dotted] (26.5, 3.6) to [out=20,in=180] (26.8, 3.8) to [out=-30,in=120] (27, 3.2) to [out=0, in=180] (27.2, 3.6) to [out=0, in=180] (27.5, 3.8) to [out=0, in=180] (27.8, 3.2) ;
\draw[fill] (26.5, 3.6) circle [radius=0.08];  \draw[fill] (27.8, 3.2) circle [radius=0.08];
\draw[thick] (27.8, 3.2) to [out=60,in=180] (29,4.8);
\begin{pgfonlayer}{background}
 \fill[color=gray!10] (4,6) rectangle (6,7);
 \fill[color=gray!10] (10,1) rectangle (12,2);
 \fill[color=gray!10] (14,-1) rectangle (16,2);
 \fill[color=gray!10] (20,4) rectangle (22,5);
  \fill[color=gray!10] (26,3) rectangle (28,4);
 \end{pgfonlayer}
\end{tikzpicture}
\caption{Space-time is partitioned into mesoscopic boxes $\cB_{\epsilon N}(\sfi,\sfa)$ as in \eqref{Bbox}. Given $\Gamma=((n_1,z_1),...,(n_r, z_r))$,
the sequence of space-time points associated with a given term in the chaos expansion in \eqref{eq:Zppchaos},  the boxes containing the dotted lines are the ones that intersect $\Gamma$. The starting point $(d, x)\in \Gamma$ and the end point $(f, y)\in \Gamma$ of each dotted line are the points of entry and exit in $\Gamma$ for that mesoscopic time interval $\cT_{\epsilon N}(\sfi)$. Each dotted line contributes a factor $X_{d, f}(x, y)$ to the corresponding term
in the chaos expansion, while each solid line contributes a random walk transition kernel.
\label{CG-fig}}
\end{figure}
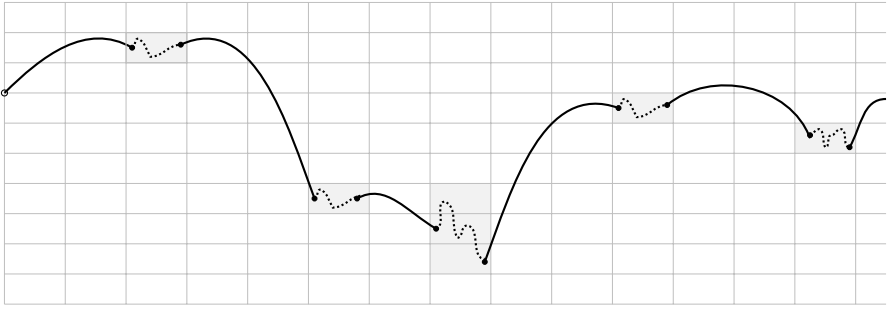

By \eqref{eq:cZNphipsi} and \eqref{eq:expan2}, we see that $\cZ^{\beta_N}_N(\varphi, \psi)$ admits a polynomial chaos expansion in the disorder variables $\xi_N$:

\begin{align}
	& \cZ^{\beta_N}_N(\varphi, \psi)   \, = \, q^N_{0,N}(\varphi,\psi)   \label{eq:Zppchaos} \\
	& \ +  \frac{1}{N} \sum_{r=1}^{\infty} \!\!\!\!\!\!\!\!
	\sum_{\substack{ z_0, \ldots, z_{r+1} \in \Z^2\\
	n_0:=0 < n_1 < \ldots < n_r < N}}\!\!\!\!\!\!\!\!\!\!\!\!
	\varphi_N(z_0)  \Bigg\{ \prod_{j=1}^r q_{n_j-n_{j-1}}(z_j-z_{j-1}) \xi_N(n_j,z_j) \Bigg\} q_{N-n_r}(z_{r+1}-z_r)\psi_N(z_{r+1}), \notag
\end{align}
where $\varphi_N(x):=\varphi(\frac{x}{\sqrt{N}})$, $\psi_N(x):=\psi(\frac{x}{\sqrt{N}})$, and
\begin{align*}
q^N_{m, n}(\varphi, \psi) & := \frac{1}{N} \sum_{x, y\in \Z^2} \varphi\Big(\frac{x}{\sqrt{N}}\Big) q_{n-m}(y-x) \psi\Big(\frac{y}{\sqrt{N}}\Big).
\end{align*}

Note that each term in \eqref{eq:Zppchaos} is associated with a sequence of space-time points
$$
\Gamma:=((n_1, z_1), \ldots, (n_r, z_r)),
$$
which contribute the product of disorder variables $\prod_{i=1}^r \xi_N(n_i, z_i)$. We will group the terms in \eqref{eq:Zppchaos} according to the set of mesoscopic blocks $\cB_{\epsilon N}(\sfi,\sfa)$ visited by points in $\Gamma$. More precisely,
we proceed as follows (see Figure \ref{CG-fig}):
\begin{itemize}
\item First, we decompose the sum in \eqref{eq:Zppchaos} according to the set of mesoscopic time intervals $\cT_{\epsilon N}(\sfi)$ visited by
points in $\Gamma=((n_1,z_1),...,(n_r, z_r))$.

\item
For each time interval $\cT_{\epsilon N}(\sfi)$ visited by points in $\Gamma$, let $(d,x)\in \Gamma$ be the point of entry into $\cT_{\epsilon N}(\sfi)$,
with $(d, x) \in \cB_{\epsilon N}(\sfi,\sfa)$ for some $\sfa\in \Z^2$. Similarly, let $(f, y) \in \Gamma$ be the point of exit from $\cT_{\epsilon N}(\sfi)$,
with $(f, y) \in \cB_{\epsilon N}(\sfi,\sfa')$ for some $\sfa'\in \Z^2$, which could be different from $\sfa$ (see Figure \ref{CG-fig}).
\end{itemize}

Given the above coarse-grained decomposition, we can rewrite
\begin{align}
	& \cZ^{\beta_N}_N(\varphi,\psi)
	 = \, q_{0,N}^N(\varphi,\psi) \notag \\
	 & \ +  \ \frac{1}{N}
	\sum_{v,w \in \Z^2} \varphi_N(v) \
	\sum_{k=1}^{\infty} \ \sum_{0 < \sfi_1 < \ldots < \sfi_k \le \frac{1}{\epsilon}} \
	\sum_{\substack{d_1 \le f_1 \,\in\,
	\cT_{\epsilon N}(\sfi_1), \, \ldots\, , \, d_k \le f_k \,\in\, \cT_{\epsilon N}(\sfi_k) \\
	x_1,\, y_1, \, \ldots \,, x_k, \, y_k \,\in\, \Z^2}} \label{eq:Zpoly3} \\
	& \  q_{d_1}(x_1-v)
	 \, X_{d_1,f_1}(x_1,y_1) \, \Bigg\{ \prod_{j=2}^k q_{d_j - f_{j-1}}(x_j - y_{j-1}) \,
	X_{d_j,f_j}(x_j,y_j) \Bigg\} \,
	q_{N-f_k}(w-y_k)  \,  \psi_N(w),\notag
\end{align}
where, recalling the point-to-point partition function from \eqref{eq:ZMN},
\begin{equation}\label{eq:Xdf}
	X_{d,f}(x,y) := \begin{cases}
	\xi_N(d, x) \ind_{\{x=y\}} & \text{if } d=f, \\
	\rule{0pt}{1.2em}\xi_N(d,x) \,
	Z^{\beta_N}_{d, f}(x,y) \, \xi_N(f,y) & \text{if } d<f.
	\end{cases}
\end{equation}

Ideally, we would like to use the local limit theorem \eqref{eq:llt} to replace the random walk transition kernels $q_{d_j - f_{j-1}}(x_j - y_{j-1})$
by a heat kernel
\begin{equation}\label{eq:qtog}
q_{d_j - f_{j-1}}(x_j - y_{j-1}) \rightsquigarrow g_{(\sfi_j-\sfi_{j-1})\eps N}((\sfa_j-\sfa_{j-1}')\sqrt{\eps N}) = \frac{1}{\eps N} g_{\sfi_j-\sfi_{j-1}}(\sfa_j-\sfa_{j-1}'),
\end{equation}
which depends only on the index of the mesocopic blocks $\cB_{\epsilon N}(\sfi_{j-1}, \sfa'_{j-1})\ni (f_{j-1}, y_{j-1})$ and $\cB_{\epsilon N}(\sfi_j,\sfa_j) \ni (d_j, x_j)$. Namely, approximate \eqref{eq:Zpoly3} by
\begin{align}\label{eq:Zcg-mock}
	  & \cZ_{N,\epsilon}^{({\rm mock}-\cg)}(\varphi, \psi | \Theta_{N, \eps})
	  =  g_1(\varphi, \psi) \\
	  & +  \epsilon \sum_{k=1}^{\lfloor \frac{1}{\eps} \rfloor}
	 \!\!\!\! \sum_{0<\sfi_1< \ldots < \sfi_k\leq \lfloor \frac{1}{\epsilon}\rfloor \atop
	\sfa'_0, \sfa_1, \sfa_1', \ldots, \sfa_k',	\sfa_{k+1} \in \Z^2} \!\!\!\!\!\!\!\!\!\!\!\!\!
	\varphi_\epsilon(\sfa_0')
	 \bigg( \, \prod_{j=1}^k
	g_{\sfi_j -\sfi_{j-1}} (\sfa_j -\sfa_{j-1}') \Theta_{N, \epsilon}(\sfi_j, \vec\sfa_j) \,\bigg)
	\cdot g_{\frac{1}{\epsilon}-\sfi_r} (\sfa_{r+1} -  \sfa'_r ) \psi_\epsilon(\sfa_{r+1}) , \notag
\end{align}
where $g_1(\varphi, \psi) := \int \varphi(x) g_1(y-x)\psi(y) {\rm d}x{\rm d}y$, and
\begin{equation}\label{aux-phi}
\begin{aligned}
&\varphi_\epsilon(\sfa)
:=  \frac{1}{\epsilon N} \!\!\sum_{v \in \cS_{\eps N}({\sfa})\cap \Z^2}
	\!\!\!\!\!\!\!\!\!\!\! \varphi_N(v) \,,
	\qquad
\psi_\epsilon(\sfa):=	
\, \frac{1}{\epsilon N}
	 \sum_{w \in \cS_{\epsilon N}({\sfa})\cap \Z^2}  \psi_N(w)  \\
	 \text{and} & \qquad
\Theta_{N, \epsilon}(\sfi, \vec\sfa)
:= \displaystyle \frac{1}{\epsilon N} \,
	\sum_{\substack{(d,x) \in \cB_{\epsilon N}(\sfi, \sfa)\\
	(f,y) \in \cB_{\epsilon N}(\sfi, \sfa')\\\text{with }d \le f}}
	X_{d,f}(x,y), \quad \text{for}\quad \vec \sfa:=(\sfa,\sfa')\in(\Z^2)^2.
\end{aligned}
\end{equation}
\begin{figure}
\begin{tikzpicture}[scale=0.4]
\draw[ystep=1cm, xstep=2,gray,very thin, lightgray] (0,-2) grid (29, 8);
\draw (0, 5)  circle [radius=0.1];    
\draw[] (0,5) to [out=45,in=150] (6,7); 
\draw [fill] (4.2, 6.5)  circle [radius=0.08];  \draw [fill] (5.8, 6.6)  circle [radius=0.08];
\draw[thick, densely dotted] (4.2, 6.5) to [out=20,in=180] (4.4,6.8) to [out=-30,in=120] (4.8,6.2) to [out=0, in=180] (5.8, 6.6);
\draw [fill] (10.2, 1.5)  circle [radius=0.08];   \draw [fill] (11.6,1.5) circle [radius=0.08];
\draw[thick, densely dotted] (10.2, 1.5) to [out=20,in=180] (10.4,1.8) to [out=-30,in=120] (10.8,1.2) to [out=0, in=180] (11.8, 1.6);
\draw[] (6.1, 7.1) to [out=25,in=110] (12, 2); 
\draw[thick, densely dotted] (14.2, 0.5) to [out=20,in=180] (14.5,1.4) to [out=-30,in=120] (14.9,0.2) to [out=0, in=180] (15.2, 0.6)
to [out=0, in=180] (15.8, -0.5);
\draw[] (12,2) to [out=30,in=150] (16, 2);
 \draw [fill] (14.2, 0.5)  circle [radius=0.08];
 \draw [fill] (15.8, -0.6)  circle [radius=0.08];
 \draw[] (16, 2) to [out=70,in=160] (22, 5); 
 \draw [fill] (20.2, 4.5)  circle [radius=0.08];
\draw [fill] (21.8, 4.6) circle [radius=0.08];
\draw[thick, densely dotted] (20.2, 4.5) to [out=20,in=180] (20.4,4.8) to [out=-30,in=120] (20.8,4.2) to [out=0, in=180] (21.8, 4.6);
\draw[] (22, 5) to [out=40,in=110] (28, 4); 
\draw[thick, densely dotted] (26.5, 3.6) to [out=20,in=180] (26.8, 3.8) to [out=-30,in=120] (27, 3.2) to [out=0, in=180] (27.2, 3.6) to [out=0, in=180] (27.5, 3.8) to [out=0, in=180] (27.8, 3.2) ;
\draw[fill] (26.5, 3.6) circle [radius=0.08];  \draw[fill] (27.8, 3.2) circle [radius=0.08];
\draw[] (28, 4) to [out=60,in=180] (29,4.8); 
\begin{pgfonlayer}{background}
 \fill[color=gray!10] (4,6) rectangle (6,7);
 \fill[color=gray!10] (10,1) rectangle (12,2);
 \fill[color=gray!10] (14,-1) rectangle (16,2);
 \fill[color=gray!10] (20,4) rectangle (22,5);
  \fill[color=gray!10] (26,3) rectangle (28,4);
 \end{pgfonlayer}
\end{tikzpicture}
\caption{A simplified coarse-graining that defines
$\cZ_{N,\epsilon}^{({\rm mock}-\cg)}(\varphi, \psi | \Theta_{N, \eps})$. A coarse-grained disorder variable $\Theta_{N, \eps}(\sfi, \vec \sfa)$
is defined for each visited mesoscopic time interval $\cT_{\epsilon N}(\sfi)$. The random walk transition kernel
connecting $\cT_{\epsilon N}(\sfi)$ and $\cT_{\epsilon N}(\sfj)$, with $\sfi<\sfj$, is replaced by a heat kernel connecting
the corner of the mesoscopic spatial box of exit from $\cT_{\epsilon N}(\sfi)$ to the corner of the box of entry in
$\cT_{\epsilon N}(\sfj)$, represented here by a solid line. Such replacements can lead to poor approximations
if  $\sfj-\sfi<K_\eps$.
\label{CG-fig2}}
\end{figure}

However, the above approximation is not good enough (this is why we gave the name ``mock-cg'')
 because to have a negiligible error from the local
limit approximation, we require that:
\begin{itemize}
\item[(R1)] The mesoscopic time intervals  $\cT_{\epsilon N}(\sfi_j)$,
$j=1,...,k$, are well separated, that is, $\sfi_j-\sfi_{j-1}> K_\epsilon$ for some $K_\epsilon\in\N$ that diverges fast enough
as $\eps\downarrow 0$, so that replacing any point in $\cB_{\epsilon N}(\sfi_{j-1}, \sfa'_{j-1})$, resp.\ $\cB_{\epsilon N}(\sfi_j, \sfa_j)$,
by a single point induces a negligible error in the local limit theorem approximation. Furthermore, we need the cumulative
error in the local limit theorem approximations to be small.

\item[(R2)] The space-time coordinates of the transition kernels $g_{\sfi_j-\sfi_{j-1}}(\sfa_j-\sfa_{j-1}')$ obey diffusive scaling, that is,
$|\sfa_j-\sfa'_{j-1} |< M_\epsilon \sqrt{\sfi_j-\sfi_{j-1}}$ for some $M_\epsilon$ that diverges slow enough as $\eps\downarrow 0$ ($M_\eps =\log\log\frac{1}{\eps}$ suffices), so that the local limit theorem approximation induces a negligible error, while restricting the spatial
coordinates to such an essentially diffusive window captures the dominant contribution to the partition function.
\end{itemize}
With regards to the first requirement (R1), a suitable choice is $K_\epsilon:=(\log \tfrac{1}{\epsilon})^6$, which ensures that in
the chaos expansion \eqref{eq:Zppchaos}, configurations of $\Gamma=((n_1, z_1), \ldots, (n_r, z_r))$ that visit three mesoscopic time
intervals $\cT_{\epsilon N}(\sfi_j)$, $\cT_{\epsilon N}(\sfi_{j+1})$, $\cT_{\epsilon N}(\sfi_{j+2})$ with $\sfi_{j+1}-\sfi_{j}< K_\eps$ and
$\sfi_{j+2}-\sfi_{j+1}< K_\eps$ give a negligible contribution to the partition function. We can therefore restrict the chaos expansion
for $\cZ_N(\varphi,\psi)$ to such {\em no-triple configurations}. However, it can be shown by $L^2$ calculations
that there will always be non-negligible contributions from configurations that contain
$O(1)$ number of mesoscopic time indices $\sfi_j$ with $\sfi_{j+1}-\sfi_j<K_\eps$. This is why
$\cZ_{N,\epsilon}^{({\rm mock}-\cg)}(\varphi, \psi)$
is not a good approximation of $\cZ_N(\varphi,\psi)$, although it captures some key features of the coarse-grained partition function
that we need.

To obtain a good coarse-grained approximation of  $\cZ_N(\varphi,\psi)$, we proceed as follows. As discussed above, we may restrict
the chaos expansion of $\cZ_N(\varphi,\psi)$ to the {\em no-triple configurations}. Given a configuration $\Gamma=((n_1, z_1), \ldots, (n_r, z_r))$
in the chaos expansion \eqref{eq:Zppchaos}, we call a visited mesoscopic time interval $\cT_{\epsilon N}(\sfi_j)$ isolated if $\sfi_j-\sfi_{j-1}, \sfi_{j+1}-\sfi_{j}\geq K_\eps$. If $\sfa_j$ (resp.\ $\sfa_j'$) denotes the mesoscopic spatial box of entry (resp.\ exit) by $\Gamma$ in the time
interval $\cT_{\epsilon N}(\sfi_j)$, then we define an associated coarse-grained disorder variable $\Theta_{N, \eps}(\sfi_j, \vec \sfa_j)$
as in \eqref{aux-phi}, with
\begin{align}\label{eq:Theta1}
\Theta_{N, \eps}(\sfi, \vec \sfa) := \frac{1}{\epsilon N} \,
	\sum_{\substack{(d,x) \in \cB_{\epsilon N}(\sfi, \sfa)\\
	(f,y) \in \cB_{\epsilon N}(\sfi, \sfa')\\\text{with }d \le f }}
	X_{d,f}(x,y).
\end{align}
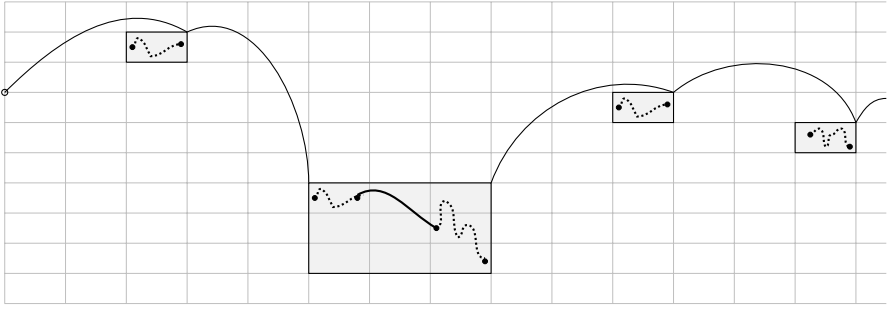
\begin{figure}
\begin{tikzpicture}[scale=0.4]
\draw[ystep=1cm, xstep=2,gray,very thin, lightgray] (0,-2) grid (29, 8);
\draw (0, 5)  circle [radius=0.1];
\draw[] (0,5) to [out=45,in=150] (6,7);
\draw[] (4, 6)--(6, 6)--(6, 7)--(4, 7)--(4, 6);
\draw[] (20, 4)--(20, 5)--(22, 5)--(22, 4)--(20, 4);
\draw[] (26, 3)--(26, 4)--(28, 4)--(28, 3)--(26, 3);
\draw [fill] (4.2, 6.5)  circle [radius=0.08];  \draw [fill] (5.8, 6.6)  circle [radius=0.08];
\draw[thick, densely dotted] (4.2, 6.5) to [out=20,in=180] (4.4,6.8) to [out=-30,in=120] (4.8,6.2) to [out=0, in=180] (5.8, 6.6);
\draw [fill] (10.2, 1.5)  circle [radius=0.08];   \draw [fill] (11.6,1.5) circle [radius=0.08];
\draw[thick, densely dotted] (10.2, 1.5) to [out=20,in=180] (10.4,1.8) to [out=-30,in=120] (10.8,1.2) to [out=0, in=180] (11.8, 1.6);
\draw[] (6, 7) to [out=25,in=90] (10, 2);
\draw[thick, densely dotted] (14.2, 0.5) to [out=20,in=180] (14.5,1.4) to [out=-30,in=120] (14.9,0.2) to [out=0, in=180] (15.2, 0.6)
to [out=0, in=180] (15.8, -0.5);
\draw[] (10,2)--(16,2)--(16,-1)--(10,-1)--(10,2);
\draw[thick] (11.6, 1.6) to [out=30,in=150] (14.2, 0.5);
 \draw [fill] (14.2, 0.5)  circle [radius=0.08];
 \draw [fill] (15.8, -0.6)  circle [radius=0.08];
 \draw[] (16, 2) to [out=70,in=160] (22, 5);
 \draw [fill] (20.2, 4.5)  circle [radius=0.08];
\draw [fill] (21.8, 4.6) circle [radius=0.08];
\draw[thick, densely dotted] (20.2, 4.5) to [out=20,in=180] (20.4,4.8) to [out=-30,in=120] (20.8,4.2) to [out=0, in=180] (21.8, 4.6);
\draw[] (22, 5) to [out=40,in=110] (28, 4);
\draw[thick, densely dotted] (26.5, 3.6) to [out=20,in=180] (26.8, 3.8) to [out=-30,in=120] (27, 3.2) to [out=0, in=180] (27.2, 3.6) to [out=0, in=180] (27.5, 3.8) to [out=0, in=180] (27.8, 3.2) ;
\draw[fill] (26.5, 3.6) circle [radius=0.08];  \draw[fill] (27.8, 3.2) circle [radius=0.08];
\draw[] (28, 4) to [out=60,in=180] (29,4.8);
\begin{pgfonlayer}{background}
 \fill[color=gray!10] (4,6) rectangle (6,7);
 \fill[color=gray!10] (10,-1) rectangle (16,2);
 \fill[color=gray!10] (14,-1) rectangle (16,2);
 \fill[color=gray!10] (20,4) rectangle (22,5);
  \fill[color=gray!10] (26,3) rectangle (28,4);
 \end{pgfonlayer}
\end{tikzpicture}
\caption{A more accurate coarse-graining approximation. Since the second and third visited boxes are
separated by less than $K_\epsilon$ mesoscopic time intervals, they need to be grouped together as
part of a single coarse-grained disorder variable $\Theta_{N, \eps}(\vec\sfi, \vec \sfa)$.
\label{CG-fig3}}
\end{figure}

We also need to consider mesoscopic time intervals $\cT_{\epsilon N}(\sfi_j)$ and $\cT_{\epsilon N}(\sfi_{j+1})$ visited by $\Gamma$ with $\sfi_{j+1}-\sfi_{j}<K_\eps$,
in which case, the no-triple configuration assumption ensures that mesoscopic time intervals $\cT_{\epsilon N}(\sfi^*)$ with $\sfi_j<\sfi^*<\sfi_{j+1}$,
or $\sfi^*\geq \sfi_j-K_\eps$, or $\sfi^*\leq \sfi_{j+1} + K_\eps$ are not visited by the configuration $\Gamma$. If $\sfa_j$ (resp.\ $\sfa'_{j+1}$)
denotes the mesoscopic spatial box of entry (resp.\ exit) by $\Gamma$ in $\cT_{\epsilon N}(\sfi_j)$ (resp.\ from $\cT_{\epsilon N}(\sfi_{j+1})$),
then we define an associated coarse-grained disorder variable $\Theta_{N, \eps}(\sfi_j, \sfi_{j+1}, \sfa_j, \sfa_{j+1}')$ by
\begin{align}\label{eq:Theta2}
\Theta_{N, \eps}(\vec \sfi, \vec \sfa) & := \Theta_{N, \eps}(\sfi, \sfi', \sfa, \sfa') \\
& := \frac{1}{\epsilon N} \!\!\!\!  \sum_{\substack{(d,x) \in \cB_{\epsilon N}(\sfi, \sfa)\\
	(f',y') \in \cB_{\epsilon N}(\sfi', \sfa')}}
	\sum_{\substack{\sfb: \, |\sfb - \sfa| \le M_\epsilon \\
	\sfb': \, |\sfb' - \sfa'| \le M_\epsilon \\
	\text{such that} \\ |\sfb' - \sfb| \le M_\epsilon \sqrt{\sfi'-\sfi}}}
	\sum_{\substack{(f,y) \in \cB_{\epsilon N}(\sfi, \sfb)\\
	(d',x') \in \cB_{\epsilon N}(\sfi', \sfb')\\
	\text{such that} \\ d \le f, \, d' \le f'}}
	\!\!\!\!\!\!\!\!\!\!  X_{d, f}(x, y) \; q_{d'-f}(x'-y) \; X_{d',f'}(x',y'), \notag
\end{align}
where $\sfi<\sfi'$, and we imposed a constraint on the mesoscopic spatial variables $\sfb, \sfb'$ relative to $\sfa$ and $\sfa'$.
To unify notation, when $\sfi=\sfi'$, we will let $\Theta_{N, \eps}(\vec \sfi, \vec \sfa):= \Theta_{N, \eps}(\sfi, \vec \sfa)$ as defined in \eqref{eq:Theta1}.
Schematically, the coarse-grained disorder variable can be represented as
\begin{equation}\label{eq:Theta-pic}
	\Theta_{N, \epsilon}(\vec\sfi, \vec\sfa) :=
	\begin{cases} \,
\begin{tikzpicture}[scale=0.4]
\draw[] (20,2.7)--(20,5.3);
\draw[] (22, 2.7)--(22, 5.3);
\draw[very thin, lightgray] (20,5)--(22,5);
\draw[very thin, lightgray] (20,3)--(22,3);
\draw[very thin, lightgray] (20, 4)--(22, 4);
\draw[thick, densely dotted] (20.2, 4.5) to [out=20,in=180] (20.5, 5.5) to [out=-30,in=120] (20.8,4.2) to [out=0, in=180] (21.8, 3.6);
\draw[fill]  (20.2, 4.5) circle [radius=0.08];
\draw[fill]  (21.8, 3.6) circle [radius=0.08];
\end{tikzpicture} & \qquad \text{if $\sfi=\sfi'$}, \\
&
\\
\begin{tikzpicture}[scale=0.4]
\draw[thick, densely dotted] (10.2, 1.5) to [out=20,in=180] (10.4,1.8) to [out=-30,in=120] (10.8,1.2) to [out=0, in=180] (11.8, 1.6);
\draw[thick, densely dotted] (14.2, 0.5) to [out=20,in=180] (14.5,1.4) to [out=-30,in=120] (14.9,0.2) to [out=0, in=180] (15.2, 0.6)
to [out=0, in=180] (15.8, -0.5);
\draw[] (16,2.3)--(16,-1.3);
\draw[] (10,-1.3)--(10,2.3);
\draw[fill]  (10.2, 1.5) circle [radius=0.08];
\draw[fill]  (11.8, 1.6) circle [radius=0.08];
\draw[fill]  (14.2, 0.5) circle [radius=0.08];
\draw[fill]  (15.8, -0.5) circle [radius=0.08];
\draw[] (11.8, 1.6) to [out=20,in=120] (14.2,0.5);
\draw[very thin, lightgray] (10,2)--(16,2);
\draw[very thin, lightgray] (10,-1)--(16,-1);
\draw[very thin, lightgray] (10,1)--(16,1);
\draw[very thin, lightgray] (10,0)--(16,0);
\draw[very thin, lightgray] (12,2)--(12,-1);
\draw[very thin, lightgray] (14,2)--(14,-1);
\end{tikzpicture}
& \qquad \text{if $\sfi'-\sfi \ge 1$}.
	\end{cases}
\end{equation}
Comparing with \eqref{eq:Zppchaos},
we can also interpret $\Theta_{N, \epsilon}(\vec\sfi, \vec\sfa)$ as an averaged partition function, where the time window has length $(\sfi'-\sfi+1)\eps N$ and the spatial average is on the scale $\sqrt{\eps N}$.

We are now ready to introduce a more accurate coarse-grained approximation of the partition function $\cZ^{\beta_N}_N(\varphi,\psi)$ than
$\cZ_{N,\epsilon}^{({\rm mock}-\cg)}(\varphi, \psi | \Theta_{N, \eps})$ in \eqref{eq:Zcg-mock}.

\begin{definition}[Coarse-grained partition function] Let $\varphi_\epsilon, \psi_\epsilon$ be defined in \eqref{aux-phi}, and
let $\Theta_{N, \eps} := (\Theta_{N, \eps}(\vec \sfi, \vec \sfa))_{\vec\sfi =(\sfi, \sfi'), \vec\sfa=(\sfa, \sfa')}$ be the family of
coarse-grained disorder variables defined in \eqref{eq:Theta1} and \eqref{eq:Theta2}. Then we define the coarse-grained
partition function
\begin{equation}\label{eq:Zcg-gen}
\begin{aligned}
	\mathscr{Z}_{\epsilon}^{(\cg)}(\varphi,\psi|\Theta_{N, \eps})
	 & :=  g_1(\varphi, \psi) + \epsilon
	 \sum_{k=1}^{(\log\frac{1}{\epsilon})^2} \!\!\! \sum_{{\sfb}, {\sfc} \in \Z^2} 	
	\sum_{\substack{(\vec\sfi_1, \ldots, \vec\sfi_k) \in \bcA_{\epsilon}^{(\notri)} \\
	(\vec\sfa_1, \ldots, \vec\sfa_k) \in \bcA_{\epsilon; \, b,c}^{(\diff)}}}
	\!\!\!\!\!\!\!\!\!\!\!\!\!\! \varphi_\epsilon({\sfb})  g_{\sfi_1}(\sfa_1 - \sfb)
	\Theta_{N, \eps}(\vec\sfi_1, \vec\sfa_1)  \\
	& \times
	\Bigg\{ \prod_{j=2}^k  g_{\sfi_j -\sfi_{j-1}'} (\sfa_j -\sfa_{j-1}')
	\, \Theta_{N, \eps}(\vec\sfi_j, \vec\sfa_j) \Bigg\}
	\, g_{\lfloor\frac{1}{\epsilon}\rfloor- \sfi_k'}({\sfc}-\sfa_k')
	 \psi_\epsilon({\sfc}) \,.
\end{aligned}
\end{equation}
\end{definition}
\noindent
In \eqref{eq:Zcg-gen}, the indices of the mesoscopic time intervals $\sfi_1\leq \sfi_1' <\sfi_2\leq \sfi_2' <\cdots \sfi_k \leq \sfi_k'$
have been grouped into blocks $\vec\sfi_j=(\sfi_j, \sfi_j')$, each associated with a coarse-grained variable
$\Theta_{N, \eps}(\vec\sfi_j, \vec \sfa_j)$. The set $\bcA_{\epsilon}^{(\notri)}$ ensures that for each $j$, $\sfi_j-\sfi'_{j-1} \geq K_\eps$,
$\sfi_j'-\sfi_j<K_\eps$, and $\sfi_{j+1}-\sfi'_j\geq K_\eps$, while the set $\bcA_{\epsilon; \, b,c}^{(\diff)}$ imposes diffusive
constraints on the mesoscopic spatial indices $\vec \sfa_j=(\sfa_j, \sfa_j')$ such that $|\sfa'_j-\sfa_j| \le  M_\epsilon \sqrt{\sfi'_j-\sfi_j+1}$ and
$|\sfa_{j} - \sfa'_{j-1}| \le M_\epsilon \sqrt{\sfi_{j} - \sfi'_{j-1} }$. For the precise definitions of $\bcA_{\epsilon}^{(\notri)}$ and
$\bcA_{\epsilon; \, b,c}^{(\diff)}$, see \cite[(4.4) and (4.6)]{CSZ23a}.

\begin{remark}{
Some remarks are in order:
\begin{itemize}
\item[1.] In \cite{CSZ23a}, we imposed a further constraint in the definition of the coarse-grained variables $\Theta_{N, \eps}$ so that
each depends only on disorder in finitely many mesocopic space-time blocks, which implies finite range dependence among the $\Theta_{N, \eps}$'s.
More precisely, in the definition of $X_{d, f}(x, y)$ in \eqref{eq:Xdf}, we replace the point-to-point partition function $Z^{\beta_N}_{d, f}(x, y)$ by a variant $Z^{(\diff)}_{d,f}(x,y)$ that does not collect any disorder
variables $\xi_N(n, z)$ outside a slightly super-diffusive window, see \cite[Section 4.3]{CSZ23a}.

\item[2.] $\mathscr{Z}_{\epsilon}^{(\cg)}(\varphi,\psi|\Theta_{N, \eps})$ is a polynomial chaos expansion in the family of coarse-grained disorder variables $\Theta_{N, \eps}= (\Theta_{N, \eps}(\vec \sfi, \vec \sfa))$ and has the same structure as the chaos expansion for $\cZ^{\beta_N}_N(\varphi,\psi)$ in \eqref{eq:Zppchaos}, which shows a degree of self-similarity.
    The time horizon $N$ in $\cZ^{\beta_N}_N(\varphi,\psi)$ is replaced by $1/\eps$, and the random walk transition kernels are replaced by heat kernels. Although the i.i.d.\ family of disorder variables $\xi_N$ has been replaced by the dependent family $\Theta_{N, \eps}$, the coarse-grained partition function $\mathscr{Z}_{\epsilon}^{(\cg)}(\varphi,\psi|\Theta_{N, \eps})$ is still critical since it approximates
$\cZ^{\beta_N}_N(\varphi,\psi)$ with $\beta_N$ in the critical window. See also \eqref{eq:SigAsymp}.

\item[3.] If we change the family $\Theta_{N, \eps}$ to $\Theta_{M, \eps}$, then $\mathscr{Z}_{\epsilon}^{(\cg)}(\varphi,\psi|\Theta_{M, \eps})$ will approximate $\cZ^{\beta_M}_M(\varphi,\psi)$ instead of $\cZ^{\beta_N}_N(\varphi,\psi)$. It was shown in \cite[Theorem 4.7]{CSZ23a} that
\begin{equation}
\lim_{\eps\downarrow 0} \limsup_{N\to\infty} \bbE\big[\big(\cZ^{\beta_N}_N(\varphi,\psi) - \mathscr{Z}_{\epsilon}^{(\cg)}(\varphi,\psi|\Theta_{N, \eps}) \big)^2\big] = 0.
\end{equation}
As we vary $N$, we only change the input $\Theta_{N, \eps}$ for the coarse-grained partition function $\mathscr{Z}_{\epsilon}^{(\cg)}(\varphi,\psi|\Theta_{N, \eps})$. This paves the way to apply a Lindeberg principle to show that $\mathscr{Z}_{\epsilon}^{(\cg)}(\varphi, \psi | \Theta_{N, \eps}) \stackrel{\rm dist}{\approx} \mathscr{Z}_{\epsilon}^{(\cg)}(\varphi, \psi | \Theta_{M, \eps})$.

\item[4.] In \eqref{eq:Zcg-gen}, the order of the chaos expansion in the variables $\Theta_{N, \eps}$ is restricted to $r\leq (\log \frac{1}{\epsilon})^2$. This induces a negligible error because it can be shown that the dominant contribution comes from terms of the order $\log \frac{1}{\eps}$, similar to the fact that in the chaos expansion for the averaged partition function $\cZ^{\beta_N}_N(\varphi,\psi)$ in the critical window, the dominant contribution comes from terms of order $\log N$, see Section \ref{S:2ndmom}. On the other hand, restricting
    to $r\leq (\log \frac{1}{\epsilon})^2$ allows us to control the cumulative error of replacing random walk transition kernels by heat kernels as in \eqref{eq:qtog}.
\end{itemize}
}
\end{remark}

\subsection{Applying Lindeberg to the coarse-grained model} \label{S:lind2}
We now sketch how to apply a Lindeberg principle to show that for each $\eps>0$,
\begin{equation}\label{eq:ZMappN}
\mathscr{Z}_{\epsilon}^{(\cg)}(\varphi, \psi | \Theta_{N, \eps}) \stackrel{\rm dist}{\approx} \mathscr{Z}_{\epsilon}^{(\cg)}(\varphi, \psi | \Theta_{M, \eps}) \qquad \mbox{as } M, N\to\infty.
\end{equation}
That is, the laws of $\mathscr{Z}_{\epsilon}^{(\cg)}(\varphi, \psi | \Theta_{N, \eps})$, $N\in\N$, form a Cauchy sequence.

The key difference from what was sketched in Section \ref{S:lind1} is that,
$$
\Psi_\eps(\Theta) := \mathscr{Z}_{\epsilon}^{(\cg)}(\varphi, \psi | \Theta)
$$
is a multilinear polynomial of the family of {\em dependant} random variables $\Theta=(\Theta(\vec\sfi, \vec\sfa))$, where
we take $\Theta=\Theta_{N, \eps}$ for $N\in\N$. Such a Lindeberg principle for multilinear polynomials of dependent random
variables was formulated in \cite[Lemma A.4]{CSZ23a}. Instead of repeating the statement of this Lindeberg principle, we
highlight here the key ingredients.

To establish \eqref{eq:ZMappN}, the first condition of the Lindeberg principle is that the covariances of
$(\Theta_{N, \eps}(\vec\sfi, \vec\sfa))$ converge to finite limits as $N\to\infty$ (see the bound in \cite[(A.11)]{CSZ23a}). It is clear from their definitions that
distinct $\Theta_{N, \eps}(\vec\sfi, \vec\sfa)$ are uncorrelated (but not independent). Therefore we only need to verify that for each
$(\vec\sfi, \vec\sfa)=( (\sfi, \sfi'), (\sfa, \sfa'))$ with $1\leq \sfi \leq \sfi' \leq \eps^{-1}$ and $\sfa, \sfa'\in\Z^2$,
the limit
\begin{equation}\label{eq:sigeps}
	\sigma_\eps^2(\vec\sfi, \vec\sfa):=  \lim_{N\to\infty}
	\bbE\Big[\big(\Theta_{N, \eps}(\vec\sfi, \vec\sfa)\big)^2\Big]
\end{equation}
exists and is finite. We will sketch the proof of this in Section \ref{S:2momAsy}. Although not needed for the Lindeberg principle,
we will also show in Section \ref{S:2momAsy} that
\begin{equation}\label{eq:SigAsymp}
\sigma_\eps^2((\sfi, \sfi), (\sfa, \sfa)) \sim \frac{4\pi}{\log \frac{1}{\eps}} \qquad \mbox{as} \quad \eps\downarrow 0.
\end{equation}

\begin{remark}
As will be explained later in \eqref{eq:ThetaV1}, among all coarse-grained disorder variables, $\Theta_{N, \eps}((\sfi, \sfi),  (\sfa, \sfa))$
are the dominant ones as $\eps\downarrow 0$. The precise asymptotics in \eqref{eq:SigAsymp} is a signature of the critical nature of the model, because it matches exactly the variance of the original disorder variable $\bbE[\xi_N(n, z)^2]=\sigma_N^2 \sim \frac{4\pi}{\log N}$
in the critical window (see \eqref{eq:sigN2}) if we identify the time horizon $\eps^{-1}$ with $N$. This shows a degree of self-similarity of the model under renormalisation.
\end{remark}

The second condition of the Lindeberg principle in \cite[Lemma A.4]{CSZ23a} is that the input $\Theta$ of $\Psi_\eps(\Theta)$ has
a local dependency structure such that if $\Theta(\vec\sfi', \vec \sfa')$ belongs to the dependency neighbourhood of $\Theta(\vec\sfi, \vec \sfa)$, and $\Theta(\vec\sfi'', \vec \sfa'')$ belongs to the dependency neighbourhood of $\{\Theta(\vec\sfi, \vec \sfa), \Theta(\vec\sfi', \vec \sfa')\}$, then no term in the chaos expansion of $\Psi_\eps(\Theta)$ can contain at least 2 of the 3 factors among $\{\Theta(\vec\sfi, \vec \sfa),
\Theta(\vec\sfi', \vec \sfa'), \Theta(\vec\sfi'', \vec \sfa'')\}$. In our setting, this is guaranteed by the facts that
we only consider $\Theta_{N, \eps}(\vec \sfi, \vec \sfa)$ with $\sfi'-\sfi<K_\eps$, $\Theta_{N, \eps}(\vec \sfi, \vec \sfa)$
and $\Theta_{N, \eps}(\vec \sfj, \vec \sfb)$ are independent when $[\sfi, \sfi']\cap [\sfj, \sfj']=\emptyset$, and each term in
the chaos expansion of $\Psi_\eps(\Theta)$ consists of well separated variables $\Theta(\vec \sfi_i, \vec \sfa_i)$ with
$\sfi_{i+1}-\sfi_i'\geq K_\eps$ thanks to the no-triple configuration condition in our coarse-graining.
Details can be found in \cite{CSZ23a}, Lemma 9.3.

The last condition of the Lindeberg principle are fourth moment bounds on $\Theta_{N, \eps}(\vec\sfi, \vec \sfa)$ and
the coarse-grained point-to-plane partition function, analogous to \eqref{eq:xiN3} and \eqref{eq:lind4}. More precisely,
we need to show that
\begin{equation}\label{eq:Theta4}
\limsup_{N\to \infty} \, \bbE \big[ \Theta_{N, \epsilon}(\vec\sfi, \vec\sfa)^4  \big] < \frac{C}{\log \frac{1}{\eps}},
\end{equation}
where $C$ is uniform in $(\vec\sfi, \vec \sfa)$ and $\eps>0$ small. We will sketch the proof of this bound in Section \ref{S:4momCG}.
This bound is not sharp, but already more than enough. In fact, it suffices to have a bound of the order
$\eps^{-c}$ for some $c>0$ sufficiently small.

For $\sfa_0\in \Z^2$ and $\sfm \in [1, 1/\eps]$, the point-to-plane analogue of the coarse-grained partition function $\mathscr{Z}_{\epsilon}^{(\cg)}(\varphi, \psi | \Theta)$ in
\eqref{eq:Zcg-gen} is defined by
\begin{align*}
	\mathscr{Z}_{[0, \sfm]}^{(\cg)}(\sfa_0,\psi_\eps|\Theta_{N, \eps})
	 & :=  g_\sfm(\sfa_0, \psi_\eps) +
	 \sum_{k=1}^{(\log\frac{1}{\epsilon})^2} \!\!\! \sum_{{\sfc} \in \Z^2} 	
	\sum_{\substack{(\vec\sfi_1, \ldots, \vec\sfi_k) \in \bcA_{\epsilon}^{(\notri)} \\
	(\vec\sfa_1, \ldots, \vec\sfa_k) \in \bcA_{\epsilon; \, \sfa_0, \sfc}^{(\diff)}}}
	\!\!\!\!\!\!\!\!\!\!\!\!\!\! g_{\sfi_1}(\sfa_1 - \sfa_0)
	\Theta_{N, \eps}(\vec\sfi_1, \vec\sfa_1)  \\
	& \times
	\Bigg\{ \prod_{j=2}^k  g_{\sfi_j -\sfi_{j-1}'} (\sfa_j -\sfa_{j-1}')
	\, \Theta_{N, \eps}(\vec\sfi_j, \vec\sfa_j) \Bigg\}
	\, g_{\sfm- \sfi_k'}({\sfc}-\sfa_k')
	 \psi_\epsilon({\sfc}) \,.
\end{align*}
Note that in contrast to \eqref{eq:Zcg-gen}, we do not multiply the sum by $\eps$ because there is no spatial averaging over the initial
point $\sfa_0$. The plane-to-point partition function $\mathscr{Z}_{[0, \sfm]}^{(\cg)}(\varphi_\eps, \sfa_0|\Theta_{N, \eps})$ is defined
similarly, and we can also shift the time window. We will then need the following analogue of \eqref{eq:lind4}: For any $\delta>0$ and
$\varphi\in C_c(\R^2)$, uniformly in $\sfm\in [1, 1/\eps]$, $\sfa_0\in \Z^2$, $N$ large and $\eps$ small, we have
\begin{align}\label{eq:Zcg4}
\bbE\big[\mathscr{Z}_{[0, \sfm]}^{(\cg)}(\sfa_0,\ind |\Theta_{N, \eps})^4\big] \leq C \eps^{-\delta} \quad \mbox{and} \quad
\bbE\big[\mathscr{Z}_{[0, \sfm]}^{(\cg)}(\sfa_0,\varphi_\eps|\Theta_{N, \eps})^4\big] \leq  C e^{-c\sqrt{\eps}|\sfa_0|} \eps^{-\delta}.
\end{align}
The same reasoning as in \eqref{eq:lind3} explains why the above bounds are sufficient for the application of the Lindeberg principle,
and why the bound in \eqref{eq:Theta4} can be replaced by $\eps^{-c}$ for any $c \in (0, 1-2\delta)$. We will sketch the proof of
\eqref{eq:Zcg4} in Section \ref{S:4momCG}.

\section{Second Moment Calculations and Dickman Subordinator}\label{S:2ndmom}

In this section, we explain how to compute the second moment of partition functions when $\beta_N$ and $\sigma_N^2$
are chosen in the critical window defined in \eqref{eq:sigma} for some $\theta\in \R$. This is based on a renewal representation
that leads to the so-called Dickman subordinator in the scaling limit. We will then discuss how these calculations can be applied to the
coarse-grained disorder variables $\Theta_{N, \epsilon}(\vec\sfi, \vec\sfa)$ introduced in Section \ref{S:CG}.

\subsection{Renewal representation}\label{S:renewal}

We recall from \eqref{eq:expan1} the polynomial chaos expansion of the point-to-plane partition function
\begin{align*}
	Z^{\beta_N}_N(0)
& = 1 + \sum_{r=1}^\infty \sum_{n_0=0<n_1<\cdots <n_r<N \atop z_0:=0, z_1, \ldots, z_r\in \Z^2} \prod_{i=1}^r q_{n_i-n_{i-1}}(z_i-z_{i-1}) \xi_N(n_i, z_i).
\end{align*}
Since the summands are $L^2$-orthogonal and $\bbE[\xi_N(n_i, z_i)^2]=\sigma_N^2$ by \eqref{eq:xi}, we have
\begin{align} \label{eq:2mom2}
	\bbE[(Z^{\beta_N}_N(0))^2]
& = 1 + \sum_{r=1}^\infty \sum_{n_0=0<n_1<\cdots <n_r<N \atop z_0:=0, z_1, \ldots, z_r\in \Z^2} \sigma_N^{2r} \prod_{i=1}^r q_{n_i-n_{i-1}}(z_i-z_{i-1})^2.
\end{align}
We now introduce a probabilistic representation of this sum by observing that
$$
\frac{ q_n(x)^2}{R_N} \, \ind_{\{n\leq  N\}}
$$
defines a probability kernel on $\{1, \ldots, N\}\times \Z^2$, which can be interpreted as the increment distribution of a time-space renewal
process. More precisely, if $(T^{(N)}_i, X^{(N)}_i)_{i\in\N}$ are i.i.d.\ random variables with distribution $\frac{ q_n(x)^2}{R_N} \, \ind_{\{n\leq  N\}}$, and we denote their partial sums by
\begin{equation} \label{eq:tauS}
	\tau^{(N)}_k := T^{(N)}_1 + \ldots + T^{(N)}_k \,, \qquad
	S^{(N)}_k := X^{(N)}_1 + \ldots + X^{(N)}_k,
\end{equation}
then $(\tau^{(N)}_k)_{k\in\N_0}$ (with $\tau^{(N)}_0:=0$) is a renewal process with increment distribution $\frac{q_{2n}(0)}{R_N} \ind_{\{n\leq  N\}}$, and we have
\begin{align}
	\bbE[(Z^{\beta_N}_N(0))^2]
& = 1 + \sum_{r=1}^\infty \sum_{n_0=0<n_1<\cdots <n_r<N \atop z_0:=0, z_1, \ldots, z_r\in \Z^2} (\sigma_N^2 R_N)^r
\P\big( (\tau^{(N)}_k, S^{(N)}_k) = (n_k, z_k) \mbox{ for } 1\leq k\leq r\big) \notag \\
& = \sum_{r=0}^\infty \bigg(1 + \frac{\theta + o(1)}{\log N}\bigg)^r \P(\tau^{(N)}_r <N). \label{eq:2mom3}
\end{align}
where we used \eqref{eq:sigma} for $\sigma_N^2$ in the critical window
and summed over the intermediate values.

We have a similar renewal representation for the second moment of the point-to-point partition function $Z^{\beta_N}_{n, m}(x, y)$, whose
chaos expansion is given in \eqref{eq:expan2}. Actually, a more fundamental object in our analysis is
\begin{align*}
U_N(n, x) := \sigma_N^2 \bbvar\big( Z_{0, n}^{\beta_N} (0, x) \big) = \sigma_N^2 \bbvar\big( Z_{\ell, \ell+n}^{\beta_N} (z, x+z) \big),
\end{align*}
for which we have
\begin{equation}\label{eq:U}
\begin{split}
	U_N(n,x) := \begin{cases}
	\displaystyle
	\ind_{\{x=0\}}
	& \text{ if } n = 0, \\
	\rule{0pt}{4.3em}
	\displaystyle
\begin{split}
	\sigma_N^2 \, q_n(x)^2 \,+\,
	& \sum_{k=1}^\infty \sum_{\substack{z_1,\ldots, z_k \in \Z^2 \\ 0 < n_1 < \ldots < n_k < n}}
	 \prod_{j=1}^{k+1} \sigma_N^2 q_{n_j-n_{j-1}}(z_j-z_{j-1})^2
\end{split}
	& \text{ if } n \ge 1,
	\end{cases}
\end{split}
\end{equation}
where $n_0:=0$, $z_0=0$, $n_{k+1}:=n$ and $z_{k+1}:=x$.
Similar to the graphical representation \eqref{eq:expan2} of the chaos expansion of $Z_{n,m}^\beta(x,y)$, we can graphically represent
the expansion for $U_N(m-n, y-x) = \sigma_N^2 \bbvar\big( Z_{n, m}^{\beta_N} (x, y) \big)$ by
\begin{align*}
U_N(m-n, y-x) =
\sum_{r\geq 1} \sumtwo{n<n_1<\cdots<n_r < m}{z_1,...,z_r \,\in \, \Z^2} \!\!
\begin{tikzpicture}[scale=0.45]]
\draw [] (0, 0)  circle [radius=0.1];
\draw[thick] (0,0) to [out=70,in=160] (2,1); \draw[thick] (0,0) to [out=-20,in=-120] (2,1);
\draw [fill] (2, 1)  circle [radius=0.1];
\draw[thick] (2,1) to [out=-50,in=110] (3,-1);\draw[thick] (2,1) to [out=-80,in=130] (3,-1);
\draw [fill] (3, -1)  circle [radius=0.1];
\draw[thick] (3,-1) to [out=80,in=-180] (4.5,0.5); \draw[thick] (3,-1) to [out=20,in=260] (4.5,0.5);
\draw [fill] (4.5, 0.5)  circle [radius=0.1];
\draw[thick] (4.5,0.5) to [out=60,in=120] (6, 0.5);\draw[thick] (4.5,0.5) to [out=-60,in=-120] (6, 0.5);
\draw [fill] (6, 0.5)  circle [radius=0.1];
\draw[thick] (6,0.5) to [out=-80,in=160] (7, -1); \draw[thick] (6,0.5) to [out=-20,in=110] (7, -1);
\draw [fill] (7, -1)  circle [radius=0.1];
\draw[thick] (7,-1) to [out=60,in=200] (8, 0); \draw[thick] (7,-1) to [out=20,in=220] (8.3, -0.5);
\node at (9,0) {$\cdots$};\node at (10.2,0) {$\cdots$};
\draw[thick] (11,1) to [out=10,in=100] (13, 0); \draw[thick] (11,-1) to [out=-10,in=-100] (13, 0);
\draw [fill] (13, 0)  circle [radius=0.1];
\node at (13.5,-1) {\scalebox{0.6}{$(n_r,z_r)$}};
\node at (2,1.7) {\scalebox{0.6}{$(n_1,z_1)$}};
\node at (0,-0.5) {\scalebox{0.6}{$(n,x)$}};
\draw[thick] (13,0) to [out=60,in=180] (14.5, 1);
\draw[thick] (13,0) to [out=10,in=-130] (14.5, 1);
\draw [fill] (14.5, 1)  circle [radius=0.1];
\node at (15.2,0.5) {\scalebox{0.6}{$(m,y)$}};
 \end{tikzpicture}
\end{align*}
where each lace corresponds to one factor of $q_{n_j-n_{j-1}}(z_j-z_{j-1})$, each solid circle corresponds to
one factor of $\sigma_N^2$, while the empty circle at $(n,x)$ means there is no factor $\sigma_N^2$ assigned
to $(n,x)$. A more compact graphical representation is
\begin{align}\label{U-diagram}
\sigma_N^2 U_N(m-n,y-x)
&\equiv \quad
\begin{tikzpicture}[baseline={([yshift=1.3ex]current bounding box.center)},vertex/.style={anchor=base,
    circle,fill=black!25,minimum size=18pt,inner sep=2pt}, scale=0.45]
\draw  [fill] (0, 0)  circle [radius=0.2];  \draw  [fill] ( 4,0)  circle [radius=0.2];
\draw [-,thick, decorate, decoration={snake,amplitude=.4mm,segment length=2mm}] (0,0) -- (4,0);
\node at (0,-1) {\scalebox{0.8}{$(n,x)$}}; \node at (4,-1) {\scalebox{0.8}{$(m,y)$}};
\end{tikzpicture} \qquad\qquad \text{or} \\
&= \sum_{r \geq 1 }\sumtwo{n_1,...,n_r}{z_1,...,z_r}
\begin{tikzpicture}[baseline={([yshift=1.3ex]current bounding box.center)},vertex/.style={anchor=base,
    circle,fill=black!25,minimum size=18pt,inner sep=2pt}, scale=0.45]
\draw  [fill] (0, 0)  circle [radius=0.2];  \draw  [fill] ( 2,0)  circle [radius=0.2]; \draw  [fill] (4,0)  circle [radius=0.2];
\draw  [fill] (8, 0)  circle [radius=0.2]; \draw  [fill] (10, 0)  circle [radius=0.2];
\draw [thick] (0,0)  to [out=45,in=135]  (2,0) to [out=45,in=135]  (4,0) to [out=45,in=180]  (4.5,0.3);
\draw [thick] (0,0)  to [out=-45,in=-135]  (2,0) to [out=-45,in=-135]  (4,0)  to [out=-45,in=180]  (4.5,-0.3);
\draw [thick] (7.5,0.3)  to [out=0,in=135]  (8,0) to [out=45,in=135]  (10,0);
\draw [thick] (7.5,-0.3) to [out=0,in=-135] (8,0)  to [out=-45,in=-135]  (10,0);
\node at (0,-1) {\scalebox{0.7}{$(n,x)$}}; \node at (2,-1) {\scalebox{0.7}{$(n_1,z_1)$}};
\node at (4,-1) {\scalebox{0.7}{$(n_2,z_2)$}}; \node at (8,-1) {\scalebox{0.7}{$(n_k,z_k)$}};
\node at (10,-1) {\scalebox{0.7}{$(m,y)$}};
\node at (6,0) {$\cdots$};
\end{tikzpicture} \notag
\end{align}
where we included an extra factor of $\sigma_N^2$ to turn the empty circle at $(n, x)$ into a solid circle.

Using the time-space renewal process introduced in \eqref{eq:tauS}, we can now write
\begin{align}\label{renewU2}
U_N(n,x)
&= \sum_{r=1}^\infty \Big( 1+\frac{\theta+o(1)}{\log N}\Big)^{r} \, \, \P( \tau_r^{(N)}=n\, , \, S_r^{(N)}=x).
\end{align}
Summing over $x\in \Z^2$ then gives
\begin{align}\label{renewU}
U_N(n) := \sum_{x\in \Z^2} U_N(n, x) = \sum_{r=1}^\infty \Big( 1+\frac{\theta+o(1)}{\log N}\Big)^{r} \, \, \P( \tau_r^{(N)}=n).
\end{align}

To identify sharp asymptotics for $\bbE[(Z^{\beta_N}_N(0))^2]$ in \eqref{eq:2mom3} and for $U_N(n,x)$ and $U_N(n)$ as $N\to\infty$,
the key observation is that $\frac{\tau^{(N)}_{\lfloor s \log N \rfloor}}{N}$ converges in distribution to a L\'evy process called the
{\em Dickman subordinator}.

\subsection{Dickman subordinator}
We now introduce the {\em Dickman subordinator}, which is
a truncated, $0$-stable L\'evy process $Y=(Y_s)_{s\geq 0}$ with L\'evy measure
\begin{equation} \label{eq:bsLevymeasure}
	\nu(\dd t) := \frac{1}{t} \, \ind_{(0,1)}(t)  \, \dd t  \,.
\end{equation}
For $s\geq0$, it has Laplace transform
\begin{equation}\label{eq:LapY}
	\E[e^{\lambda  Y_s}] =
	\exp \bigg\{s \, \int_{0}^1
	(e^{\lambda t } - 1) \, \frac{\dd t}{t}
	 \bigg\} \,.
\end{equation}
The density of the $Y_s$ can be computed explicitly as (see \cite[Theorem 1.1]{CSZ19a})
\begin{equation}\label{eq:scalingf}
	 f_s(t)\,  := \frac{\P(Y_s \in \dd t)}{\dd t} =  \left\{
	 \begin{aligned}
	 & \frac{s\, t^{s-1} \, e^{-\gamma \, s}}{\Gamma(s+1)} \quad \quad & \mbox{for } t\in (0, 1], \\
	 & \frac{s\, t^{s-1}e^{-\gamma s}}{\Gamma(s+1)} - st^{s-1} \int_0^{t-1} \frac{f_s(a)}{(1+a)^s} \dd a \quad \quad & \mbox{for } t\in (1, \infty).
	 \end{aligned}
	 \right.
\end{equation}
The calculation in \cite{CSZ19a} was partly based on the observation that conditional on all jumps of $Y$ up to time $s$ being smaller than a value $t\in (0, 1)$, $Y_s/t$ has the same law as the original $Y_s$ (see \cite[Prop.~B.1]{CSZ19a}). More properties of the Dickman subordinator can be found in \cite{GKLV24}.

The name {\it Dickman subordinator} is motivated by the fact that
\begin{align*}
e^{\gamma} f_1(t)=\rho(t),
\end{align*}
where $\gamma= - \int_0^\infty (\log u)  e^{-u} {\rm d}u$ is the Euler–Mascheroni constant and $\rho(\cdot)$ is the so-called {\em Dickman function}. It was originally derived by Dickman \cite{Dickman} in studying the distribution of the largest prime factor of a uniformly chosen number in $\{1, \ldots, n\}$. It also arises from the distribution of the size of the longest cycle in a uniformly chosen random permutation of $n$ elements. More instances can be found in \cite{PW04} and the references therein.

The key observation in the asymptotic analysis of $\bbE[(Z^{\beta_N}_N(0))^2]$, $U_N(n,x)$ and $U_N(n)$ is the following convergence result for
the time-space renewal process $(\tau^{(N)}, S^{(N)})$ defined in \eqref{eq:tauS} (see \cite[Prop.~2.2]{CSZ19a} and \cite[Lemma 3.3]{CSZ23a}).
\begin{lemma}\label{L:DickConv}
We have the following weak convergence in c\`adl\`ag space:
\begin{equation}\label{joint-renew}
	\Bigg( \frac{\tau^{(N)}_{\lfloor s \log N \rfloor}}{N},
	\frac{S^{(N)}_{\lfloor s \log N \rfloor}}{\sqrt{N}} \Bigg)_{s\geq 0} \xrightarrow[N\to\infty]{d}
	(Y_s, \tfrac{1}{\sqrt 2} W_{Y_s})_{s\geq 0},
\end{equation}
where $(Y_s)_{s\geq 0}$ is the Dickman subordinator and $(W_s)_{s\geq 0}$ an independent standard Brownian motion on $\R^2$.
\end{lemma}
\begin{proof}[Proof sketch]
The finite-dimensional distribution convergence in \eqref{joint-renew} can be verified by showing the convergence of their Laplace transforms. We will illustrate this for
$\tau^{(N)}$:
\begin{align*}
\E\Big[ e^{\tfrac{\lambda}{N} \,\tau_{s\log N}^{(N)}}\Big] = \E\Big[ e^{\tfrac{\lambda}{N} \,T_{1}^{(N)}}\Big]^{s\log N}
&=  \Big(\, 1 +\frac{1}{\log N}\sum_{n=1}^N \big(e^{\tfrac{\lambda}{N} n} -1\big) \frac{1+o(1)}{n}   \,\Big)^{s\log N}\\
&\xrightarrow[N\to\infty]{} \exp\Big( s \int_0^1\big(e^{\lambda x} -1\big) \frac{\dd x}{x}\Big),
\end{align*}
where in the second equality we used the local central limit theorem to approximate $\tfrac{1}{R_N} q_{2n}(0)= \tfrac{1}{\log N}\tfrac{1}{n} (1+o(1))$
for $n$ large. The full f.d.d.\ convergence can be found in the proof of \cite[Prop.~2.2]{CSZ19a}, while process level tightness was proved in
\cite[Lemma 3.3]{CSZ23a}.
\end{proof}
\medskip

In our analysis of the polymer partition functions in the critical window with parameter $\theta$, defined in \eqref{eq:sigma},
we will need the following weighted Green's function of the Dickman subordinator:

\begin{equation} \label{dick-green1}
	G_{\theta}(t) := \int_0^\infty
	e^{\theta s} \, f_s(t) \, \dd s \quad t \in (0,\infty), \,\, \theta\in \R.
\end{equation}
When $t\in (0,1]$, by \eqref{eq:scalingf}, we have the more explicit form
\begin{equation} \label{dick-green2}
	G_{\theta}(t) = \int_0^\infty
	\frac{ e^{(\theta - \gamma) s} \,
	s \, t^{s-1}}{\Gamma(s+1)} \, \dd s \,,
	\qquad  \quad t \in (0,1] \,, \
	\theta \in \R \,.
\end{equation}
We note that $G_\theta$ is related to the Volterra function
\begin{align*}
\nu(t):=\int_0^\infty \frac{t^s}{\Gamma(s+1)} \dd s.
\end{align*}
See \cite{A10} for more information on functions of this type.

We also record here the small $t$ asymptotics of $G_\theta(t)$, which plays a crucial role in our analysis of the critical $2d$ SHF (see
\cite[Prop.~1.6]{CSZ19a}):

\begin{proposition}
For $t\in (0,1]$, $G_\theta(t)$ is $C^\infty$ and strictly positive. As $t \downarrow 0$, we have
\begin{equation}\label{eq:Gas}
	G_\theta(t) = \frac{1}{t(\log\frac{1}{t})^2} \bigg\{ 1 + \frac{2\theta}{\log\frac{1}{t}}
	+ O\bigg(\frac{1}{(\log\frac{1}{t})^2}\bigg) \bigg\}
\end{equation}
and
\begin{equation}\label{eq:G5}
\int_0^t G_\theta(s) {\rm d}s = \frac{1}{\log \frac{1}{t}} \bigg\{ 1 + \frac{\theta}{\log\frac{1}{t}}
	+ O\bigg(\frac{1}{(\log\frac{1}{t})^2}\bigg) \bigg\} .
\end{equation}
\end{proposition}
These asymptotics were derived in \cite{CSZ19a} via a renewal framework. But they can also be derived from the
asymptotic theory of Volterra functions \cite{A10}. Note that $G_\theta$ is barely integrable near $0$, which
is the hallmark of a model at the critical dimension.

\subsection{Asymptotics} \label{S:2momAsy}

We are now ready to state the asymptotics for $\bbE[(Z^{\beta_N}_N(0))^2]$, $U_N(n,x)$ and $U_N(n)$ in terms
of the weighted Green's function $G_\theta(t)$ for the Dickman subordinator. We will also illustrate how
to identify the asymptotics of the coarse-grained disorder variables $\Theta_{N, \epsilon}(\vec\sfi, \vec\sfa)$
introduced in Section \ref{S:CG}.

\begin{lemma}\label{L:Z2Asymp}
Let $\beta_N$ and $\sigma_N^2$ be chosen in the critical window defined in \eqref{eq:sigma} for some $\theta\in \R$. Then
we have
\begin{equation}
\bbE[(Z^{\beta_N}_N(0))^2] = (\bar G_\theta +o(1)) \log N \qquad \mbox{with} \quad \bar G_\theta:= \int_0^1 G_\theta(t) {\rm d}t.
\end{equation}
\end{lemma}
\begin{proof}
By \eqref{eq:2mom3}, we have
\begin{align*}
\bbE[(Z^{\beta_N}_N(0))^2]
& = \sum_{r=0}^\infty \bigg(1 + \frac{\theta + o(1)}{\log N}\bigg)^r \P(\tau^{(N)}_r <N) \\
& = \log N \int_0^\infty \bigg(1 + \frac{\theta + o(1)}{\log N}\bigg)^{\lfloor s\log N\rfloor} \P\Big(\frac{1}{N}\tau^{(N)}_{\lfloor s\log N\rfloor} <1\Big) {\rm d}s \\
& = (1+o(1))\log N \int_0^\infty e^{\theta s} \P(Y_s<1) {\rm d}s = \Big(\int_0^1 G_\theta(t) {\rm d}t + o(1)\Big) \log N,
\end{align*}
where the last line follows from the convergence in Lemma \ref{L:DickConv} and the lower large deviation bounds for the renewal process $\tau^{(N)}$ stated in \cite[Lemma 6.1]{CSZ19a}.
\end{proof}

We also record here the asymptotics for $U_N(n)$ and $U_N(n)$, which are of the local limit theorem type and play a crucial in the analysis in \cite{CSZ23a}. These results were proved in \cite[Theorems 1.4, 2.3, 3.7]{CSZ19a}, although the constants therein differ because the underlying random was the simple symmetric random walk on $\Z^2$.

\begin{proposition}\label{UN-as}
Let $\theta\in \R$ be as in Lemma \ref{L:Z2Asymp}. For any fixed $\delta > 0$ and uniformly in $\delta N \le n \le N$,
as $N\to\infty$, we have
\begin{align}
	U_N(n) &= \frac{\log N}{N} \, G_\theta\big(\tfrac{n}{N}\big)
	(1+ o(1)), \label{eq:asU1} \\
	U_N (n,x)
	&= \frac{\log N}{N^2} \,
	G_{\theta}\big(\tfrac{n}{N}\big) \, g_{ \frac{n}{2N}}
	\big(\tfrac{x}{\sqrt{N}}) \,\big(1+ o(1)\big) \,. \label{eq:asU2}
\end{align}
\end{proposition}
\begin{proof}[Proof Sketch]
Because the asymptotics are of the local limit theorem type, they do not follow directly from Lemma \ref{L:DickConv}. Instead,
using a renewal decomposition, \eqref{eq:asU1} and \eqref{eq:asU2} were deduced from their integrated versions in \cite[Sections 6.2 \& 8.3]{CSZ19a}.

We sketch here an alternative approach that is more transparent. First of all, with a bit more effort, the weak convergence in Lemma \ref{L:Z2Asymp} can be upgraded to a local limit theorem:
\begin{gather}
\P\bigg( \frac{1}{N} \tau_{\lfloor s\log N\rfloor}^{(N)} =\frac{n}{N}\bigg) = \frac{1}{N} f_s\big( \tfrac{n}{N} \big) \, (1+o(1)), \label{renewLLT1} \\
\P\bigg( \frac{1}{N} \tau_{\lfloor s\log N\rfloor}^{(N)} =\frac{n}{N}\, , \, \frac{1}{\sqrt{N}}S^{(N)}_{\lfloor s\log N\rfloor}=\frac{x}{\sqrt{N}} \bigg)
 = \frac{1}{N^2} \, f_s\big( \tfrac{n}{N} \big) g_{\tfrac{n}{2N}}\big( \frac{x}{\sqrt{N}}\big) \, (1+o(1)), \label{renewLL2}
\end{gather}
where $f_s(\cdot)$ is the density of the Dickman subordinator in \eqref{eq:scalingf}, $g_t(\cdot)$ is the heat kernel,
and the  $o(1)$ error term is uniform in $\delta N \leq n \leq N$ and $|x|\leq \frac{1}{\delta}\sqrt{N}$.

To see \eqref{eq:asU1}, note that by \eqref{renewU}, we have
\begin{align*}
U_N(n) & = \sum_{r=1}^\infty \Big( 1+\frac{\theta+o(1)}{\log N}\Big)^{r} \, \, \P\Big( \frac{1}{N}\tau_r^{(N)}= \frac{n}{N}\Big) \\
&= \log N \int_0^\infty \Big(1+\frac{\theta+o(1)}{\log N}\Big)^{\lfloor s\log N\rfloor}
\,  \P\big( \tfrac{1}{N} \tau_{\lfloor s\log N\rfloor}^{(N)} = \tfrac{n}{N} \big) \,\dd s\\
&= \frac{\log N}{N}  \int_0^\infty (1+o(1))  e^{\theta s} f_s\big( \tfrac{n}{N} \big) \, \dd s \\
&=(1+o(1)) \frac{\log N}{N}  G_\theta\big( \tfrac{n}{N} \big).
\end{align*}

Similarly, to see \eqref{eq:asU2}, we use \eqref{renewU2} to write
\begin{align*}
U_N(n,x)
&= \sum_{r=1}^\infty \Big( 1+\frac{\theta+o(1)}{\log N}\Big)^r \, \, \P\Big( \frac{1}{N}\tau_r^{(N)}=\frac{n}{N}\, ,\, \frac{1}{\sqrt N}S_r^{(N)}= \frac{x}{\sqrt N}\Big) \\
&= \log N \int_0^\infty \Big(1+\frac{\theta+o(1)}{\log N}\Big)^{\lfloor s\log N\rfloor}
\,  \P\big( \tfrac{1}{N} \tau_{\lfloor s\log N\rfloor}^{(N)} = \tfrac{n}{N}, \tfrac{1}{\sqrt N} S_{\lfloor s\log N\rfloor}^{(N)} = \tfrac{x}{\sqrt N} \big) \,\dd s\\
&= \frac{\log N}{N^2}  \int_0^\infty (1+o(1))  e^{\theta s} f_s\big( \tfrac{n}{N} \big) g_{\tfrac{n}{2N}}\big( \tfrac{x}{\sqrt{N}}\big)\, \dd s \\
&=(1+o(1)) \frac{\log N}{N^2}  G_\theta\big( \tfrac{n}{N} \big) g_{\tfrac{n}{2N}}\big( \tfrac{x}{\sqrt{N}}\big).
\end{align*}
\end{proof}

We now sketch the proof of \eqref{eq:sigeps}, that is, the second moment of the coarse-grained disorder variables $\Theta_{N, \epsilon}(\vec\sfi, \vec\sfa)$ converge as $N\to\infty$.

\begin{proof}[Proof Sketch for \eqref{eq:sigeps}]
For simplicity, we focus on $\Theta_{N, \eps}(\sfi, \vec \sfa)$ defined in \eqref{eq:Theta1} and \eqref{eq:Theta-pic}, with $\sfi=1$ and $\vec \sfa=(0, \sfa)$ for some $\sfa \in \Z^2$. By \eqref{eq:Theta1} and \eqref{eq:Xdf}, we have
\begin{align*}
\Theta_{N, \eps}(1, (0, \sfa)) := \frac{1}{\epsilon N} \sum_{(d,x) \in \cB_{\epsilon N}(1, 0)\cap \cB_{\epsilon N}(1, \sfa)} \!\!\!\!\!\!\!\!\!\!\!\!\!\!\!\! \xi_N(d, x)
+ \frac{1}{\epsilon N} \, \sum_{\substack{(d,x) \in \cB_{\epsilon N}(1, 0)\\
	(f,y) \in \cB_{\epsilon N}(1, \sfa)\\\text{with }d<f }} \xi_N(d,x) \,
	Z^{\beta_N}_{d, f}(x,y) \, \xi_N(f,y).
\end{align*}
Note that the first sum is $0$ if  $\cB_{\epsilon N}(1, 0)\cap \cB_{\epsilon N}(1, \sfa) = \emptyset$.
The summands are $L^2$-orthogonal, and
$$
\bbE[\xi_N(d,x)^2 (Z^{\beta_N}_{d, f}(x,y))^2 \xi_N(f,y)^2] = \sigma_N^2 U_N(f-d, y-x).
$$
Therefore,
\begin{align*}
\bbE[\Theta_{N, \eps}(1, (0, \sfa))^2] = \frac{1}{(\epsilon N)^2} \!\! \sum_{(d,x) \in \cB_{\epsilon N}(1, 0)\cap \cB_{\epsilon N}(1, \sfa)} \!\!\!\!\!\!\!\!\!\!\!\! \sigma_N^2
+ \frac{1}{(\epsilon N)^2} \!\! \sum_{\substack{(d,x) \in \cB_{\epsilon N}(1, 0)\\
	(f,y) \in \cB_{\epsilon N}(1, \sfa)\\\text{with }d<f }} \!\!\!\!\!\! \sigma_N^2 U_N(f-d, y-x),
\end{align*}
where the first term is bounded by $\sigma_N^2\sim C/\log N$ and hence vanishes in the limit $N\to\infty$. For the second term, we interpret
$(\epsilon N)^{-2} \sum_{(d, x) \in \cB_{\epsilon N}(1, 0)}$ as an expectation over a point $(d, x)$ chosen uniformly from $\cB_{\epsilon N}(1, 0)$.
Using the same renewal representation of $U_N(f-d, y-x)$ as before, we find that
\begin{align*}
\lim_{N\to\infty} \frac{1}{(\epsilon N)^2} \, \sum_{\substack{(d,x) \in \cB_{\epsilon N}(1, 0)\\
	(f,y) \in \cB_{\epsilon N}(1, \sfa)\\\text{with }d<f }} \sigma_N^2 U_N(f-d, y-x)
= 4\pi \int_0^\infty e^{\theta s} \P_{\cB_{\epsilon}(1, 0)}((Y_s, V_s)\in \cB_{\epsilon}(1, \sfa))  {\rm d}s,
\end{align*}
where $4\pi$ comes from $\sigma_N^2 \sim 4\pi/\log N$, and $\P_{\cB_{\epsilon}(1, 0)}$ denotes the law of the
continuum time-space renewal process $(Y_s, V_s):= (Y_s, \tfrac{1}{\sqrt 2} W_{Y_s})$ in Lemma \ref{L:DickConv} with
$(Y_0, V_0)$ chosen uniformly from $\cB_{\epsilon}(1, 0)$, with $\cB_{\epsilon}(\sfi, \sfa):= ((\sfi-1)\eps, \sfi\epsilon)\times ((\sfa-(1,1))\sqrt{\eps}, \sfa\sqrt{\epsilon})$, cf.\ \eqref{Bbox}.
Therefore,
\begin{equation}\label{eq:sigepsa}
\sigma^2_\eps (1, (0, \sfa)) := \lim_{N\to\infty} \bbE[\Theta_{N, \eps}(1, (0, \sfa))^2] = 4\pi \int_0^\infty e^{\theta s} \P_{\cB_{\epsilon}(1, 0)}((Y_s, V_s)\in \cB_{\epsilon}(1, \sfa))  {\rm d}s.
\end{equation}
For more general coarse-grained disorder variables $\Theta_{N, \epsilon}(\vec\sfi, \vec\sfa)$, the calculations are similar. See \cite[Section 7.1]{CSZ23a} for more details.
\end{proof}

Lastly, we sketch the proof of \eqref{eq:SigAsymp}, which shows that the coarse-grained variable $\Theta_{N, \epsilon}(\sfi, (\sfa, \sfa))$
has an asymptotic variance of $4\pi/\log \eps^{-1}$ as $N\to\infty$ and $\eps\downarrow 0$, which is comparable to ${\mathbb V}{\rm ar}(\xi_N(n, x))=\sigma_N^2\sim 4\pi/\log N$ for the original model if we identify $\eps^{-1}$ as the time horizon of the coarse-grained model.

\begin{proof}[Proof Sketch for \eqref{eq:SigAsymp}]
It was shown in \cite[Lemma 7.1]{CSZ23a} that whenever $\sfi>1$ or $\sfa\neq 0$,
\begin{equation}\label{eq:ThetaV1}
\sigma^2_\eps ((1, \sfi), (0, \sfa))  := \lim_{N\to\infty} {\mathbb V}{\rm ar}(\Theta_{N, \epsilon}((1, \sfi), (0, \sfa))) \leq \frac{C e^{-c |\sfa|^2/\sfi}}{(\log \tfrac{1}{\eps})^2} ,
\end{equation}
while $\sigma^2_\eps(\sfi, (\sfa, \sfa)):= \sigma^2_\eps((\sfi, \sfi), (\sfa, \sfa))$ is of the order $1/\log \tfrac{1}{\eps}$. Therefore among all the coarse-grained disorder variables, $\Theta_{N, \epsilon}(\sfi, (\sfa, \sfa))$ are dominant as $N\to\infty$ and $\eps\downarrow 0$. Furthermore,  as $\eps\downarrow 0$,
\begin{align*}
\sigma^2_\eps(1, (0, 0)) = (1+o(1)) \sum_{\sfa\in\Z^2} \sigma^2_\eps(1, (0, \sfa)) = (1+o(1)) \lim_{N\to\infty} {\mathbb V}{\rm ar}\Big(\sum_{\sfa\in\Z^2}\Theta_{N, \epsilon}(1, (0, \sfa))\Big),
\end{align*}
where the $o(1)$ correction is due to \eqref{eq:ThetaV1}.
By \eqref{eq:sigepsa}, we have
\begin{align*}
\sigma^2_\eps(1, (0, 0)) & = (1+o(1)) 4\pi \int_0^\infty e^{\theta s} \P_{\cB_{\epsilon}(1, 0)}(Y_s <\eps)  {\rm d}s \\
& = \frac{4\pi +o(1)}{\eps} \int_0^\eps  \Big(\int_0^\infty e^{\theta s} \P(Y_s <\eps-u) {\rm d}s\Big)\,  {\rm d}u  \\
& = \frac{4\pi +o(1)}{\eps} \int_0^\eps  \Big(\int_0^{\eps-u} G_\theta(t) \dd t \Big) \,  {\rm d}u \\
& = \frac{4\pi +o(1)}{\eps} \int_0^\eps  \frac{1}{\log \tfrac{1}{\eps -u}} \,  {\rm d}u =  \frac{4\pi +o(1)}{\log \tfrac{1}{\eps}},
\end{align*}
where we used that $Y_0$ is uniformly distributed in $[0, \eps]$ under $\P_{\cB_{\epsilon}(1, 0)}$ and $Y_0=0$ under $\P$. We
also used \eqref{eq:G5} in the last line.
\end{proof}

\subsection{Covariance kernel of the critical $2d$ SHF} \label{S:2ndCG}

We sketch here the covariance kernel for the critical $2d$ stochastic heat flow $\SHF(\theta)=(\SHF_{s,t}^\theta(\dd x , \dd y))_{0 \le s \le t <\infty}$ given in Theorem~\ref{th:main0}.

\begin{proof}[Proof Sketch]
The mean of
$\SHF^\theta_{s,t}(\dd x,\dd y)$ follows directly
from the convergence in Theorem \ref{th:main0}, the definition of
$\cZ_{N;\, s,t}$ in \eqref{eq:rescZmeas}, and the local limit theorem in \eqref{eq:llt}.

For illustration, we will identify the kernel
$$
\widetilde K^\theta_t(x, x') \dd x \, \dd x' :=\bbcov[\SHF^\theta_{0,t}(\dd x, \ind), \SHF^\theta_{0,t}(\dd x', \ind)].
$$
By the convergence in Theorem \ref{th:main0}, we see that
\begin{align*}
\widetilde K^\theta_t(x, x') = \lim_{N\to\infty} \bbcov[Z_{tN}^{\beta_N}(\sqrt{N} x), Z_{tN}^{\beta_N}(\sqrt{N} x')],
\end{align*}
where $Z_{tN}^{\beta_N}(z)$ is the point-to-plane partition function with chaos expansion given in \eqref{eq:expan1}.
By this chaos expansion, we note that
\begin{align}
\bbcov[Z_{tN}^{\beta_N}(\sqrt{N} x), Z_{tN}^{\beta_N}(\sqrt{N} x')] & = \!\!\!\!\!\! \sum_{z\in \Z^2, 1\leq n\leq tN} q_n(z-\sqrt{N} x) q_n(z-\sqrt{N} x')
\sigma_N^2 \bbE[Z_{tN-n}(0)^2] \notag \\
& = \sum_{n=1}^{tN} q_{2n}(\sqrt{N}(x'-x)) \sigma_N^2 \bbE[Z_{tN-n}(0)^2], \label{eq:CovZtN}
\end{align}
where $\sigma_N^2$ comes from $\bbE[\xi^2_N(n, z)]$, and we used the fact that the second moment of the point-to-plane partition function starting from $(n, z)$ and terminal time $tN$ is exactly $\bbE[Z_{tN-n}(0)^2]$.

Denote $u:=n/N$. By the same calculations as in the proof of Lemma \ref{L:Z2Asymp}, we have
$$
\bbE[Z_{tN-n}(0)^2] = \Big(\int_0^{t-u} G_\theta(a) {\rm d}a + o(1)\Big) \log N.
$$
Substituting this into \eqref{eq:CovZtN}, using $\sigma_N^2 \sim 4\pi/\log N$, the local limit theorem \eqref{eq:llt}, and a Riemann sum
approximation, we obtain
\begin{equation}\label{def:Ktilde}
\begin{aligned}
\widetilde K^\theta_t(x, x') & = (1+o(1)) \sum_{n=1}^{tN} \frac{1}{4\pi n} e^{-\frac{N |x-x'|^2}{4n}} \frac{4\pi}{\log N} \cdot \log N \int_0^{t-u}  G_\theta(a) {\rm d}a \\
& = 4\pi \iint_{0<u<v<t}  g_{2u}(x'-x) G_\theta(v-u) {\rm d}u\, {\rm d}v.
\end{aligned}
\end{equation}
We refer to \cite[Prop.~3.6]{CSZ23a} for the complete derivation of $K^\theta_{t-s}$, where the constants differ due to the periodicity of  the simple symmetric random walk.
\end{proof}

\section{Higher Moment Bounds} \label{S:highmom}

In this section, we explain how to bound higher moments of the averaged partition function (cf.\ \eqref{eq:rescZmeas})
\begin{equation} \label{eq:rescZmeas2}
\cZ^{\beta_N}_{N; 0, t}(\varphi, \psi) =
\frac{1}{N} \sum_{x, y\in \Z^2} \varphi\Big(\frac{x}{\sqrt{N}}\Big) \psi\Big(\frac{y}{\sqrt{N}}\Big) Z_{0, \lfloor Nt\rfloor}(x, y),
\end{equation}
where $\beta_N$ is chosen in the critical window defined in \eqref{eq:sigma} for some $\theta\in \R$. As explained in Section \ref{S:lind1}
and in particular in \eqref{eq:lind4}, such moment bounds (4th moment will suffice) are needed to apply the Lindeberg principle.
As explained in Section \ref{S:sketch} and \ref{S:lind2}, to prove Theorem \ref{th:main0} on the critical $2d$ stochastic heat flow (SHF),
we will apply a Lindeberg principle for dependent random variables, which requires the moment bound \eqref{eq:Theta4} for the
coarse-grained disorder variables $\Theta_{N, \epsilon}(\vec\sfi, \vec\sfa)$ and the moment bound \eqref{eq:Zcg4} for the coarse-grained
partition function $\mathscr{Z}_{\epsilon}^{(\cg)}(\varphi, \psi | \Theta_{N, \eps})$. We will focus mainly on the derivation of
moment bounds for $\cZ^{\beta_N}_{tN}(\varphi, \psi)$. We will then sketch how this implies the desired moment bound for $\Theta_{N, \epsilon}(\vec\sfi, \vec\sfa)$ and how the same strategy can be used to derive moment bounds for $\mathscr{Z}_{\epsilon}^{(\cg)}(\varphi, \psi | \Theta_{N, \eps})$.
At the end of this section, we will also formulate moment bounds for the critical $2d$ SHF and discuss related results on the $2d$ Delta-Bose gas.

\subsection{Moment bounds and proof strategy}\label{momandmethods}
We now formulate the higher moment bounds for the averaged partition function $\cZ^{\beta_N}_{N; 0, t}(\varphi, \psi)$ defined in
\eqref{eq:rescZmeas2} for some parameter $\theta\in \R$ in the critical window \eqref{eq:sigma}. For consistency with \cite[Theorem 6.1]{CSZ23a}
where such a bound was established, we will let $N'=t N$ and consider the partition averaged on the spatial scale $\sqrt{N'}$ instead
of $\sqrt{N}$. Namely we consider
\begin{equation} \label{eq:rescZmeas3}
\cZ^{\beta_N}_{N'}(\varphi, \psi) =
\frac{1}{N'} \sum_{x, y\in \Z^2} \varphi_{N'}(x)  Z_{0, N'}(x, y) \psi_{N'} (y),
\end{equation}
where for $\phi:\R^2 \to \R$, $\phi_{N'}: \Z^2 \to\R$ is defined by setting $\phi_{N'}(z)$ to be the average of
$\phi$ on the square of volume $1/N'$ centered at $z/\sqrt{N'}$.

\begin{theorem}[Higher moments]\label{th:mom}
Fix $p, q \in (1,\infty)$ with $\frac{1}{p}+\frac{1}{q}=1$, any integer $h\ge 3$, and any
weight function $w: \R^2 \to (0,\infty)$ such that $\log w$ is Lipschitz continuous. Then
there exist $\sfC = \sfC(h), \sfC' = \sfC'(h) < \infty$ such that, uniformly in large $N'\leq N\in \N$ and
locally integrable $\varphi, \psi: \R^2 \to \R$, we have
\begin{align}
	\Big|\bbE\Big[\Big(\cZ_{N'}^{\beta_{N}}(\varphi, \psi) -
	\bbE[\cZ_{N'}^{\beta_{N}}(\varphi, \psi)]\Big)^h\Big]\Big|
	& \,\leq\, \frac{\sfC}{\log (1+\tfrac{N}{N'} )}
	\, \frac{1}{(N')^{h}} \, \Big\Vert \frac{\varphi_{N'}}{w_{N'}}
	\Big\Vert_{\ell^p}^h \, \Vert \psi_{N'}\Vert_{\ell^\infty}^h
	\, \Vert w_{N'} \ind_{B_{N'}} \Vert_{\ell^q}^h  \label{eq:mombd1} \\
	& \,\leq\, \frac{\sfC'}{\log (1+ \tfrac{N}{N'})} \, \Big\Vert \frac{\varphi}{w}
	\Big\Vert_{p}^h \, \Vert \psi\Vert_{\infty}^h
	\, \Vert w \ind_{B} \Vert_{q}^h \,,   \label{eq:mombd2}
\end{align}
where $\Vert \phi\Vert_{\ell^p}:= (\sum_z |\phi(z)|^p)^{1/p}$,
$B\subset \R^2$ is any ball (could be $\R^2$) that contains the support of $\psi$, and $B_{N'} := B \sqrt{N'}$.
\end{theorem}

\begin{remark}
Some remarks are in order:
\begin{itemize}
\item The bound in \eqref{eq:mombd2} follows from \eqref{eq:mombd1} via a simple Riemann sum approximation. For
$\cZ^{\beta_N}_{N; 0, t}(\varphi, \psi)$ defined in \eqref{eq:rescZmeas2} with averaging on the spatial scale $\sqrt{N}$,
the same bound \eqref{eq:mombd2} holds with $t=N'/N$. This can be seen by applying \eqref{eq:mombd1} with $\varphi$ and $\psi$
therein replaced with $\varphi_t(x) :=\varphi(\sqrt{t} x)$ and $\psi_t(x) :=\psi(\sqrt{t} x)$, and using the weight function $w_N$
instead of $w_{N'}$. But there is probably room for improvement since we expect that averaging on larger spatial scales will decrease
the variance and the higher moments of the averaged partition function.

\item Theorem \ref{th:mom} with $N'=\eps N$ will be used later to deduce the moment bound \eqref{eq:Theta4} for the coarse-grained variables $\Theta_{N, \epsilon}(\vec\sfi, \vec\sfa)$.

\item The weight function $w$ in \eqref{eq:mombd1} and \eqref{eq:mombd2} allows us to include the case $\psi\equiv 1$, and to
control the spatial decay when $\Vert \varphi\Vert_p<\infty$ and the support of $\psi$ moves away to $\infty$, which is needed for the bounds in \eqref{eq:lind4} and \eqref{eq:Zcg4}.

\item The bound \eqref{eq:mombd1} implies the moment bound for the point-to-plane partition function $Z_{0, N'}(z, \ind)$ in \eqref{eq:lind4}.
Indeed, assume $z=0$ and let $w(x)=e^{-\Vert x\Vert}$,  $\psi\equiv 1$, and choose $\varphi$ such that $\varphi_{N'}(y)=N'\cdot \ind_{y=0}$ and
$\cZ_{N'}^{\beta_{N}}(\varphi, \ind) = Z_{0, N'}(0, \ind)$. Then \eqref{eq:mombd1} gives
\begin{equation}\label{eq:p2plane}
\big|\bbE\big[\big(
	Z_{0, N'}(0, \ind) - 1 \big)^h\big]\big|
	\leq C \Vert w_{N'}\Vert_{\ell^q}^h \leq C' \, N^{\frac{h}{q}} \,,
\end{equation}
where the exponent can be made arbitrarily small by choosing $q$ large.

\item
Theorem~\ref{th:mom} can be strengthened by showing that the constants $\sfC$ and $\sfC'$
in \eqref{eq:mombd1} and \eqref{eq:mombd2} are proportional to $pq$, which allows one to
send $p=p_N\downarrow 1$ and $q=q_N\uparrow \infty$. This was done in the subcritical
regime in \cite[Theorem 1.5]{LZ23} and in the quasi-critical regime in \cite{CCR25} (cf.\ Section \ref{sec:quasi-critical}).
There is in fact a unified bound for all disorder strength $\beta_N$ in or below
the critical window, which says that the $h$-th centred moment of the averaged partition
function $\cZ_N^{\beta_N}(\varphi) :=\cZ_N^{\beta_N}(\varphi, \ind)$ can be controlled by its variance raised to the power
$h/2$, provided $\varphi$ is sufficiently regular and is supported on an appropriate spatial
scale such that \emph{the variance remains bounded} (we considered flat terminal condition
$\psi = \ind$ for simplicity): for some $C = C(h) < \infty$
\begin{equation*}
	\Big|\bbE\Big[\Big(\cZ_{N}^{\beta_N}(\varphi) -
	\bbE[\cZ_{N'}^{\beta_N}(\varphi)]\Big)^h\Big]\Big|
	\le C \, \bbvar [\cZ_{N}^{\beta_N}(\varphi) ]^{\frac{h}{2}}.
\end{equation*}
The details can be found in the forthcoming paper \cite{CSZ25}.
Such a bound in the quasi-critical regime can be found in \cite[Proposition~2.2]{CCR25}.
\end{itemize}
\end{remark}

\begin{proof}[Proof Sketch]
The proof of Theorem \ref{th:mom} is based on the analysis of collision diagrams of $h$ independent random walks (see
Figure \ref{figure-fourth-moment}). Let us first explain at the heuristic level before formalising it in terms of operators.
For simplicity, we will assume $N'=N$.

Recalling from \eqref{eq:Zppchaos} the polynomial chaos expansion for $\cZ_{N}(\varphi, \psi):= \cZ_{N}^{\beta_{ N}}(\varphi, \psi)$,
we can write
\begin{align} \label{Exp_h}
   &\big|\bbE\big[\big(\cZ_{N}(\varphi, \psi) -
	\bbE[\cZ_{N}(\varphi, \psi)]\big)^h\big]\big|
	 =  \frac{1}{N^h} \times    \\
 & \bbE \Bigg [ \bigg(\sum_{k=1}^\infty \!\!\!\!\! \sumtwo{x_1,\cdots, x_k\in \Z^2}{0<n_1<\dots<n_k< N} \!\!\!\!\!\!\!\!\!\!
 q_{n_1}(\varphi_N, x_1)\xi_N(n_1, x_1) \bigg\{ \prod_{j=2}^k q_{n_j-n_{j-1}}(x_j- x_{j-1}) \xi_N(n_j, x_j) \bigg\} \notag
 q_{N-n_k}(x_k,\psi_N)\bigg)^h \Bigg],
\end{align}
where $\varphi_N(x):=\varphi(\frac{x}{\sqrt{N}})$, $\psi_N(x):=\psi(\frac{x}{\sqrt{N}})$,
\begin{align}
q_m(\varphi_N, z) := \sum_{x\in \Z^2} \varphi_N(x) q_m(z-x) \quad \mbox{and} \quad q_m(z, \psi_N) := \sum_{y\in \Z^2} q_m(y-z) \psi_N(y).
\end{align}

We can expand the power inside the expectation in \eqref{Exp_h} as
\begin{align} \label{h_power_exp}
    &\sum_{k_1,\dots,k_h\geq 1} \sumtwo{(n^{(r)}_i,x^{(r)}_i) \in \, \N \times \Z^2\, }{1\leq i \leq k_r, 1\leq r \leq h} \prod_{r=1}^h q_{n^{(r)}_1}(\varphi_N,x^{(r)}_1) \Big\{ \prod_{j=2}^{k_r} q_{n^{(r)}_j-n^{(r)}_{j-1}}(x^{(r)}_j-x^{(r)}_{j-1}) \Big\} q_{N-n^{(r)}_{k_r}}(x^{(r)}_{k_r},\psi_N) \notag \\
    &\hspace{4cm} \times\ \bbE\Big[\prod_{r=1}^h \, \prod_{j=1}^{k_r} \, \xi_N(n^{(r)}_j,x^{(r)}_j)\Big],
\end{align}
which sums over $h$ sequences of time-space points $\Gamma^{(r)} :=(n_i^{(r)}, x_i^{(r)})_{1\leq i \leq k_r}$ with $0<n^{(r)}_1<\cdots < n^{(r)}_{k_r}<N$,
and consecutive points in each sequence are connected by the random walk transition kernel $q$. Since $\xi_N$ is an i.i.d.\ family with
zero mean, the factor $\bbE[\prod \xi_N(\cdot, \cdot)]$ is non-zero only when each $\xi_N(n,x)$ that appears in the product appears at least twice.
It was shown in \cite{CSZ23a} that the dominant contribution (as $N\to\infty$) to the expansion in \eqref{h_power_exp} comes from configurations of $(\Gamma^{(r)})_{1\leq r\leq h}$  where $\xi_N(n^{(r)}_j, x^{(r)}_j)$ are matched in pairs, with the expectation of each matched pair contributing a factor of $\sigma_N^2$ as in \eqref{eq:xi}. In other words, any $\xi_N(n,x)$ that is collected by one of the $h$ sequences $\Gamma^{(r)}$ is
collected by exactly 2 of the $h$ sequences, and we say there are only pair collisions among $(\Gamma^{(r)})_{1\leq r\leq h}$.
We will focus on such configurations of $\Gamma^{(r)}$, $1\leq r\leq h$, and keep track of the matchings between points in $\Gamma^{(r)}$. In particular, we will partition the sum in \eqref{h_power_exp} according to streaks of consecutive matching between
points in the same pair of sequences $\Gamma^{(i)}$ and $\Gamma^{(j)}$.

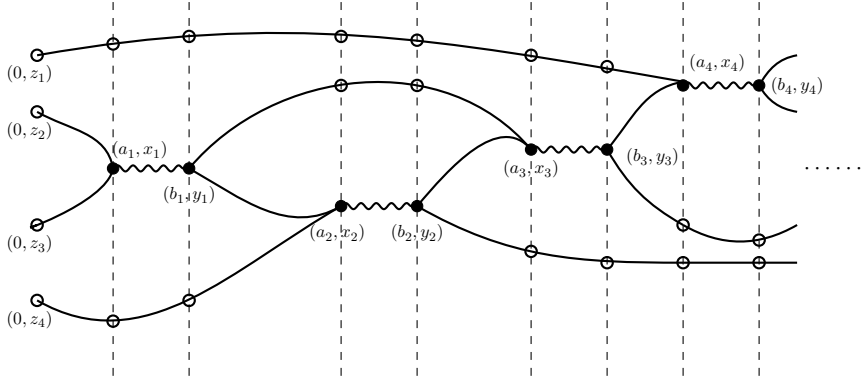
\begin{figure}
\hskip -0.2cm
\begin{tikzpicture}[scale=0.5]
\draw[dashed] (-6,-5)--(-6,5); \draw[dashed] (-4,-5)--(-4,5); \draw[dashed] (0,-5)--(0,5); \draw[dashed] (2,-5)--(2,5); \draw[dashed] (5,-5)--(5,5);
\draw[dashed] (7,-5)--(7,5); \draw[dashed] (9,-5)--(9,5);  \draw[dashed] (11,-5)--(11,5);
\draw  [thick] (-6, -3.55)  circle [radius=0.15]; \draw  [thick] (-4, -3)  circle [radius=0.15]; \draw  [thick] (0, 2.7)  circle [radius=0.15];
\draw  [thick] (2, 2.7)  circle [radius=0.15]; \draw  [thick] (5, -1.7)  circle [radius=0.15];
\draw  [thick] (7, -2)  circle [radius=0.15]; \draw  [fill] (9, 2.7)  circle [radius=0.15];
\draw  [thick] (9, -2)  circle [radius=0.15];
\draw  [thick] (-8,3.5) circle [radius=0.15];
\node at (-8.2,3.0) {\scalebox{0.6}{$(0,z_1)$}};
\draw  [thick] (-6, 3.8)  circle [radius=0.15];
\draw  [thick] (-4, 4.0)  circle [radius=0.15];
\draw  [thick] (2, 3.9)  circle [radius=0.15];
\draw  [thick] (0, 4.0)  circle [radius=0.15];
\draw  [thick] (5, 3.5)  circle [radius=0.15];
\draw  [thick] (7, 3.2)  circle [radius=0.15];
\draw  [fill] (0, -0.5)  circle [radius=0.15]; \node at (-0.1,-1.2) {\scalebox{0.6}{$(a_2,x_2)$}};
\draw [-,thick, decorate, decoration={snake,amplitude=.4mm,segment length=2mm}] (0,-0.5) -- (2,-0.5);
\draw  [fill] (2, -0.5)  circle [radius=0.15]; \node at (2,-1.2) {\scalebox{0.6}{$(b_2,y_2)$}};
\draw[thick] (2,-0.5) to [out=50,in=130] (5,1);
\draw  [fill] (5, 1)  circle [radius=0.15]; \node at (5,0.5) {\scalebox{0.6}{$(a_3,x_3)$}};
 \draw [-,thick, decorate, decoration={snake,amplitude=.4mm,segment length=2mm}] (5,1) -- (7,1);
\draw  [fill] (7, 1)  circle [radius=0.15]; \node at (8.2,0.8) {\scalebox{0.6}{$(b_3,y_3)$}};
\draw[thick] (-4,0.5) to [out=50,in=130] (5,1);
\draw  [fill] (-4,0.5)  circle [radius=0.15]; \node at (-4,-0.2) {\scalebox{0.6}{$(b_1,y_1)$}};
 \draw [-,thick, decorate, decoration={snake,amplitude=.4mm,segment length=2mm}] (-6,0.5) -- (-4,0.5);
 \draw[thick] (-8,-1) to [out=200,in=260] (-6,0.5);
  \draw  [thick] (-8,-1)  circle [radius=0.15];  \node at (-8.2,1.5) {\scalebox{0.6}{$(0,z_2)$}};
   \draw[thick] (-8,2) to [out=-30,in=100] (-6,0.5);
   \draw  [thick] (-8,2)  circle [radius=0.15];  \node at (-8.2,-1.5) {\scalebox{0.6}{$(0,z_3)$}};
 \draw  [fill] (-6,0.5)  circle [radius=0.15]; \node at (-5.3,1) {\scalebox{0.6}{$(a_1,x_1)$}};
 \draw[thick] (-4,0.5) to [out=-30,in=210] (0,-0.5);
  \draw[thick] (-8,-3) to [out=-30,in=210] (0,-0.5); \node at (-8.2,-3.5) {\scalebox{0.6}{$(0,z_4)$}};
   \draw  [thick] (-8,-3)  circle [radius=0.15];
  \draw[thick] (2,-0.5) to [out=-30,in=180] (9,-2);
   \draw  [thick] (9,-1)  circle [radius=0.15];  \node at (9.9,3.3) {\scalebox{0.6}{$(a_4,x_4)$}};
    \draw  [thick] (11,-1.4)  circle [radius=0.15];  \draw  [thick] (11,-2)  circle [radius=0.15];
    \draw[thick] (7,1) to [out=-60,in=150] (9,-1);

    \draw [-,thick, decorate, decoration={snake,amplitude=.4mm,segment length=2mm}] (9,2.7) -- (11,2.7);
    \draw  [fill] (11,2.7)  circle [radius=0.15];  \node at (12,2.7) {\scalebox{0.6}{$(b_4,y_4)$}};
    \draw[thick] (11,2.7) to [out=60,in=190] (12,3.5);
    \draw[thick] (11,2.7) to [out=-60,in=-190] (12, 2);
    \draw[thick] (7,1) to [out=50,in=190] (9,2.8);
    \draw[thick] (-8,3.5) to [out=10,in=170] (9,2.8);
     \draw[thick] (9,-1) to [out=-30,in=-150] (12,-1);  \draw[thick] (9,-2)--(12,-2);
     \node at (13,0.5) {\scalebox{0.8}{$\cdots\cdots$}};
\end{tikzpicture}
\caption{Illustration of the expansion in \eqref{h_power_exp} for the fourth moment. Each wiggle line represents
a streak of consecutive matchings of time-space points in the same pair of sequences $\Gamma^{(i)}$ and $\Gamma^{(j)}$, which
is assigned weight $\sigma_N^2 U_N(b-a, y-x)$ (see \eqref{U-diagram}) if the streak starts at $(a, x)$ and ends at $(b, y)$.
Each solid curve between two circles is assigned a random walk transition kernel $q$, while each hollow circle at a point
$(t, y)$ arises from a sequence $\Gamma^{(r)}$ that is unmatched at time $t$ and the Chapman-Kolmogorov decomposition of the
associated transition kernel $q_{n^{(r)}_{j-1}, n^{(r)}_j}(x^{(r)}_{j-1}, x^{(r)}_{j})
= \sum_y q_{n^{(r)}_{j-1}, t}(x^{(r)}_{j-1}, y) q_{t, n^{(r)}_j}(y, x^{(r)}_{j})$ at the intermediate time $t$. In each time
interval between consecutive dashed vertical lines, the product of
the curve weights define a kernel which we call Replica Evolution (see \eqref{eq:sfG}) if it contains a wiggle line, or Constrained
Evolution (see \eqref{eq:sfQ}) if there is no wiggle line in this time interval.}
\label{figure-fourth-moment}
\end{figure}

As illustrated in Figure \ref{figure-fourth-moment}, we will rewrite the expansion \eqref{h_power_exp} of the $h$-th moment by first identifying
the time intervals $[a_\ell, b_\ell]$, $1\leq \ell\leq m$, where the streaks of consecutive matching occurs (the wiggle lines
in Figure \ref{figure-fourth-moment}). Denote by $I_\ell=\{i_\ell, j_\ell\}$ the indices for the pair of $\Gamma^{(r)}$ for
which matching occurs in the $\ell$-th streak. Our decomposition ensures that $I_\ell \neq I_{\ell+1}$. Summing over $m$,
$a_\ell, b_\ell$, $I_\ell$, and the spatial locations of the solid and hollow circles at times $a_\ell$ and $b_\ell$ in
Figure \ref{figure-fourth-moment}, we can then rewrite the expansion \eqref{h_power_exp} using kernels:

\begin{align}\label{discrete-mom}
       & \, \frac{1}{N^h} \sum_{m=1}^\infty (\sigma_N^2)^m \!\!\!\!\!
        \sum_{\substack{  I_1,..., I_m \subset \{1,\ldots,h\} \\
        |I_\ell|=2, \, I_\ell \neq I_{\ell+1}}}  \,\,\,
        \sum_{\substack{0<a_1\leq b_1<\cdots <a_m\leq b_m<N \\ \bx_1,...,\bx_m \,,\, \by_1,...,\by_m \in (\Z^2)^{h}}}
         \!\!\!\!\!\!\!\!\!\! \sfQ_{a_1}^{*,I_1}(\varphi_N, \bx_1) \sfU_N^{I_1}(b_1-a_1, \by_1-\bx_1) \notag  \\
        &  \times \prod_{\ell=2}^{m} \sfQ_{a_\ell-b_{\ell-1}}^{I_{\ell-1}, I_\ell }(\by_{\ell-1}, \bx_\ell)
	\, \sfU_{N}^{I_\ell}( b_\ell-a_\ell, \bx_\ell, \by_\ell)
	\,  \sfQ_{N-b_{m}}^{I_m,*}(\by_{m}, \psi_N) +o(1),
\end{align}
where $o(1)$ accounts for the contribution from configurations of $(\Gamma^{(r)})_{1\leq r\leq h}$ that contain triple
collisions, that is, there exists $(n, x)$ that belongs to more than two different $\Gamma^{(r)}$.

The kernels $\sfU_N^{I_\ell}, \sfQ_{a_1}^{*,I_1}, \sfQ_{a_\ell-b_{\ell-1}}^{I_{\ell-1}, I_\ell }, \sfQ_{N-b_{m}}^{I_m,*}$ are defined as follows:
\begin{itemize}
\item[{\bf 1.  Replica Evolution}.] Recall $U_N$ from \eqref{eq:U}. For $I=\{i<j\}\subset \{1,...,h\}$ and $\bx=(x_a)_{1\leq a\leq h}, \by=(y_a)_{1\leq a\leq h}\in (\Z^2)^h$, we define
\begin{align} \label{eq:sfG}
	\sfU_N^I(n,\bx, \by)
	\,:= \ind_{\{x_i=x_j\}} \cdot \, \Big( \prod_{a \in \{1,\ldots, h\} \setminus \{i,j\}}
	\!\!\!\!\!\!\!\! q_{n} (y_a - x_a)  \Big) \,
          U_N (n, y_i - x_i) \, \cdot \ind_{\{y_i=y_j\}} \,.
\end{align}
\item[{\bf 2. Constrained Evolution.}]
For $I=\{i<j\}, J=\{k<\ell\} \subset \{1,...,h\}$ and $\bx, \by\in (\Z^2)^h$, we define
\begin{align}\label{eq:sfQ}
	\sfQ_{n}^{I,J}(\by,  \bx)
	\,:=\, \ind_{\{y_i=y_j\}}  \cdot \Bigg( \prod_{a=1}^h \,
	q_{n} (x_a - y_a)  \Bigg)
	\cdot \ind_{\{x_k=x_l\}} \,.
\end{align}
Removing the spatial constraint on either side, we define
\begin{align}\label{eq:sfQ*}
	\sfQ_{n}^{*, J}(\by,  \bx)
	\,:=\, \prod_{a=1}^h \,
	q_{n} (x_a - y_a)\
	\cdot \ind_{\{x_k=x_l\}}, \quad
	\sfQ_{n}^{I,*}(\by,  \bx)
	\,:=\,  \ind_{\{y_i=y_j\}} \prod_{a=1}^h \,
	q_{n} (x_a - y_a),
\end{align}
while
\begin{align}\label{eq:sfQ**}
\sfQ^{*,J}_n(\varphi_N, \bx)
:=\sum_{\bz \in (\Z^2)^h} \Big(\prod_{i=1}^n \varphi\big(\tfrac{z_i}{\sqrt{N}}\big)\Big) \sfQ^{*,J}_n(\bz, \bx),
\end{align}
and similarly for $\sfQ^{I,*}_n(\by, \psi_N)$.
\end{itemize}
To decouple the summation range for the time increments $a_1, b_1-a_1, a_2-b_1, \ldots$ in \eqref{discrete-mom}, we consider the Laplace transforms
of the kernels above:
\begin{equation}\label{eq:QU}
\begin{aligned}
	& \sfQ_{\lambda,N}^{I, J} (\by, \bz)
	:= \sum_{n=1}^{2N} e^{-\frac{\lambda}{N} n} \, \sfQ^{I, J}_n(\by, \bz),   \qquad  &\by, \bz \in (\Z^2)^h , \\
	& \sfU_{\lambda, N}^{J}(\by, \bz)
	:= \sum_{n=0}^{2N} e^{-\frac{\lambda}{N} n} \,
	\sfU^J_{ N} (n, \by, \bz),  &\by, \bz \in (\Z^2)^h.
\end{aligned}
\end{equation}
Inserting the factor $e^{\lambda} e^{-\frac{\lambda}{N} \sum_{i=1}^{m-1}((a_i-b_{i-1})+(b_{i+1}-a_i )) - \frac{\lambda}{N} (N-b_m)}=1$
in \eqref{discrete-mom} and enlarging the range of summation for $a_1, a_2-b_1, \ldots$ to $\{1, \ldots, 2N\}$ and
the range of summation
for $b_i-a_i$ to $\{0, \ldots, 2N\}$, we obtain the upper bound
\begin{align}\label{Laplace-mom}
&\frac{1}{N^{h+1} }\sum_{m=1}^\infty (\sigma_N^2)^m \!\!\!\!\! \!\!\!\!\!\!\!\!
        \sum_{\substack{  I_1,..., I_m \subset \{1,\ldots,h\} \\
        |I_\ell|=2, \, I_\ell \neq I_{\ell+1} \\
        \bx_1,...,\bx_m \,,\, \by_1,...,\by_m \in (\Z^2)^{h} }}
        \sfQ_{\lambda, N}^{*,I_1}(\varphi_N, \bx_1) \sfU^{I_1}_{\lambda, N}(b_1-a_1, \by_1-\bx_1) \notag \\
& \qquad \qquad \qquad \times \prod_{\ell=2}^{m} \sfQ_{\lambda, N}^{I_{\ell-1}, I_\ell }(\by_{\ell-1}, \bx_\ell)
	\, \sfU_{\lambda, N}^{I_\ell}(  \bx_\ell, \by_\ell)
	\times \,  \sfQ_{\lambda, N}^{I_m,*}(\by_{m}, \psi_N),
\end{align}
where there is an additional factor $N^{-1}$ because to introduce the last kernel $\sfQ_{\lambda, N}^{I_m,*}$, we need to sum over
$a_{m+1}-b_m \in \{1, \ldots, 2N\}$, instead of $a_{m+1}=N$ as in \eqref{discrete-mom}. The additional sum over $a_{m+1}$ is compensated
by the averaging factor $N^{-1}$, and the fact this averaging over $a_{m+1}$ is comparable to having fixed $a_{m+1}=N$ is justified by
the following inequality (see \cite[(6.11)]{CSZ23a})
\begin{equation*}
    \sfQ^{I_m,*}_{N-b_m} \big(\by_m,\bz' \big)
    \leq \frac{c}{N}\, \sum_{a_{m+1} \in \{N+1,\dots,2N\}} \sfQ^{I_m,*}_{a_{m+1}-b_m}\big(\by_m,\bz' \big).
\end{equation*}

\begin{remark}
Instead of averaging over the additional variable $a_{m+1}$,
one could also take the \emph{maximum} over $b_m$ of the last random walk operator
$\sfQ_{N-b_{m}}^{I_m,*}(\by_{m}, \psi_N)$ (which requires no assumption on $\psi_N$).
This altenative approach, which has been developed in \cite{CCR25}, leads to a slightly
sharper version of the bound \eqref{eq:mombd2}. The key point is that
the maximal random walk operator described above
is bounded in $L^q$ for $q > 1$, thanks to
a variant of the Hardy-Littlewood maximal inequality, see \cite[Lemma~4.18]{CCR25}.
\end{remark}

The kernels $\sfU_{\lambda, N}^{I_\ell}$ and $\sfQ_{\lambda, N}^{*,I_1}, \sfQ_{\lambda, N}^{I_{\ell-1}, I_\ell }, \sfQ_{\lambda, N}^{I_m,*}$
define integral operators and we can bound \eqref{Laplace-mom} in terms of their operator norms. More precisely, the indicator constraints
in \eqref{eq:sfG} and \eqref{eq:sfQ} suggest that we should regard $\sfU_{\lambda, N}^{I}$, with $I=\{i<j\}$, as an integral operator acting on functions
defined on the following subset of $(\Z^2)^h$:
\begin{equation} \label{eq:ZhI}
(\Z^2)^h_I:=\{ \bx\in (\Z^2)^h: x_i=x_j\}.
\end{equation}
Let $\ell^q((\Z^2)^h_I)$ be the space of functions $f: (\Z^2)^h_I \to \R$ with $\Vert f\Vert_q:= ( \sum_{x\in (\Z^2)^h_I} |f(x)|^q)^{1/q}<\infty$. For $q>1$, we regard $\sfU_{\lambda, N}^{I}$ as an operator from $\ell^q((\Z^2)^h_I)$ to $\ell^q((\Z^2)^h_I)$ with operator norm
\begin{align*}
\| \sfU_{\lambda, N}^{I} \|_q := \sup_{f, g: (\Z^2)^h_I \to \R \atop \|f\|_p, \|g\|_q=1} \langle f, \sfU_{\lambda, N}^{I} g\rangle,
\end{align*}
where $\langle \cdot, \cdot\rangle$ is the inner product and $\frac{1}{p}+\frac{1}{q}=1$.

For $I=\{i <j\}$ and $J=\{k <l\}$, we will regard $\sfQ_{\lambda, N}^{I, J}$ as an integral operator from $\ell^q((\Z^2)^h_J)$ to $\ell^q((\Z^2)^h_I)$ with operator norm $\Vert \sfQ_{\lambda, N}^{I, J}\Vert_q$ defined as before. Similarly, $\sfQ_{\lambda, N}^{*, I}$ will be an integral operator from
$\ell^q((\Z^2)^h_I)$ to $\ell^q((\Z^2)^h)$, and $\sfQ_{\lambda, N}^{I,*}$ an integral operator from $\ell^q((\Z^2)^h)$ to $\ell^q((\Z^2)^h_I)$.
Using the norms of these operators for a fixed $q>1$, we can now bound \eqref{Laplace-mom} by
\begin{align}\label{geom-op}
& \frac{C e^{\lambda }}{N^{h+1}} \! \sum_{m=1}^\infty (\sigma_N^2)^m \!\!\!\! \sum_{I_1,...,I_m} \!\!\!\!
       \big\langle \varphi_N^{\otimes h}, \sfQ_{\lambda, N}^{*,I_1} \, \sfU_{\lambda, N}^{I_1} \, \sfQ_{\lambda, N}^{I_1,I_2}
       \cdots \sfU_{\lambda, N}^{I_m} \, \sfQ_{\lambda, N}^{I_m,*} \, \psi_N^{\otimes h}\big \rangle \notag \\
\leq & \frac{C e^{\lambda }}{N^{h+1}} \! \sum_{m=1}^\infty (\sigma_N^2)^m \!\!\!\! \sum_{I_1,...,I_m} \!\!\!\!
        \|\varphi_N^{\otimes h} \|_p \| \sfQ_{\lambda, N}^{*,I_1}\|_q \, \|\sfU_{\lambda, N}^{I_1}\|_q \,\| \sfQ_{\lambda, N}^{I_1,I_2}\|_q
       \cdots \|\sfU_{\lambda, N}^{I_m}\|_q \, \|\sfQ_{\lambda, N}^{I_m,*}\|_q \, \|\psi_N^{\otimes h}\|_q ,
\end{align}
where $\phi^{\otimes h}(\bx):= \prod_{i=1}^h \phi(x_i)$. The key is to bound these operator norms, which is summarised
in the following proposition.

\begin{proposition}\label{prop:opnorm}
Fix an integer $h\ge 2$ and $p, q>1$ with $\frac{1}{p}+\frac{1}{q}=1$.
There exists $C = C(h) < \infty$ such that uniformly in $\lambda>0$, $N$ large, and $I\neq J \subset \{1,...,h\}$ with $|I|=|J|=2$,
we have
\begin{gather}
	\big\Vert  \sfQ^{I, J}_{\lambda, N} \big\Vert_q
	\leq C \, pq \,;\label{eq:norm1} \\
	\big\Vert  \sfQ^{*,I}_{\lambda,N} \big\Vert_q
	\leq C\, N^{\frac{1}{q}} \,,
	\qquad \big\Vert  \sfQ^{I,*}_{\lambda,N} \big\Vert_q \leq
	C\, N^{\frac{1}{p}} \,;
	\label{eq:norm2} \\
	\text{and } \quad \quad \big\Vert \sfU^{I}_{\lambda, N} \big\Vert_q \leq \frac{C}{(\log \lambda )
	\, \sigma_{ N}^2} .
	\label{eq:norm3}
\end{gather}
\end{proposition}
Substituting these bounds into \eqref{geom-op} leads to a geometric series that is convergent provided $\lambda>0$ is chosen
large enough. To arrive at the final form in Theorem \ref{th:mom}, we need two modifications:
 \begin{itemize}
 \item Firstly, we need to modify \eqref{geom-op} by introducing a weight function $w_N^{\otimes h}(\bx)$ with $w_n(x)=w(x/\sqrt{N})$.
  This is done by rewriting \eqref{Laplace-mom} in terms of the weighted operators
\begin{equation}\label{eq:QUw}
\begin{aligned}
	&\widehat \sfQ_{\lambda,N}^{I, J} (\by, \bz):= \sfQ_{\lambda,N}^{I, J} (\by, \bz) \frac{w_N^{\otimes h}(\by)}{w_N^{\otimes h}(\bz)} \quad \text{and} \quad
	\widehat \sfU_{\lambda, N}^{J}(\by, \bz):=   \sfU^J_{ N} (\by, \bz)\frac{w_N^{\otimes h}(\by)}{w_N^{\otimes h}(\bz)}.
\end{aligned}
\end{equation}
This allows us to replace the bound in \eqref{geom-op} by
\begin{align*}
&\Big|\bbE\Big[\Big(\cZ_{N}^{\beta_{ N}}(\varphi, \psi) -
	\bbE[\cZ_{N}^{\beta_{N}}(\varphi, \psi)]\Big)^h\Big]\Big|  \notag\\
& \leq \,\,\, \frac{C e^{\lambda }}{N^{h+1}}\sum_{m\geq 1} (\sigma_N^2)^m\sum_{I_1,...,I_m}
       \Big\langle \frac{\varphi_N^{\otimes h}}{w_N^{\otimes h}}, \widehat \sfQ_{\lambda, N}^{*,I_1} \,\widehat U_{\lambda, N}^{I_1} \, \widehat \sfQ_{\lambda, N}^{I_1,I_2}
       \cdots \widehat U_{\lambda, N}^{I_m} \, \widehat \sfQ_{\lambda, N}^{I_m,*} \, w_N^{\otimes h} \psi_N^{\otimes h}\Big \rangle \notag \\
&  \leq \,\,\, \frac{C e^{\lambda }}{N^{h+1}}\sum_{m\geq 1} (\sigma_N^2)^m\sum_{I_1,...,I_m}
        \Big\|\frac{\varphi_N^{\otimes h}}{w_N^{\otimes h}} \Big \|_p \| \widehat \sfQ_{\lambda, N}^{*,I_1}\|_q \, \| \widehat \sfU_{\lambda, N}^{I_1}\|_q \,\| \widehat \sfQ_{\lambda, N}^{I_1,I_2}\|_q
       \cdots \|\widehat \sfU_{\lambda, N}^{I_m}\|_q \, \| \widehat \sfQ_{\lambda, N}^{I_m,*}\|_q \, \| w_N^{\otimes h} \psi_N^{\otimes h} \|_q.
\end{align*}
The bounds in Proposition \ref{prop:opnorm} remain valid for $\widehat \sfQ_{\lambda, N}$ and $\widehat \sfU_{\lambda, N}$, see \cite[Prop.~6.6]{CSZ23a}.

\item  Secondly, to obtain the pre-factor $1/\log(1+\tfrac{N}{N'})$ in Proposition \ref{th:mom} when $N'<N$, we need to replace the operator
$\widehat \sfU_{\lambda, N}^{I}$ by a variant that takes into account the shorter time horizon, even though $\beta_N$ and $\sigma_N^2$ remain
the same. For more details, see \cite[Prop.~6.6 and (6.24)]{CSZ23a}.
\end{itemize}
\end{proof}

\subsection{Functional inequalities}
In this subsection, we sketch how to obtain the bounds in Proposition  \ref{prop:opnorm}. We note that
$L^2$-variants of \eqref{eq:norm1} first appeared in the work of Dell'Antonio, Figari and Teta \cite{DFT94}
(see the proof of Lemma 3.1 therein) and was then used in \cite[(5.2)]{GQT21} to bound the moments of
the mollified SHE \eqref{eq:mollSHE} in the critical window. Since it is non-trivial to adapt the
Fourier transform techniques of \cite{DFT94} to our discrete setting, we instead work directly
in real space. Our method is robust enough to be applied to the coarse-grained model introduced in Section
\ref{S:CG}. Our extension from $L^2$ bounds to $L^q$ bounds for arbitrary $q>1$ is also important as discussed in
\eqref{eq:p2plane}.

The inequality \eqref{eq:norm1} is reminiscent of the Hardy-Littlewood-Sobolev (HLS) inequality \cite[Theorem 4.3]{LL01}:
for $f, g: \R^d \to \R$, $p, r>1$ and $0<\nu<d$ satisfying $\frac{1}{p} + \frac{1}{r}+\frac{\nu}{d}=2$,
\begin{equation} \label{HLS}
\Big| \iint_{\R^d \times \R^d} \frac{f(x) g(y)}{|x-y|^\nu} {\rm d}x {\rm d}y \Big| \leq C(d, \nu, p) \|f\|_p \|g\|_r.
\end{equation}
Indeed, \eqref{eq:norm1} is equivalent to showing that for all $f\in \ell^p((\Z^2)_I^h)$ and $g\in \ell^q((\Z^2)_J^h)$
\begin{equation} \label{eq:HLS2}
\sum_{\bx \in (\Z^2)_I^h, \by \in (\Z^2)_J^h} f(\bx) \sfQ^{I,J}_{\lambda,N}(\bx, \by) g(\by) \leq C pq \Vert f\Vert_p \Vert g\Vert_q,
\end{equation}
where for $\bx \in (\Z^2)_I^h$ and $\by \in (\Z^2)_J^h$,  we have
\begin{equation}\label{eq:QlambN}
\sfQ^{I,J}_{\lambda,N}(\bx, \by) = \sum_{n=1}^{2N} e^{-\frac{\lambda}{N} n} \prod_{i=1}^h q_n(y_i-x_i)
\leq
\left\{
\begin{aligned}
\frac{C}{(1+|\by-\bx|^2)^{h-1}} \quad & \mbox{for   } |\bx-\by| \leq \sqrt{N}, \\
\frac{C}{N^{h-1}} e^{-\frac{|\by-\bx|^2}{CN}} \quad & \mbox{for   } |\bx-\by|> \sqrt{N},
\end{aligned}
\right.
\end{equation}
with a lower bound of the same order. In other words, for $|\bx-\by| \leq \sqrt{N}$, $\sfQ^{I,J}_{\lambda,N}(\bx, \by)$
is comparable to the Green's function of a random walk in $\Z^{2h}$. Note that  in \eqref{eq:HLS2}, the variables $\bx$ and $\by$
are being summed in a space of dimension $2(h-1)$ which corresponds to dimension $d$ in \eqref{HLS}, $r=q$, while at first glance,
$\nu$ in \eqref{HLS} is comparable to the exponent $2(h-1)$ in $\sfQ^{I,J}_{\lambda,N}(\bx, \by)$. This will be exactly the borderline
case where the HLS inequality fails because $\nu=d$. But what saves us is that when we restrict to $\bx \in (\Z^2)_I^h$ and
$\by \in (\Z^2)_J^h$, we have made the identifications $x_i=x_j$ and $y_k=y_l$, which makes $(1+|\by-\bx|^2)^{1-h}$ more regular
because $|\by-\bx|$ is greater than the Euclidean distance between $\bx \in (\Z^2)_I^h$ and $\by \in (\Z^2)_J^h$ regarded as two points
in $\Z^{2(h-1)}$. Indeed, $|\bx-\by|^2$ contains terms such as $|x_i-y_i|^2 + |x_j-y_j|^2 \geq \frac{1}{3} (|x_i-y_i|^2 + |y_i-y_j|^2)$,  which
penalises the summation in $y_i$ and $y_j$ in \eqref{eq:HLS2} when they are far apart.

A naive attempt to bound \eqref{eq:HLS2} is to write $\sfQ^{I,J}_{\lambda,N}(\bx, \by) = \sfQ^{I,J}_{\lambda,N}(\bx, \by)^{\frac{1}{p}}  \sfQ^{I,J}_{\lambda,N}(\bx, \by)^{\frac{1}{q}}$ and then apply H\"older's inequality to bound
\begin{align}\label{eq:naive}
\langle f,  \sfQ^{I,J}_{\lambda,N} g\rangle \leq \Big(\sum_{\bx \in (\Z^2)_I^h, \by \in (\Z^2)_J^h} |f(\bx)|^p \sfQ^{I,J}_{\lambda,N}(\bx, \by) \Big)^{1/p}
\Big(\sum_{\bx \in (\Z^2)_I^h, \by \in (\Z^2)_J^h}  \sfQ^{I,J}_{\lambda,N}(\bx, \by) |g(\by)|^q \Big)^{1/q}.
\end{align}
This will not work because it can be shown that
\begin{align}\label{failHLS}
\sum_{\by \in (\Z^2)_J^h} \sfQ_{\lambda, N}^{I,J}(\bx, \by) \geq C \log |x_i-x_j|.
\end{align}
The trick is to insert a factor $\frac{\gamma(\bx, \by)}{\gamma(\bx, \by)}$ and then apply H\"older such that the factor $\gamma$ or $1/\gamma$
will remove the logarithmic singularity. More precisely, we choose a small $\alpha>0$ and bound
\begin{align}
&\sum_{\bx \in (\Z^2)_I^h, \by \in (\Z^2)_J^h}  \!\!\!\!\!\! f(\bx) \sfQ^{I,J}_{\lambda,N}(\bx, \by) g(\by)
=  \sum_{\bx, \by}f(\bx) \sfQ^{I,J}_{\lambda,N}(\bx, \by)
       \frac{|x_k-x_\ell |^\alpha}{|y_i-y_j|^\alpha} \cdot \frac{|y_i-y_j |^\alpha}{|x_k-x_\ell|^\alpha} g(\by)  \notag \\
&\leq \Bigg( \sum_{\bx, \by} f(\bx)^p \sfQ^{I,J}_{\lambda,N}(\bx, \by)
       \frac{|x_k-x_\ell |^{\alpha p}}{|y_i-y_j|^{\alpha p}}\Bigg)^{\tfrac{1}{p}}
        \Bigg( \sum_{\bx,\by}  \sfQ^{I,J}_{\lambda,N}(\bx, \by)
       \frac{|y_i-y_j |^{\alpha q}}{|x_k-x_\ell|^{\alpha q}} g(\by)^q\Bigg)^{\tfrac{1}{q}}. \label{eq:HLS3}
\end{align}
A straightforward computation shows that
 \begin{align*}
 \sum_{\by  \in (\Z^2)_J^h} \sfQ^{I,J}_{\lambda,N}(\bx, \by)
       \frac{|x_k-x_\ell |^{\alpha p}}{|y_i-y_j|^{\alpha p}} \leq C,
 \end{align*}
 which gives \eqref{eq:norm1} (see \cite[Section 6.2]{CSZ23a} for details). Keeping track of the constants more carefully and optimising
 over $\alpha$ will reveal the dependence on $p$ and $q$ as in \eqref{eq:HLS2} (see \cite[Prop.~3.3]{LZ23}).

The proof of \eqref{eq:norm2} and \eqref{eq:norm3} are similar. To see how the factor $(\sigma_N^2 \log \lambda)^{-1}$ arises in \eqref{eq:norm3},  we first recall that $\sfU^{I}_{\lambda, N}$ differs from $\sfQ^{I, I}_{\lambda, N}(\bx, \by)$ in that the transition kernels $q_n(y_i-x_i) q_n(y_j-x_j)$
in \eqref{eq:sfQ} are replaced by $U_N(n, y_i-x_i)$ in \eqref{eq:sfG}.
Here we can apply the naive approach \eqref{eq:naive}:
\begin{align*}
\sum_{\bx, \by \in (\Z^2)_I^h } f(\bx) \, \sfU^I_{\lambda, N}(\bx, \by) \, g(\bx)
\leq  \Big(\sum_{\bx, \by} f(\bx)^p \, \sfU^I_{\lambda, N}(\bx, \by) \Big)^{\frac{1}{p}}  \Big(\sum_{\bx, \by}
\sfU^I_{\lambda, N}(\bx, \by) \, g(\bx)^q \Big)^{\frac{1}{q}} ,
\end{align*}
where by \eqref{eq:sfG}, the definition of $U_N(n)$ in \eqref{renewU}, and the asymptotics in \eqref{eq:asU1} and \eqref{eq:Gas},
\begin{align*}
\sum_{\by \in (\Z^2)_I^h} \, \sfU^I_{\lambda, N}(\bx, \by)
= \sum_{n=1}^{2N} e^{-\frac{\lambda}{N} n} U_N(n)
\leq C \log N \int_0^\infty e^{-\lambda t} G_\theta(t) \,\dd t= \frac{C}{\sigma_N^2 \log \lambda} .
\end{align*}

\subsection{Moment bounds for the coarse-grained model} \label{S:4momCG}
In this section, we explain how to obtain the moment bounds \eqref{eq:Theta4} and \eqref{eq:Zcg4}, which are crucial in the application
of the Lindeberg principle to the coarse-grained model $\mathscr{Z}_{\epsilon}^{(\cg)}(\varphi, \psi | \Theta)$ as outlined in Section \ref{S:lind2}.

To deduce \eqref{eq:Theta4}, the key observation is that the chaos expansion for $\Theta_{N, \epsilon}(\vec\sfi, \vec\sfa)$ in \eqref{eq:Theta2} and \eqref{eq:Theta-pic} is similar to the chaos expansion for the averaged partition function $\cZ^{\beta_N}_N(\varphi, \ind)$ in \eqref{eq:Zppchaos}. Indeed, if we follow the renewal interpretation of the sequence of time-space points $(n_1, z_1)$, \ldots, $(n_r, z_r)$ in the chaos expansion
for $\cZ^{\beta_N}_N(\varphi, \ind)$ as we did in Section \ref{S:renewal}, which is based on an expansion for $\bbE[\cZ^{\beta_N}_N(\varphi, \ind)^2]$, then the first renewal point $(n_1, z_1)$ is sampled according to the mass function
$$
M(n_1, z_1):= \frac{1}{N^2} \Big(\sum_{z_0\in \Z^2} \varphi(\tfrac{z_0}{\sqrt N}) q_{n_1}(z_1-z_0)\Big)^2, \qquad 0<n_1<N, z_1\in \Z^2.
$$
It is easy to see that the total mass $\sum_{0<n_1<N, z\in \Z^2} M(n_1, z_1)$ converges to a positive constant as $N\to\infty$, and if $\varphi(x)\geq 1$
for all $\Vert x\Vert \leq 1$, then uniformly in $0<n_1<N$ and $\|z_1\|\leq \sqrt{N}$, $M(n_1, z_1)$ is bounded from below by $c/N^2$, i.e., a positive multiple of the uniform probability mass function on the set of such $(n_1, z_1)$.

If we apply the same renewal interpretation to the chaos expansion for $\Theta_{N, \epsilon}(\vec\sfi, \vec\sfa)$ in \eqref{eq:Theta2}, then we see that
the first renewal point $(n_1, z_1)$ is sampled uniformly from the time-space box with $0<n_1<\eps N$ and $\Vert z_1\Vert \leq \sqrt{\eps N}$, which
can be compared to the chaos expansion for the averaged partition function $\cZ^{\beta_N}_{\eps N}(\varphi, \ind)$ defined in \eqref{eq:rescZmeas3}
with $\varphi(x)=1$ for $\|x\|\leq 1$, time horizon $N'=\eps N$ and spatial averaging on the scale $\sqrt{N'}$. This allows us to apply Theorem \ref{th:mom} with $N'=\eps N$, $\varphi (x) = 1_{\{\Vert x\Vert \leq 1\}}$ and $\psi\equiv 1$ to deduce that for all $N$ large,
\begin{equation}
\bbE \big[ \Theta_{N, \epsilon}(\vec\sfi, \vec\sfa)^4  \big] < \frac{C}{\log \frac{1}{\eps}}.
\end{equation}
See \cite[Section 7.2]{CSZ23a} for more details.

The moment bound \eqref{eq:Zcg4} for the coarse-grained partition function follows from the following analogue of Theorem \ref{th:mom}.

\begin{theorem}\label{th:cgmom}
Let $\mathscr{Z}_{N, \epsilon}(\varphi,\psi)
:=\mathscr{Z}_{\epsilon}^{(\cg)}(\varphi,\psi|\Theta_{N, \eps})$ be the coarse-grained partition function
defined in \eqref{eq:Zcg-gen}.
Further, assume that $\Vert\psi\Vert_\infty<\infty$ and $\psi$ is supported on a
ball $B$ (possibly $B=\R^2$). Then for any $p, q \in (1,\infty)$ with $\frac{1}{p}+\frac{1}{q}=1$ and
any $w: \R^2 \to (0,\infty)$ such that $\log w$ is Lipschitz continuous, there exists $C\in (0,\infty)$
such that uniformly in $\eps\in (0,1)$,
\begin{equation}\label{eq:mombdcg}
	\limsup_{N\to\infty} \bbE\Big[\big(\mathscr{Z}_{N, \epsilon}(\varphi,\psi)
	- \bbE[\mathscr{Z}_{N, \epsilon}(\varphi,\psi)]\big)^4\Big] \leq C \eps^{\frac 4p}
	\Big\Vert \frac{\varphi_{1/\eps}}{w_{1/\eps}}\Big\Vert_{\ell^p(\Z^2)}^4 \, \Vert \psi\Vert_\infty^4
	\Vert w \ind_B\Vert_q^4,
\end{equation}
where for $\phi: \R^2 \to\R$, $\phi_\epsilon: \Z^2\to\R$ is defined as above Theorem \ref{th:mom}.
\end{theorem}
This is Theorem 8.1 in \cite{CSZ23a}, which is proved by adapting the proof for Theorem \ref{th:mom}. Complications
arise because the coarse-grained disorder variables $\Theta_{N, \epsilon}(\vec\sfi, \vec\sfa)$ is a family of
dependent random variables. We refer to \cite[Section 8]{CSZ23a} for details.

\subsection{Moments of the Critical $2d$ SHF}
We already gave the second moment of the Critical $2d$ SHF in Theorem~\ref{th:main0}. Now we give a formula for higher
integer moments of the SHF, which arises as the limit of the expansion in \eqref{discrete-mom} and hence consists of
an infinite series. In particular, the kernels $\sfU_N^{I_\ell}, \sfQ_{a_1}^{*,I_1}, \sfQ_{a_\ell-b_{\ell-1}}^{I_{\ell-1}, I_\ell }, \sfQ_{N-b_{m}}^{I_m,*}$ in \eqref{discrete-mom} will be replaced by their continuum analogues as follows.

Similar to \eqref{eq:ZhI}, for $I=\{i<j\}\subset \{1, \ldots, h\}$, define
\begin{equation} \label{eq:R2hij}
	(\R^2)^h_I \, := \, \big\{ \bx = (x_1, \ldots, x_h) \in (\R^2)^h: \
	x_i = x_j \big\} \,.
\end{equation}
\begin{itemize}
\item[{\bf 1. Replica Evolution.}] The continuum analogue of $\sfU_N^I(n,\bx, \by)$ in \eqref{eq:sfG}
is given by the kernel with density w.r.t.\ Lebesgue measure on $(\R^2)^h_I$
\begin{align} \label{eq:sfG2}
	\sfG_{\theta,t}^I(\bx,  \by)
	\,:=\, \prod_{\ell \in \{1,\ldots, h\} \setminus I}
	\!\!\!\! g_t (y_\ell - x_\ell) \ \cdot
	G_\theta(t) \, g_{\frac{t}{2}} (y_i - x_i), \qquad \bx,\by \in (\R^2)^h_I,
\end{align}
which is consistent with the definition of $\sfU_N^I(n,\bx, \by)$ by \eqref{eq:llt} and \eqref{eq:asU2}. This kernel
defines an integral operator from $L^p((\R^2)^h_I)$ to $L^p((\R^2)^h_I)$. Similar to $U_N(n, y_i-x_i)$ in \eqref{eq:sfG},
the factor $G_\theta(t) \, g_{\frac{t}{2}} (y_i - x_i)$ gives the weight of the wiggle line in Figure \ref{figure-fourth-moment}.

\item[{\bf 2. Constrained Evolution.}] For $I=\{i<j\}$ and $J=\{k< \ell\}$, the continuum analogue of $\sfQ^{I, J}_t$ in
\eqref{eq:sfQ} is given by the kernel with density
\begin{align}\label{eq:sfQ2}
\cQ_{t}^{I,J}(\by, \bx) \,:=\, \prod_{a=1}^h \, g_t (x_a - y_a), \qquad \by \in (\R^2)^h_I, \bx \in (\R^2)^h_J.
\end{align}
This kernel defines an integral operator from $L^p((\R^2)^h_J)$ to $L^p((\R^2)^h_I)$ and corresponds to time strips in
Figure \ref{figure-fourth-moment} that contain no wiggle line. As in \eqref{eq:sfQ*} and \eqref{eq:sfQ**}, we can define the boundary kernels
$\cQ_t^{I,*}(\by,  \bx)$, $\cQ_t^{*, J}(\by,  \bx)$, $\cQ^{*,J}_t(\varphi, \bx)$, and $\cQ^{I,*}_t(\by, \psi)$.
\end{itemize}

We can now give the formula for higher integer moments of the critical $2d$ SHF, which was first derived in \cite{GQT21}. The
series representation for the $3$rd moment was derived earlier in \cite{CSZ19b}.

\begin{theorem}\label{th:moments}
Fix $h\in\N$ with $h \ge 3$. For $\varphi\in C_c(\R^2)$ and $\psi\in C_b(\R^2)$, the $h$-th moment of $\SHF_{0, t}^\theta(\varphi, \psi)$ is finite and admits the expression
\begin{align} \label{eq:Zmomh}
	\bbE\big[ \SHF^\theta_{0, t}(\varphi, \psi)^h \big]
	&= \idotsint\limits_{(\R^2)^h\times (\R^2)^h } \varphi^{\otimes h}(\bz) \,
	\mathscr{K}_t^{(h)}(\bz, \bw) \, \psi^{\otimes h}(\bw)\, \dd \bz\, \dd \bw \,,
\end{align}
where $\bx=(x_1, \ldots, x_h)\in (\R^2)^h$ and $\phi^{\otimes h}(\bx) =\prod_{i=1}^h \phi(x_i)$, and
\begin{align}
       &  \mathscr{K}_t^{(h)}(\bz, \bw) := \notag \\
        & \, 1 +\!\!
       \sum_{m=1}^\infty
        (4\pi)^m
       \!\!\!\!\!\!\!\!
       \sum_{\substack{  I_1,..., I_m \subset \{1,\ldots,h\} \\
        |I_\ell|=2, \, I_\ell \neq I_{\ell+1}}} \ \
        \idotsint\limits_{0 < a_1 < b_1 < \ldots < a_m < b_m < t} \!\!\!\!\!\!\!\!\!
        \dd \vec{a} \, \dd \vec{b} \
        \idotsint\limits_{\bx^{(\ell)}, \by^{(\ell)} \in (\R^2)^{h}_{I_\ell} \atop \mbox{\tiny for } \ell=1, \ldots, m} \prod_{\ell=1}^m \dd \bx^{(\ell)} \, \dd \by^{(\ell)} \label{eq:Zmomh-kernel} \\
        & 	\cQ_{a_1}^{*,I_1}(\bz, \bx^{(1)}) \, \sfG_{\theta, b_1-a_1}^{I_1} (\bx^{(1)}, \by^{(1)})
	\Big(\prod_{\ell=2}^{m} \cQ_{a_\ell-b_{\ell-1}}^{I_{\ell-1}, I_\ell}(\by^{(\ell-1)}, \bx^{(\ell)})
	\, \sfG_{\theta, b_\ell-a_\ell}^{I_\ell}(\bx^{(\ell)},  \by^{(\ell)})\Big) \cQ_{t-b_m}^{I_m, *}(\by^{(m)}, \bw). \notag
\end{align}
\end{theorem}
The identity \eqref{eq:Zmomh-kernel} can be proved by taking the limit $N\to\infty$ in \eqref{discrete-mom}, see \cite{GQT21}. Similar bounds as in Theorem
\ref{th:mom} also hold for $\SHF^\theta_{0, t}(\varphi, \psi)$, which allows one to take more general $\varphi$ and $\psi$ in \eqref{eq:Zmomh}
as long as the terms in the r.h.s.\ of \eqref{eq:mombd2} are finite.

\subsection{Related literature}
The positive integer moments of the directed polymer partition function can be expressed in terms of the collision local times of independent
random walks, which are also connected to the so-called Delta-Bose gas.
\medskip

\noindent{\bf Exponential moments of collision local times.}
Recall the definition of the point-to-point partition function $Z^\beta_N(z, w):=Z^\beta_{0, N}(z, w)$ from \eqref{eq:ZMN}. Let us consider mixed moments of the form $\bbE[Z^\beta_N(z_1, w_1) \cdots Z^\beta_N(z_h, w_h)]$. Assuming the disorder variables $\omega(n, z)$ are i.i.d.\ standard normal, we can then compute directly
\begin{align}\label{collision}
\bbE\Big[ \prod_{i=1}^h Z_{N}^{\beta}(z_i, w_i) \Big]
&=\bbE \E^{\otimes h}_{\bz}\Big[  e^{\sum_{i=1}^h
\sum_{n=1}^{N-1} \big(\beta \omega(n, S_n^{(i)}) - \lambda(\beta) \big)}  \prod_{i=1}^h \ind_{\{S_N^{(i)}=w_i\}}  \Big] \notag \\
& = \E^{\otimes h}_{\bz}\bigg[ \prod_{(n, x)\in \{1, \ldots, N-1\}\times \Z^2} e^{\beta^2 \sum_{1\leq i<j\leq h}  \ind_{\{S_n^{(i)}=S_n^{(j)}=x\}}} \prod_{i=1}^h \ind_{\{S_N^{(i)}=w_i\}} \bigg] \notag\\
&=  \E^{\otimes h}_{\bz}\bigg[ e^{ \beta^2  \sum_{1\leq i<j\leq h}  \sfL_N^{\{i,j\}}} \prod_{i=1}^h \ind_{\{S_N^{(i)}=w_i\}}\bigg],
\end{align}
where $\E^{\otimes h}_{\bz}$ denotes expectation w.r.t.\ $h$ independent random walks $(S^{(i)})_{1\leq i\leq h}$ starting from $\bz=(z_1, \ldots, z_h)$ respectively, and $\sfL_N^{\{i,j\}} := \sum_{n=1}^{N-1} \ind_{\{S_n^{(i)}=S_n^{(j)}\}}$ denotes the collision local time between $S^{(i)}$ and $S^{(j)}$.

When $z_i=0$ for all $1\leq i\leq h$, the classic Erd\H{o}s-Taylor theorem \cite{ET60} shows that each $ \sfL_N^{\{i,j\}}/R_N=\sfL_N^{\{i,j\}}/\E[\sfL_N^{\{i,j\}}]$ converges in distribution to an Exp$(1)$ random variable. When $I$ is a set of pairs $\{i<j\}$ such that the graph with vertex set $\{1, \ldots, h\}$ and edge set $I$ contains no loops, it follows from \cite{GS09} that $(\sfL^{i,j}_N/R_N)_{\{i<j\}\subset I}$ converges jointly to a family of i.i.d.\ Exp$(1)$ random variables. Finally, \cite{LZ24} proved the joint convergence of the full family $(\sfL^{i,j}_N/R_N)_{1\leq i<j\leq h}$ to a family of i.i.d.\ Exp$(1)$ random variables. This implies that if we choose $\beta=\beta_N = (\hat\beta/R_N)^{1/2}$ for some $\hat\beta<1$, i.e., the subcritical window defined in \eqref{eq:hbeta}, then we can identify the limit of \eqref{collision} as $N\to\infty$ in terms of the exponential moments of independent exponential random variables even if $z_i=z_j$ for some $i\neq j$.

However, when $\beta=\beta_N$ is critical with $\hat\beta=1$, allowing $z_i=z_j$ for $i\neq j$ will lead to divergence in \eqref{collision} as $N\to\infty$ because
$$
\E[e^{\hat\beta(\sfL^{i,j}_N/R_N)}] = \bbE[Z_N^{\beta_N}(0)^2] \sim C\log N,
$$
where the aysmptotics follows from Lemma \ref{L:Z2Asymp}. Nevertheless, as long as $(z_i/\sqrt{N})_{1\leq i\leq h}$ converge to a vector of distinct points $\bz'$ as $N\to\infty$ and the same holds for $(w_i/\sqrt{N})_{1\leq i\leq h}\to \bw'$, we expect \eqref{collision} to converge (after normalising by $N^h$) to the kernel $\mathscr{K}_1^{(h)}(\bz', \bw')$ in Theorem \ref{th:moments}.

\medskip
\noindent{\bf Delta-Bose gas and singular interacting diffusions.}
Formally, the continuum analogue of \eqref{collision} is
\begin{align}\label{Bose-FK}
\mathscr{S}_t^{(h)}(\bz, \bw) := \E^{\otimes h}_{\bz}
 \bigg[ e^{ \beta^2  \sum_{1\leq i<j\leq h} \int_0^t \delta(B_s^i-B_s^j)  \,\dd s}
  \,\prod_{i=1}^h \delta(B^i_t-w_i)\bigg],
\end{align}
where $(B^i)_{1\leq i \leq h}$ are i.i.d.\ standard Brownian motion on $\R^2$, starting from $\bz=(z_1, \ldots, z_h)$ respectively,
and $\delta(\cdot)$ is the delta function at the origin. This is the Feynman-Kac semigroup associated with the Schr\"odinger operator
with delta potential on the diagonals
\begin{equation}\label{eq:Sch}
\frac{1}{2}\Delta + \beta^2\sum_{1\leq i<j\leq h} \delta(x_i-x_j),
\end{equation}
known as the Delta-Bose gas. In other words, $\mathscr{S}_t^{(h)}(\bz, \bw)$ solves (formally)  the parabolic equation
\begin{equation}\label{BosePDE}
\begin{aligned}
\partial_t u(t, \bx) & =\frac{1}{2}\Delta u(t, \bx) +\beta^2\sum_{1\leq i<j\leq h} \delta(x_i-x_j) u(t, \bx), \qquad t>0, \bx \in (\R^2)^h, \\
u(0, \bx) & = \prod_{i=1}^h \delta(x_i-w_i).
\end{aligned}
\end{equation}
Of course, the presence of the delta function makes \eqref{Bose-FK}-\eqref{BosePDE} ill-defined, especially because Brownian motion in
$\R^2$ does not hit points. The challenge is to make sense of the semigroup $\mathscr{S}_t^{(h)}(\bz, \bw)$ and its associated operator. A natural approach is to replace the delta function by its mollified version  $\delta_\epsilon(x)= \frac{1}{\epsilon^2} j(\frac{x}{\epsilon})$ for some smooth probability density function $j$. This introduces a minimal spatial scale $\eps>0$ and has the same effect as discretising space by considering the discrete time-space kernel considered in \eqref{collision}. Similar to \eqref{eq:epshbeta}, the most interesting choice will be to choose
\begin{equation}\label{eq:DBbeta}
\beta^2=\beta^2_\epsilon= \frac{2\pi}{\log \frac{1}{\epsilon}} \Big(1+\frac{\theta+o(1)}{|\log \epsilon|}\Big),
\end{equation}
which lies in the critical window with $\hat\beta=1$. This has been the choice considered in the literature, see e.g.\ \cite{AFHKKL92, DFT94, ABD95, BC98, DR04, AGHKH05}. In particular, \cite{DFT94} defined self-adjoint extensions of the operator in \eqref{eq:Sch} using an
 approach of quadratic forms and $\Gamma$-convergence, while \cite{DR04} followed an approach using resolvents and $L^2$ Fourier analysis. The latter approach is closely related to the diagrammatic expansions used in \cite{GQT21} and \cite{CSZ23a}, which is illustrated in Figure \ref{figure-fourth-moment}.

In light of the connection with the directed polymer and the stochastic heat equation, the semigroup $\mathscr{S}_t^{(h)}(\bz, \bw)$ in \eqref{Bose-FK} should in fact be a family of semigroups indexed by $\theta\in \R$, given by the kernels $\mathscr{K}_t^{(h)}(\bz, \bw)$ in Theorem \ref{th:moments} for the $h$-th moment of the critical $2d$ SHF.

The formal semigroups $\mathscr{S}_t^{(h)}(\bz, \bw)$ in \eqref{Bose-FK} were defined rigorously in \cite{Che21} as the limit of approximate semigroups via mollification of the delta function and choosing $\beta$ as in \eqref{eq:DBbeta}. This naturally leads to a family of $(\R^2)^h$-valued time-inhomogeneous Markov processes $Y^{T, (h)}_t=(Y^{T, (h)}_{1, t}, \ldots, Y^{T, (h)}_{h, t})$, such that given terminal time $T$, it has
transition densities
$$
g^{T, (h)}_{s, t} (\bx, \by):=  \mathscr{S}_{t-s}^{(h)}(\bx, \by) \cdot
\frac{\int \mathscr{S}_{T-t}^{(h)}(\by, \bz) \dd \bz}{\int \mathscr{S}_{T-s}^{(h)}(\bx, \bz) \dd \bz}, \qquad 0\leq s<t\leq T, \, \bx, \by \in (\R^2)^h.
$$
When $h=2$, \cite{CM25} considered the difference $X^T_t:= Y^{T, (2)}_{2, t} - Y^{T, (2)}_{1, t}$ between the two-components and showed that it solves a stochastic differential equation (SDE) with a singular drift toward the origin, which enables $X^T_\cdot$ to visit and accumulate a local time at the origin, defined as the limit (as $\eps\downarrow 0$) of
$$
L_t^\eps :=\frac{1}{2\epsilon^2 (\log \frac{1}{\epsilon})^2} \big| \{ s\in [0,t] \colon |X^T_s|\leq \epsilon \} \big|, \qquad t\in [0, T].
$$
This stands in contrast to a standard Brownian motion in $\R^2$, which does not hit the origin.

The semigroups $\mathscr{S}_t^{(h)}(\bz, \bw)$ were also constructed in \cite{Che22, Che24} via a stochastic representation in the spirit of
\eqref{Bose-FK}, but in terms of a special diffusion process instead of a Brownian motion. In particular, when $h=2$, the analogue of the semigroup
in \eqref{Bose-FK} in terms of the relative motion $B^1_t-B^2_t$ is given by
\begin{align*}
P_t^{\rho} f(z_0) & \ ``="\  \E^{\otimes h}
 \bigg[ e^{ \beta^2 \int_0^t \delta(B_s^1-B_s^2)  \,\dd s}
  \, f(B^1_t-B^2_t) \,\bigg|\, B^1_0-B^2_0=z_0 \bigg] \\
  &\ \  := \ \ \E_{z_0}\Bigg[ \frac{e^{\rho t} K_0(\sqrt{\rho} |z_0| )}{K_0(\sqrt{\rho} |Z_t|)} f(Z_t)\, \bigg|\, Z_0=z_0\Bigg],
\end{align*}
where $K_0$ is the Macdonald function or generalised Bessel function of second kind of order $0$,
the parameter $\rho$ is determined by $\theta$ in \eqref{eq:DBbeta} and the mollifier $j$,
and $(Z_t)_{t\geq 0}$ is an $\R^2$-valued diffusion whose radial component $|Z_t|$ is a transformation of a
Bessel process constructed by Donati-Martin and Yor \cite{DMY06},
with $0$ being a point of instantaneous reflection for $|Z_t|$.

There have also been recent studies of delta-Bose gas in dimension $d \ge 3$ \cite{Wan24}.

\section{Further properties of critical $2d$ SHF} \label{sec:prop}

In this section, we list some further properties of the critical $2d$ SHF with parameter $\theta\in \R$,
that we denote by $\SHF(\theta) =
(\SHF_{s,t}^{\theta}(\dd x , \dd y))_{0 \le s \le t <\infty}$.

\subsection{Shift invariance and scaling covariance}
The critical $2d$ SHF satisfies the following shift invariance and scaling covariance properties \cite[Theorem 1.2]{CSZ23a}.
\begin{theorem}[Translation invariance and scaling covariance] \label{th:main1}
The critical $2d$ SHF is translation invariant in law:
\begin{equation*}
	(\SHF_{s+\sfa, t+\sfa}^\theta(\dd (x+\sfb) , \dd (y+\sfb)))_{0 \le s \le t <\infty}
	\stackrel{\rm dist}{=}
	(\SHF_{s,t}^{\theta}(\dd x , \dd y))_{0 \le s \le t <\infty}
	\quad \ \forall \sfa \ge 0, \ \forall \sfb \in \R^2 \,,
\end{equation*}
and it satisfies the following scaling relation:
\begin{equation}\label{eq:scaling}
	(\SHF_{\sfa s, \sfa t}^\theta(\dd (\sqrt{\sfa} x) , \dd (\sqrt{\sfa} y)))_{0 \le s \le t <\infty}
	\stackrel{\rm dist}{=}
	(\sfa\, \SHF_{s,t}^{\theta+ \log \sfa}(\dd x , \dd y))_{0 \le s \le t <\infty}
	\quad \ \forall \sfa > 0 \,.
\end{equation}
\end{theorem}
The shift invariance is inherited from the shift invariance of the directed polymer model and that of the
underlying disorder $\omega$. The scaling covariance relation can be checked by computing the
covariance of the critical $2d$ SHF. It shows that zooming out in space-time diffusively
has the effect of increasing the disorder strength $\theta$.

\subsection{Flow property}
The heat equation induces the so-called heat flow in the sense that the solution is linear in the initial condition and can be written as a mixture of heat kernels
that satisfy the Chapman-Kolmogorov equation $p_{s, u}(x, {\rm d}y) = \int p_{s, t}(x, {\rm d}z) p_{t, u}(z, {\rm d}y)$. If the family of kernels
$p_{s, t}(x, {\rm d}y)_{s<t, x\in \R}$ is replaced by a random family $(p^\omega_{s, t}(x, {\rm d}y))_{s<t, x\in \R}$ such that for all $s<t<u$ and
$x\in \R$, almost surely $p^\omega$ satisfies the Chapman-Kolmogorov equation
\begin{equation}\label{eq:CK}
p^\omega_{s, u}(x, {\rm d}y) = \int p^\omega_{s, t}(x, {\rm d}z) p^\omega_{t, u}(z, {\rm d}y),
\end{equation}
and $p^\omega_{s_i, t_i}$ are independent over disjoint time intervals $[s_i, t_i]$ and translation invariant in law under space-time shifts, then $(p^\omega_{s, t}(x, {\rm d}y))_{s<t, x\in \R}$ is a {\em stochastic flow of kernels} introduced by Le Jan and Raimond \cite{LR04} and can be interpreted as the transition kernels of a random motion in an `i.i.d.' space-time random environment (see e.g.~\cite{SSS14}).The critical $2d$ SHF $(\,\SHF^\theta_{s,t} (\dd x, \dd y))_{s<t}$ satisfies a similar property, except that it is not meaningful to consider delta initial condition $\int \delta_x\, \SHF^\theta_{s,t} (\dd x, \dd y)$ as in \eqref{eq:CK}. Instead, we should consider $\SHF^\theta_{s,t} (\mu, \dd y)$ for initial condition $\mu({\rm d}x)$ that is sufficiently regular.

In a forthcoming work \cite{CSZ25+}, we will show that the critical $2d$ SHF defines a linear measure-valued Markov process with a state space
$S$, which consists of locally finite measures $\mu$, such that for any $R>0$,
\begin{align*}
\iint_{|x|, |y|\leq R}  \log \tfrac{1}{|x-y|} \, \mu(\dd x)\, \mu(\dd y)  \,<\infty,
\end{align*}
and $\int e^{-|x|^2/2t}\mu(\dd x)<\infty$ for all $t>0$. This is a natural state space for the measure valued stochastic process induced by the critical $2d$ SHF, since
\begin{align*}
\bbvar\big( \SHF_{0, t}^\theta(\phi) \big) = \iint_{(\R^2)^2} \phi(x) \, \phi(x') \, \widetilde K_t^\theta(x,x')  \,\dd x \,\dd x',
\end{align*}
where $\widetilde K_t^\theta(x,x')$ is as in \eqref{def:Ktilde} and $\widetilde K_t^\theta(x,x') \sim c \log \frac{1}{|x-x'|}$ as $|x-x'|\to 0$.

We will show in \cite{CSZ25+} that the critical $2d$ SHF satisfies the following almost sure Chapman-Kolmogorov (flow) property: for all $s<t<u$ and initial measure $\mu\in S$, almost surely, $\SHF^\theta_{s,t} (\mu, \dd z)\in S$ and
\begin{equation}\label{eq:CK2}
\SHF^\theta_{s, u} (\mu, \dd y) = \int \SHF^\theta_{s,t} (\mu, \dd z)
\, \SHF^\theta_{t,u} (z, \dd y) =  \SHF^\theta_{t,u} \Big(\SHF^\theta_{s,t} (\mu, \dd z), \dd y\Big).
\end{equation}

Clark and Mian \cite{CM24} have given a different formulation of the flow property, which starts
by defining a notion of $\epsilon$-regularised convolution for measures. More precisely,
let $X,X'$ be Polish spaces and let $\mu_1,\mu_2$ be Borel measures on $X\times \R^2$ and $\R^2\times X'$, respectively.
Then the $\epsilon$-regularised convolution of $\mu_1$ and $\mu_2$ is defined as a measure on $X\times \R^2 \times X'$ with
\begin{align*}
\mu_1 \circ_\epsilon \mu_2 (\dd s, \dd x, \dd s') := \iint_{x'\in \R^2} \mu_1(\dd s, \dd x) g_{\epsilon}(x-x')  \,\mu_2(\dd x', \dd s'),
\end{align*}
where $g_{\epsilon}$ is the heat kernel at time $\epsilon$. The following result was established in \cite{CM24}.
\begin{theorem} \label{thm:Clark}
Let $\SHF^\theta_{s,r} (\dd x, \dd y)$ and $\SHF^\theta_{r,t} (\dd y, \dd z)$ be marginals
of the critical $2d$ SHF at times $s<r<t$. Then there exists a random measure $\SHF^\theta_{s,r, t} (\dd x, \dd y, \dd z)$ on $(\R^2)^3$ such
that
\begin{align}\label{eq:Ch-Ko}
\SHF^\theta_{s,r} (\dd x, \dd y) \circ_\epsilon \SHF^\theta_{r,t} (\dd y, \dd z)
\xrightarrow[\epsilon \downarrow 0]{L^2(\bbP)} \SHF^\theta_{s,r, t} (\dd x, \dd y, \dd z),
\end{align}
where the convergence is in $L^2(\bbP)$. Moreover,
\begin{align*}
\SHF^\theta_{s, t} (\dd x, \dd z) = \int_{y\in\R^2} \SHF^\theta_{s,r, t} (\dd x, \dd y, \dd z).
\end{align*}
\end{theorem}
This result also establishes an almost sure Chapman-Komogorov type of relation for the critical $2d$ SHF.
The procedure can be iterated to show that for $0<t_1<\cdots < t_n$, the critical $2d$ SHF satisfies
\begin{align*}
\SHF^\theta_{t_1,t_2} (\dd x_1, \dd x_2)\circ_\epsilon \cdots \circ_\epsilon \SHF^\theta_{t_{n-1}, t_n} (\dd x_{n-1}, \dd x_n)
\xrightarrow[\epsilon \downarrow 0]{L^2(\bbP)} \SHF^\theta_{t_1,..., t_n} (\dd x_1,..., \dd x_n).
\end{align*}

\subsection{Characterisation of the critical $2d$ SHF} \label{S:Tsai}
Recently, Tsai \cite{T24} gave an axiomatic characterisation of the critical $2d$ SHF, which
says that, a process $(Z_{s,t}(\cdot, \cdot))_{s\leq t}$ taking values in $M_+(\R^2\times\R^2)$, the space of locally
finite non-negative measures on $\R^2\times\R^2$ equipped with the vague topology, must have the same law as the critical
$2d$ SHF with parameter $\theta$, denoted by SHF$(\theta)$, if $Z_{s,t}$ is continuous in $s$ and $t$,
satisfies the flow property formulated in \eqref{eq:Ch-Ko}, has independent increments, and has matching
first four moments with $(\SHF^\theta_{s,t}(\cdot, \cdot))_{s\leq t}$.
\begin{theorem}[Characterisation of the critical $2d$ SHF]\label{thm:tsai}
Let $Z=(Z_{s,t}(\cdot, \cdot))_{s\leq t}$ be a stochastic process taking values in $M_+(\R^2\times\R^2)$ such that:
\begin{itemize}
\item[(1)] for any $s<t<u$, the random measures $Z_{s,t}$, $Z_{t, u}$ and $Z_{s, u}$ satisfy the conclusions of Theorem \ref{thm:Clark};\footnote{In \cite{T24}, the condition is formulated such that the Gaussian kernels $g_\eps$ in Theorem \ref{thm:Clark} can be replaced by any family of mollifiers that ensure $L^2$ convergence.}
\item[(2)] for any $s<t<u$, $Z_{s,t}$ and $Z_{t,u}$ are independent;
\item[(3)] for any $s<t<u$, $1\leq n\leq 4$, and $\phi_i,\psi_i\in L^2(\R^2)$ for $i=1,\ldots,4$, the mixed moments
$\bbE\big[\prod_{i=1}^n Z_{s,t}(\phi_i,\psi_i) \big]$ agree with that of {\rm SHF}$(\theta)$, cf. \eqref{eq:Zmomh}, for some $\theta\in\R$.
\end{itemize}
Then $Z$ has the same law as the critical $2d$ stochastic heat flow {\rm SHF}$(\theta)$. Furthermore, $(Z_{s,t}(\cdot, \cdot))_{s\leq t}$ admits
a version that is almost surely continuous in $s\leq t$.
\end{theorem}

The characterization in Theorem \ref{thm:tsai} was proved in \cite{T24} via a Lindeberg principle, while the continuity was established by verifying
the Kolmogorov moment criterion.
If $Z$ and $\widetilde Z$ are two processes satisfying the
assumptions in Theorem \ref{thm:tsai} for the same $\theta$, then by the Chapman-Kolmogorov property, $Z_{0, 1}$ is a functional of
$(Z_{(i-1)/n, i/n})_{1\leq i\leq n}$, and the same holds for $\widetilde Z$. The basic idea is that, in the reconstruction
of $Z_{0, 1}$ from $(Z_{(i-1)/n, i/n})_{1\leq i\leq n}$, one can successively replace $Z_{(i-1)/n, i/n}$ by
$\widetilde Z_{(i-1)/n, i/n}$ and control the error of each replacement. If the cumulative error can be shown to tend to $0$ as
$n\to\infty$, then $Z_{0,1}$ and $\widetilde Z_{0, 1}$ must be equal in law, where we can take $\widetilde Z$ to be SHF$(\theta)$.
Similar to the Lindeberg principle explained in Section \ref{S:lind1}, controlling the error of each replacement requires Taylor
expansions and matching first two moments for $Z_{(i-1)/n, i/n}$ and $\widetilde Z_{(i-1)/n, i/n}$, plus suitable bounds on higher
moments (4-th moment should suffice). Of course $Z_{(i-1)/n, i/n}$ are much more complicated objects than $\R$-valued random variables,
which requires more delicate analysis. We refer to \cite{T24} for further details.

Theorem \ref{thm:tsai} makes it much easier to prove convergence to SHF$(\theta)$ since one only needs to verify that every subsequential
limit satisfies the axioms in Theorem \ref{thm:tsai}. This was carried out in \cite{T24} for the solution $u^\eps$ of the $2d$ mollified SHE in the critical window.

We also mention that there have been some recent progress in formulating martingale problems 
for the SHF \cite{N25, C25a}. The issue of uniqueness, which would yield
a well-posed martingale problem characterisation for the SHF, is still open.

\subsection{Non-GMCness of fixed time marginals}
Let $\big(\sfX(x) \big)_{x\in \R^d}$ be a centered (generalised) gaussian field with correlation function $k(x,y)$ for $x,y\in \R^d$. The Gaussian Multiplicative Chaos (GMC) is a random measure on $\R^d$, which formally has the following ``density'' with respect to a reference measure $\sigma(\dd x)$:
 \begin{align}\label{GMCdef}
 \mathscr{M}_\gamma(\dd x):= e^{\gamma \sfX(x) -\frac{\gamma^2}{2} \bbE[\sfX^2] } \,\sigma(\dd x).
 \end{align}
 Various assumptions can be imposed on the measure $\sigma(\dd x)$ (see \cite{RV14}), but for simplicity, we will just assume that
 $\sigma(\dd x)$ is the Lebesque measure. For any test function $\phi$, we denote
 \begin{align*}
 \mathscr{M}_\gamma(\phi):=\int_{\R^d} \phi(x) \mathscr{M}_\gamma(\dd x).
 \end{align*}
 The most interesting case is when the field $\sfX$ is log-correlated in the sense that $k(x, y) \sim c \log\frac{1}{|x-y|}$ as $|x-y|\to 0$.
 Since $k(x, x)=\infty$, $X$ is a generalised Gaussian field and $\mathscr{M}_\gamma(\dd x)$ is a priori undefined. To rigorously define
 $\mathscr{M}_\gamma(\dd x)$, the standard procedure is to first replace $X$ by a regularised field $\sfX_\epsilon$ (e.g., via mollification) and then show that
    \begin{align}\label{GMCapprox}
 \mathscr{M}_{\epsilon, \gamma}(\dd x):= e^{\gamma \sfX_\epsilon(x) -\frac{\gamma^2}{2} \bbE[\sfX_\epsilon^2] } \,\sigma(\dd x),
 \end{align}
 has a limit as $\epsilon\to0$. Log-correlated Gaussian field $X$ is of special interest because there is a phase transition in the parameter
 $\gamma$ in \eqref{GMCapprox} that is similar to the phase transition in the polymer partition function in Theorem \ref{T:subcritical}.
 In particular, when $\sigma(\dd x)$ is the Lebesque measure, $\mathscr{M}_{\epsilon, \gamma}(\dd x)$ has a non-trivial limit if
 $\gamma<\gamma_c:=\sqrt{2d}$, and $\mathscr{M}_{\epsilon, \gamma}(\dd x)$ converges to the zero measure if $\gamma \geq \gamma_c$. We refer
 to \cite{RV14} for more details.
 \vskip 2mm
 It is natural to wonder whether the one-time marginal of SHF$(\theta)$ can be identified with a GMC measure on $\R^2$. The answer turns out to be no,
 which suggests heuristically that the logarithm of SHF$(\theta)$ (if defined rigorously, it will give a meaning to
 the solution of the critical $2d$ KPZ equation with parameter $\theta$), would be a Gaussian field.

Since GMC is defined from a Gaussian field, its law is entirely determined by its first two moments. If we denote by $\mathscr{M}^\theta_t$
the candidate GMC with the same first two moments as the critical $2d$ SHF $\SHF^\theta_t$ at time $t$, then we can rule out the
possibility of $\SHF^\theta_t$ being a GMC by showing that its higher moments do not match that of $\mathscr{M}^\theta_t$. This
is what we proved in \cite{CSZ23b}.

\begin{theorem}[The critical $2d$ SHF is not a GMC]\label{th:mmom}
Let $g_\delta$ be the heat kernel at time $\delta>0$. Let $\SHF_t^\theta(\cdot)$ be the critical $2d$ SHF started from Lebesgue measure
and evaluated at time $t$, and let $\mathscr{M}_t^\theta(\cdot)$ be the GMC measure on $\R^2$ with the same first two moments as $\SHF_t^\theta(\cdot)$. For any $\theta\in \R$ and $t>0$,there exists $\eta = \eta_{t,\theta} > 0$ such that for all $h\in\N$ with $h \ge 3$,
we have
\begin{equation} \label{eq:asystrict}
	\liminf_{\delta\downarrow 0}  \;
	\frac{\bbE\big[ \SHF_t^\theta(g_\delta)^h \big]}{\bbE\big[ \mathscr{M}_t^\theta(g_\delta)^h \big]}
	\ge 1+\eta > 1 \,.
\end{equation}
\end{theorem}
To see heuristically why Theorem \ref{th:mom} might hold, we note that Wick's theorem for the moments of gaussian variables
implies that
\begin{equation*}
	\bbE\big[ \mathscr{M}^\theta(g_\delta)^h\big]
	= \int_{(\R^2)^h} g_\delta(z_1) \cdots g_\delta(z_h) \,
	e^{\sum_{1\leq i<j\leq h}k(z_i, z_j) }
	\, \dd z_1 \cdots  \dd z_h  \sim \bbE\big[ \mathscr{M}^\theta(g_\delta)^2\big]^{h\choose 2},
\end{equation*}
where each term $k(z_i, z_j)$ represents a two-body interaction among the $h$ particles at locations $z_1, \ldots, z_h$,
and the last asymptotic uses the fact that
$$
\E[ \mathscr{M}^\theta(\dd x)  \mathscr{M}^\theta (\dd y)] = \E[ \SHF_t^\theta(\dd x)  \SHF_t^\theta (\dd y)]
=: K^\theta_t(x, y) \dd x\, \dd y
$$
with $K^\theta_t(x, y) \sim c \log \frac{1}{|x-y|}$ as $|x-y|\to 0$. In other words, $\bbE\big[ \mathscr{M}^\theta(g_\delta)^h\big]$
has the same asymptotics as if there are only independent two-body interactions. This is no longer the case for $\bbE\big[ \SHF_t^\theta(\varphi)^h \big]$, where the interactions among the $h$ particles are illustrated in Figure \ref{figure-fourth-moment}. At the heart
of the proof of Theorem \ref{eq:asystrict} is the use of the Gaussian Correlation Inequality \cite{R14, LM17}, which together
with additional analysis on the collision local times of Brownian motions, shows that there are positive correlations among the two-body
interactions.	
	
\subsection{Singularity and Regularity}
In \cite{CSZ25}, we show that the marginal distribution of the critical $2d$ SHF at each time $t>0$ is a random measure that
is almost surely singular respect to Lebegues measure (see Figure \ref{fig:SHFzoom}).
\begin{theorem}[Singularity of SHF] \label{th:singularity-SHF}
Fix any $t > 0$ and $\theta \in \R$.
Almost surely,
$\SHF_{t}^\theta(\dd x):= \int_{y\in \R^2} \SHF_{0, t}^\theta(\dd y, \dd x) $ is singular with respect to the Lebesgue measure on~$\R^2$.
\end{theorem}

However, $\SHF_{t}^\theta(\dd x)$ barely fails to be a function.

 \begin{theorem}[Regularity of the SHF]\label{th:regularity-SHF}
Fix any $t > 0$ and $\theta\in\R$.
Almost surely, the critical $2d$ SHF
$\SHF_{t}^\theta(\dd x)$ belongs to $\mathcal{C}^{0-}
:= \bigcap_{\epsilon > 0} \mathcal{C}^{-\epsilon}$, where $\mathcal{C}^{-\epsilon}$ is the negative H\"older space
of order $-\eps$.
\end{theorem}
A consequence of Theorem \ref{th:regularity-SHF} is that, almost surely, $\SHF_{t}^\theta(\dd x)$ contains no atoms,
since delta measures on $\R^d$ belong to $C^{-d}$.
Theorem \ref{th:regularity-SHF} follows from moment estimates, see \eqref{eq:p2plane}, which
 imply that for any $h\in \N$ and $\epsilon>0$,
there exists a constant $C$, depending on $h, \epsilon, t, \theta$, such that for all small $\delta$ we have
\begin{equation}\label{eq:moments-SHF}
	\bbE \Big[ \SHF_{t}^\theta\big( \cU_{B(x,\delta)}\big)^h \Big]^{1/h}
	\le C \, \delta^{-\epsilon}
	\qquad \forall x\in\R^2  \,.
\end{equation}
Theorem \ref{th:regularity-SHF} then follows from this moment bound and a general tightness criterion for negative H\"older spaces
\cite[Theorem~2.30]{FM17}.

Theorem \ref{th:singularity-SHF} follows by applying the Lebesgue differentiation theorem and showing that, almost surely,
\begin{equation}\label{eq:singularity-SHF}
	\lim_{\delta \downarrow 0} \; \SHF_{t}^\theta\big(\cU_{B(x,\delta)}\big) = 0
	\quad \text{for Lebesgue a.e.\ $x\in\R^2$} \,.
\end{equation}
where $\cU_{B(x,\delta)}(\cdot) := \frac{1}{\pi \delta^2} \, \ind_{B(x,\delta)}(\cdot)$ for the Euclidean ball $B(x,\delta) := \big\{y \in \R^2 \colon \ |y-x| < \delta \big\}$. This is accomplished by bounding the fractional moments of $\SHF_{t}^\theta\big(\cU_{B(x,\delta)}\big)$, where
we use the monotonicity of the fractional moment in $\theta$ to decrease $\theta$ and send $\theta=\theta(\delta)\to-\infty$ at a suitable rate as $\delta\downarrow 0$. More precisely, $\theta(\delta)$ will be chosen such that $\SHF_{t}^\theta\big(\cU_{B(x,\delta)}\big)$ converges to a log-normal limit, similar to the point-to-plane partition function in the subcritical regime in Theorem \ref{T:subcritical}.

\begin{theorem} \label{th:log-normality-SHF}
Let $\theta(\delta)=2\log \delta$. For any $t > 0$ and $x\in\R^2$, the following convergence in distribution holds:
\begin{equation}\label{eq:log-normal-SHF}
	\forall \rho \in (0,\infty) \colon \qquad
	\SHF_{t}^\theta \big(\cU_{B(x,\delta^\rho)}\big)
	\ \xrightarrow[\ \delta \downarrow 0 \ ]{d} \
	e^{\cN(0,\sigma^2) - \frac{1}{2}\sigma^2}
	\qquad \text{with } \sigma^2 = \log(1+\rho) \,.
\end{equation}
\end{theorem}
\begin{remark}
The meta-theorem suggested by Theorem \ref{th:log-normality-SHF} is that, averaging the polymer partition functions (or averaging the critical $2d$ SHF) has the effect of reducing the disorder strength.
If the averaging is on a suitably chosen spatial scale such that the mean and variance remain of a constant order, then the averaged partition function behaves
like the point-to-plane partition function of a $2d$ polymer in the subcritical regime with a log-normal limit. Theorem \ref{T:subcritical} is one special case. We refer to \cite{CSZ25} for details.
\end{remark}

\section{Discussions and open questions}\label{S:open}

\subsection{Disordered systems and singular SPDEs} \label{S:disrel}
Our motivation for studying the $2d$  SHE came from the study of continuum limits of disordered systems, with the DPM being one particular example.

Recalling its definition from Section \ref{S:DPM}, we can regard the DPM as a disorder perturbation of the underlying random walk $S$. In dimensions $d\geq 3$, this perturbation is called {\em irrelevant} in the language of renormalisation group theory because the critical disorder strength (inverse temperature) $\beta_c(d)>0$, and it has been shown that \cite{CY06} for $\beta<\beta_c$, the polymer measure converges to the same limiting Wiener measure
as the underlying random walk $S$. On the other hand, in $d=1$ and $2$, $\beta_c(d)=0$ and hence at any $\beta>0$, the polymer measure experiences path localisation \cite{CSY03} and the path is expected to be super-diffusive. Therefore the disorder perturbation is {\em relevant} in $d=1$ and $2$.
Dimension $d=2$ turns out to be critical, and only finer details of the model determine whether disorder perturbation is relevant (called {\em marginal relevance}) or irrelevant (called {\em marginal irrelevance}). In general, a disorder perturbation of an underlying pure model (without disorder) is called {\em relevant} if any fixed disorder strength $\beta>0$, no matter how small, changes the large scale behaviour of the model (in particular, its scaling
limit), and it is called {\em irrelevant} if small $\beta>0$ does not change the large scale behaviour and the scaling limit. We refer to \cite{Gi11} for more discussions on disordered systems and the {\it Harris criterion} \cite{H74} on when disorder perturbation is predicted by physicists to be relevant/irrelevant.

The connection between disordered systems and singular SPDEs is that, we can regard such SPDEs as a disorder perturbation of the deterministic PDE without the noise term. As we saw in \eqref{eq:SHErg}, in dimensions $d<2$, the effective strength of the noise tends to zero as we zoom into smaller
and smaller space-time scales, which makes the disorder perturbation of the heat equation an irrelevant perturbation, while in dimensions $d>2$,
the disorder perturbation is a relevant perturbation, and $d=2$ is marginal. If we consider instead large scale behaviour by sending $\eps\uparrow\infty$ in \eqref{eq:SHErg}, then the disorder perturbation is relevant in $d<2$, irrelevant in $d>2$, and marginal in $d=2$, which corresponds exactly to
the DPM. Therefore the notion of disorder relevance (resp.\ irrelevance) for disordered systems, which is concerned with large scale behaviour, corresponds to the notion of the singular SPDE being subcritical (resp.\ supercritical), which is concerned with small scale behaviour. Marginality
for disordered systems corresponds to criticality for singular SPDEs.

If a disorder perturbation of a pure model is relevant, then heuristically, the effective strength of disorder will diverge as we zoom out in space-time
(equivalently, take the continuum limit by sending the lattice spacing to $0$). It should then be possible to send the disorder strength to $0$ at a suitable rate as the lattice spacing tends to $0$, such that we obtain a continuum limit that has non-trivial disorder dependence. For disordered systems
that are disorder relevant (not marginal), this was first carried for the one-dimensional directed polymer in \cite{AKQ14a}, which led to the continuum directed
polymer model \cite{AKQ14b}. Subsequently, it was shown in \cite{CSZ17a} that similar results should hold for more general disorder relevant binary-valued spin systems,
including long range DPM \cite{CSZ17a}, the disordered pinning model \cite{CSZ16}, and random field perturbation of the critical $2d$  Ising model \cite{CSZ17a, BS22}. See \cite{CSZ17a} for a more detailed discussion, and \cite{LMS24} for an extension to non-binary valued spin systems.

It is our attempt to investigate the disordered continuum limit of the DPM in the critical dimension $d=2$ that led to the results discussed in this article. Following a similar approach, the scaling limit of disordered pinning model with critical tail exponent~$\alpha=\frac{1}{2}$
was  recently obtained in \cite{WY24}, which is also connected to singular stochastic Volterra equations.

Note that disordered systems in general, such as the random field Ising model, have no time dimension in the noise and hence do not correspond to any singular SPDE. In light of our results for the $2d$  DPM, a natural question is:
\medskip

{\bf Q.} Are there marginally relevant disordered systems without a time dimension, for which we can obtain a non-trivial disordered continuum limit?

\subsection{More open questions}
We close by presenting some open questions and possible directions for future research.

\begin{itemize}

\item[1.] {\bf Universality}. It would be interesting to show that the critical $2d$ SHF arises as the universal scaling limit of models
beyond the directed polymer model and the mollified SHE. Potential candidates include the non-linear SHE considered in \cite{DG22, Tao24, DG23, DG24}, although even identifying the correct critical point appears challenging. It will also be interesting to explore connections to other statistical mechanics models.

\item[2.] {\bf Structure of the critical $2d$ SHF at a given time.}
We have established in \cite{CSZ25} that the critical $2d$ SHF at any deterministic time is almost surely singular with respect to the
Lebesque measure. This raises the question: is the measure supported on a set of fractal Hausdorff dimension? It is plausible that this dimension is marginally below $2$. Another unresolved question is whether the critical $2d$ SHF has infinite speed of propagation, that is, starting with an initial measure with compact support, at any later time, the SHF will almost surely assign positive mass to any open ball.\footnote{After the completion of these lecture notes, this question has been answered in the affirmative recently in \cite{CT25} and \cite{N25b}.}

\item[3.] {\bf Structure of the maxima.} The critical $2d$ SHF is a non-negative measure-valued process, which is a log-correlated field at each time.
Over the past decade there has been significant interests in studying log-correlated fields or their exponential (Gaussian Multiplicative Chaos), and in particular the structure of their maxima. Some examples are \cite{RV14, BDZ16, DSRV17, CGRV19, B20, CFLW21}, though the list is much longer.
A natural question is whether a similarly detailed picture can be achieved for the peaks of the critical $2d$ SHF. The study of maxima is already very interesting in the sub-critical regime, which is more closely connected to Gaussian Multiplicative Chaos (see \cite{CNZ25} for some recent progress).

\item[4.] {\bf Relation to GMC}. We have shown in \cite{CSZ23b} that at any given time, the critical $2d$ SHF cannot be realised as the (Wick) exponential of a generalised Gaussian field, and hence is not a Gaussian Multiplicative Chaos (GMC).
However, it remains open whether the critical $2d$ SHF could be the exponential of a perturbation of a Gaussian field. Alternatively, one might ask whether the law of the critical $2d$ SHF at each time is absolutely continuous with respect to the law of a GMC.

Given the form of the exponential weight in \eqref{eq:FK1}, it is natural to regard it as the exponential of a Gaussian field indexed by paths
in the Wiener space, which could lead to a GMC on path space. This was carried out in~\cite{BM20, BLM24} in dimension $d\geq 3$ in the subcritical regime. In dimension $2$, Clark and Mian \cite{CM24} constructed a continuum polymer measure that is the path space extension of the critical $2d$ SHF.  Although this polymer measure cannot be realised as a GMC with respect to the underlying Wiener measure, an interesting open question raised in \cite{CM24} is whether the continuum polymer measure satisfies the {`\it conditional GMC'} structure similar to what is known on the continuum diamond hierarchical lattice \cite{Cla23}. More precisely, this means that for any $\theta>\theta'$, the continuum polymer measure with
parameter $\theta$ can be realised as a GMC with the reference measure being the continuum polymer measure with parameter $\theta'$.\footnote{This conjecture has recently been proved in \cite{CT25}.}

\item[5.] {\bf On the moments.} It follows from the results in \cite{GQT21} that all moments of 
 $\SHF_{s,t}^{\theta} (\phi,\psi)$, with $\phi, \psi \in L^2(\R^2)$, are finite (the extension to $L^p(\R^2)$ spaces with $p\in(1,\infty)$ was achieved in \cite{CSZ23a}).
The upper bound on the growth of the $n$-th moment is $\exp(e^{c n^2})$ (see \cite[(8.26)]{GQT21}), while existing lower bounds are of the order $\exp(c n^2)$ \cite{CSZ23b}.
However, predictions in the physics literature \cite{R99} suggest that the growth should be $\exp(e^{n})$.\footnote{A lower bound of this order has been obtained recently in \cite{GN25}.} Although the rapid growth of moments does not allow
one to uniquely determine the distribution, it would be valuable to identify the correct order of growth and  understand the underlying mechanism.
Such asymptotics could shed light on the tail statistics of $\SHF_{s,t}^{\theta} (\phi,\psi)$ (for suitable $\phi,\psi$). On the other hand, it will also be interesting to investigate whether $\SHF_{s,t}^{\theta} (\phi,\psi)$ possesses negative moments. Recently, moment asymptotics have also been investigated in the sub- and quasi-critical regime \cite{CN25} or over shrinking balls in the critical window \cite{LiuZ24}. Extending and sharpening these results will also be interesting.

\item[6.] {\bf Interpolating nature and going beyond the critical window.}
The critical $2d$ SHF sits precisely at the boundary between the weak and strong disorder phases ($\hat\beta<1$ vs $\hat\beta>1$ in \eqref{eq:hbeta}).
The parameter $\theta$ in $\SHF_{s,t}^{\theta}$ provides an interpolation between the weak disorder phase
(as $\theta\to -\infty)$ and the strong disorder phase (as $\theta\to+\infty$). An interesting question is whether we can obtain information
on the strong disorder phase by taking $\theta\to+\infty$. A first step is to determine whether the random measure $\SHF_{s,t}^{\theta}$
converges locally to $0$ as $\theta\uparrow \infty$.\footnote{This has been proved recently in \cite{CT25} and \cite{BCT25}, with the latter providing quantitative bounds.} Any progress in understanding potential scaling limits in the strong disorder phase $\hat\beta>1$, or the very strong disorder phase $\beta_N\equiv \beta>0$ in \eqref{eq:hbeta}, would be very interesting.

\item[7.] {\bf Construct the Critical $2d$ Polymer Measure.}
The critical $2d$ SHF arises as the scaling limit of the directed polymer partition functions on the intermediate disorder scale.
From a statistical physics point of view, it will be very interesting to construct a corresponding continuum polymer model and investigate
its path properties and phase transition. The continuum polymer measure constructed in \cite{CM24} is an infinite measure whose law is
translation invariant in space. An open question is to construct a continuum polymer probability measure started at a single point,
similar to its discrete analogue defined in \eqref{eq:pathmeasure}. The main difficulty is that, the point-to-plane partition function,
which serves as a normalising constant in the discrete polymer measure in \eqref{eq:pathmeasure}, converges to $0$ in the critical window
as shown in Theorem \ref{T:subcritical}.

\item[8.] {\bf Black noise.} The critical $2d$ SHF
determines a family of $\sigma$-fields
$\cF_{s,t} := \sigma\big(\SHF_{u,v}^\theta \colon s \le u < v \le t\big)$
on the underlying probability space. The almost sure Chapman-Kolmogorov property established by Clark and Mian \cite{CM24} and stated in
Theorem \ref{thm:Clark} already implies that
this family of $\sigma$-fields is a \emph{noise}
in the sense of Tsirelson,
which is a continuum generalisation of the notion of a family of i.i.d.\ random variables. (The
family $(\cF_{s,t})_{s < t}$ being a noise means that, for $s < t < u$,
the $\sigma$-fields $\cF_{s,t}$ and $\cF_{t,u}$
are independent and $\cF_{s,u}$ is generated by $\cF_{s,t}$ and $\cF_{t,u}$;
moreover, there is a group $(\theta_h)_{h \in \R}$ of measure preserving maps (time shifts)
such that $\theta_h(A) \in \cF_{s+h,t+h}$ for $A \in \cF_{s,t}$.)

It would be very interesting to show that the noise generated by the
critical $2d$ SHF is in fact a (one-dimensional) \emph{black noise}\footnote{The black noise property was recently established in 
\cite{GT25}. Closely related results on noise sensitivity were obtained in \cite{CD25}.}
(we refrain from giving here
a precise definition of black noise and refer to the survey by Tsirelson \cite{Tsi04a}).
Examples of one-dimensional black noise include the Brownian web \cite{Tsi04b} and the directed landscape \cite{HP24},
and examples of two-dimensional black noise, which is much more challenging to prove, include the continuum limit
of critical planar percolation \cite{SS11} and the Brownian web \cite{EF16}.

\item[9.] {\bf Renormalisation of the 1-point statistics.} Theorem \ref{T:subcritical} establishes
that, in the critical window of the intermediate disorder scale, the point-to-plane partition function
$Z^{\beta_N}_N(0)$ of the directed polymer converges in distribution to $0$. This aligns with the fact that the critical $2d$
SHF is almost surely singular with respect to the Lebesque measure.
A natural question is to determine the rate at which $\log Z^{\beta_N}_N(0)$
(i.e., the one-point statistic of the discretized $2d$ KPZ equation in the critical window) converges to $-\infty$.
For critical GMC, which is connected to the directed polymer on trees, such a rate of convergence to $0$ has been identified
\cite{RV14}.

Based an extrapolation from the subcritical regime, we conjecture that in the critical window,
$$
\E[\log Z^{\beta_N}_N(0)] \sim -\frac{1}{2}\log \log N,
$$
with a variance of the same order $\log \log N$.  It will be very interesting
to identify the limiting law of $\log Z^{\beta_N}_N(0)$ after proper centering and scaling, which would provide a partial
interpolation between the Gaussian distribution in the weak disorder phase $\hat\beta<1$ and the $2d$ analogue of the
Tracy-Widom distribution in the very strong disorder phase $\beta_n\equiv \beta>0$ in \eqref{eq:hbeta}.

\item[10.] {\bf Construct the Critical $2d$ KPZ.}
As shown in Theorem \ref{T:KPZ}, in the subcritical regime $\hat\beta<1$, the $2d$ KPZ has Gaussian fluctuations and solves
the $2d$ Edwards-Wilkinson equation. It remains open to define the solution of the $2d$ KPZ in the critical window since we cannot
perform the Cole-Hopf transformation on the critical $2d$ SHF, which is a random measure and not a function. The challenge is to
study the field of log partition functions $\log Z^{\beta_N}_N(z)$, not only its mean and variance at each $z\in\Z^2$,
but also its covariance at different $z\in \Z^2$. It would also be interesting to compare with physics works \cite{FT94} where
a dynamic renormalisation approach suggests that a non-Gaussian limit (referred to as `fixed point' in the physics language)
 exists if the parameter $\theta=\theta_N$ in the critical window \eqref{eq:sigma} is chosen to diverge at the rate of $\log\log N$.\footnote{We thank A. Kupiainen for bringing \cite{FT94} to our attention. NZ also thanks R. Bauerschmidt and P. Le Doussal for illuminating discussions on
 the implied parameter scaling $\theta_N$.}

\item[11.] {\bf Step into critical singular SPDEs}.
The theory of singular SPDEs has undergone revolutionary developments thanks to the frameworks
of regularity structures \cite{H14}, paracontrolled calculus \cite{GIP15}, renormalisation group approach \cite{Kup16, D22}, energy solutions \cite{GJ14, GP18}, etc. However, these theories are restricted to subcritical singular SPDEs. For SHE and KPZ, this means dimension $d<d_c=2$.
The critical $2d$ SHF provides a rare example of a model in the critical dimension and at the critical point, which has a non-Gaussian scaling
limit. A natural question is whether the critical $2d$ SHF can help shed some light on other critical singular SPDEs, such as dynamics for the $\phi^4$ model \cite{MW17, AK20, BG20, GH21} and the Yang-Mills model \cite{CCHS22, CCHS24, CS23} at the critical dimension $4$. Recently, there have much progress in the analysis of critical and supercritical singular SPDEs that admit an explicit Gaussian stationary measure, see the lecture notes
\cite{CT24} and the references therein. However, these models do not exhibit a phase transition as we see in the $2d$ SHE. It will be interesting
to find other critical or supercritical singular SPDEs that exhibit a phase transition.

 \end{itemize}

\section*{Acknowledgements}
The material of these notes was motivated by and presented in mini-courses in the Summer School ``{\it Statistical Mechanics and Stochastic PDEs}'' at Cetraro, Italy,
under Fondazione CIME, in the ``{\it Spring School on Critical Singular Stochastic PDEs}''  at
the Beijing International Center for Mathematical Research, and in the ``{\it Intensive Lecture Series}'' at Seoul National University.
We thank the organisers of these meetings -- CIME foundation,
Weijun Xu, Insuk Seo and Sung-Soo Byun -- for providing the opportunity to present these lectures.
We are also grateful to \href{https://www.flickr.com/photos/antistath/}{Stathis Floros}
for his kind permission to use his \href{https://en.wikipedia.org/wiki/Meteora\#/media/File:Meteora's_monastery_2.jpg}{picture of Meteora}.

F.C.~is supported by INdAM/GNAMPA.
R.S.~is supported by NUS grant A-8001448-00-00 and NSFC grant 12271475.



\newcommand{\etalchar}[1]{$^{#1}$}

\end{document}